\documentclass[leqno]{amsart}
\usepackage{times}
\usepackage{amsfonts,amssymb,amsmath,amsgen,amsthm,faktor}
\usepackage{hyperref}
\usepackage[T1]{fontenc}
\usepackage[utf8]{inputenc}
\usepackage[american]{babel}
\usepackage{cancel}
\usepackage[pdftex]{color,graphicx}
\usepackage{multicol}
\newcommand{\msc}[2][2000]{%
  \let\@oldtitle\@title%
  \gdef\@title{\@oldtitle\footnotetext{#1 \emph{Mathematics subject
        classification.} #2}}%
}

\theoremstyle{plain}
\newtheorem{theorem}{Theorem}[section]
\newtheorem{definition}[theorem]{Definition}

\newtheorem{lemma}[theorem]{Lemma}
\newtheorem{corollary}[theorem]{Corollary}
\newtheorem{proposition}[theorem]{Proposition}

\theoremstyle{remark}
\newtheorem{remark}[theorem]{Remark}

\newtheorem{example}[theorem]{Example}

\def\R{{\mathbb R}}
\def\N{{\mathbb N}}

\def\({\left(}
\def\){\right)}
\def\<{\left\langle}
\def\>{\right\rangle}

\def\Tend#1#2{\mathop{\longrightarrow}\limits_{#1\rightarrow#2}}

\def\eps{\varepsilon}

\def\op{{\rm op}}

\numberwithin{equation}{section}

\begin{document}
\title{Semiclassical analysis of the Schrödinger equation with conical singularities}
\author{Victor Chabu}
\address{LAMA, UMR CNRS 8050, Université Paris-Est \\ 61 avenue du Général de Gaulle \\ Créteil CEDEX 94010 \\ FRANCE}
\email{victor.chabu@math.cnrs.fr}

\begin{abstract}
In this article we study the propagation of Wigner measures linked to solutions of the Schrödinger equation with potentials presenting conical singularities and show that they are transported by two different Hamiltonian flows, one over the bundle cotangent to the singular set and the other elsewhere in the phase space, up to a transference phenomenon between these two regimes that may arise whenever trajectories in the outsider flow lead in or out the bundle. We describe in detail either the flow and the mass concentration around and on the singular set and illustrate with examples some issues raised by the lack of unicity for the classical trajectories at the singularities despite the unicity of the quantum solutions, dismissing any classical selection principle, but in some cases being able to fully solve the propagation problem. 
\end{abstract}

\maketitle

\section{Introduction} 
\subsection{Initial considerations}\label{sec:initial_considerations}
Classically, a particle with mass $m = 1$ submitted to a time- and momentum-independent smooth potential $V$ in $\R^d$ is constrained to move following the phase space trajectory given by the Hamilton equations
\begin{equation}\left\lbrace
\begin{array}{l}
\dot{x}\left(t\right) = \xi\left(t\right) \\
\dot{\xi}\left(t\right) = -\nabla V \left( x\left(t\right) \right).
\end{array}\right.
\label{eq:hamilton}
\end{equation}
Smoothness in $V$ guarantees that the equations above have a unique solution $(x(t),\xi(t))$ around all initial condition $(x_0,\xi_0)$ (\emph{i.e.}, for $t$ sufficiently small), thus we can define the classical Hamiltonian flow $\Phi$ by setting $\Phi_{t}\left(x_0,\xi_0\right) = (x(t),\xi(t))$. Further conditions on the regularity and growth rate of $V$ imply more good properties, so if $\nabla^2 V$ is bounded, one can extend $\Phi_t(x_0,\xi_0)$ for all $t \in \R$, for any $(x_0,\xi_0)$ in the phase space\cite{ReedSimon}. 

In Quantum Mechanics, the state evolution of a similar system is described by a function $\Psi^\eps \in L^\infty\left(\R,L^2(\R^d)\right)$ obeying to the Schrödinger equation with initial data
\begin{equation}
\left\lbrace
\begin{array}{l}
i\eps \partial_{t} \Psi_{t}^{\eps}\left(x\right) = -\frac{\eps^2}{2}\Delta \Psi^\eps_t (x) + V(x)\Psi^\eps_t (x) \\
\Psi_{t = 0}^\eps\left(x\right) = \Psi_{0}^\eps\left(x\right),
\end{array}
\right.
\label{eq:schrodinger}
\end{equation}
where the initial $L^2(\R^d)$ data satisfy $\| \Psi_0^\eps \|_{L^2(\R^d)} = 1$ and $\eps \ll 1$ is a parameter generally reminiscent from some rescaling procedure, but that can also be seen as the Planck's constant in a system with mass of order $m \approx 1$, in which case the Wigner measures can be viewed as the classical limits of the system's mass distribution, as we will explain next.
 
If $V$ satisfies the Kato-Rellich conditions ($V$ continuous and $V(x) \lesssim \| x \|^2$), then the Hamiltonian operator 
\begin{equation}\label{eq:original_hamiltonian}
\hat{H}^\eps = -\frac{\eps^2}{2}\Delta + V
\end{equation}
with domain in the Sobolev space $H^2(\R^d)$ is essentially self-adjoint (\cite{Kato},\cite{ReedSimon}) and \eqref{eq:schrodinger} has a unique solution for all $t \in \R$, which is given by
$$\Psi_t^\eps = e^{-\frac{i}{\eps}t\hat{H}^\eps}\Psi_0^\eps .$$

\medskip

Now, for $t \in \R$, we define the \emph{Wigner transform} associated to the solution of \eqref{eq:schrodinger},
\begin{equation}
W^{\eps}\Psi_{t}^{\eps}\left(x,\xi\right) = \frac{1}{\left(2\pi\eps\right)^{d}} \int_{\R^d} e^{\frac{i}{\eps} y \centerdot \xi} \, \Psi_{t}^{\eps}\left(x - \frac{y}{2} \right) \overline{\Psi_{t}^{\eps}\left(x + \frac{y}{2}\right)} \, dy,
\label{eq:wigner}
\end{equation} 
which can be interpreted as the quantum version of the mass probability density on the phase space in classical Statistical Physics and provides a link between these two theories\cite{Wigner}.

When $\eps \longrightarrow 0$, the Wigner transform converges to a finite and positive measure $\mu$ on $\R \times \R^{2d}$ (\cite{gerard}, \cite{GMMP}, \cite{Wigner}) in the sense that, given a sequence $(\eps_n)_{n \in \N}$ converging to $0$, we can extract a subsequence\footnote{Of course $\mu$ may depend on the subsequence, this is why we refer to it as \emph{a} semiclassical limit, not necessarily \emph{the} classical one (examples of non-unicity in \cite{Wigner}).} $(\eps_{n_k})_{k \in \N}$ such that, for all $a \in C_0^\infty(\R^{2d}_{x,\xi})$ and $\Xi \in C^\infty_0(\R_t)$, 
$$\int_\R \Xi(t) \left< \, W^{\eps_{n_k}}\Psi_t^{\eps_{n_k}} \, , \, a \, \right>_{\R^{2d}}dt \: \underset{k \rightarrow \infty}{\longrightarrow} \: \int_{\R \times \R^{2d}} \Xi(t) a(x,\xi) \, \mu(dt,dx,d\xi).$$
 
Again, more regularity on $V$ implies more good properties\cite{GMMP}. The Wigner measure is always absolutely continuous with respect to the Lebesgue measure $dt$, as one sees from its very construction, so one has $d\mu(t,x,\xi) = \mu_t(x,\xi)dt$, where $t \longmapsto \mu_t$ is a $L^\infty\left(\R,\mathcal{D}'(\R^{2d})\right)$ function. However, if for instance $\nabla^2 V$ is bounded, one can show by the Ascoli-Arzelà theorem that in this case, given $T > 0$, $[-T,T] \ni t \longmapsto \mu_t \in \mathcal{D}'(\R^{2d})$ is continuous and that there exists another subsequence $\left(\eps_{k_{k'}}\right)_{k' \in \mathbb{N}}$ such that, for each $t \in \left[-T,T\right]$ and $a \in C_0^\infty(\R^{2d})$,
$$\left< \, W^{\eps_{k_{k'}}}\Psi^{\eps_{k_{k'}}}_t \, , \, a \, \right>_{\R^{2d}} \: \underset{k' \rightarrow \infty}{\longrightarrow} \: \int_{\R^{2d}} a \, d\mu_t.$$
Furthermore, for such a regular $V$, the semiclassical measures satisfy a Liouville equation
\begin{equation}
\partial_{t} \mu \left(t,x,\xi\right) + \xi \centerdot \partial_{x} \mu \left(t,x,\xi\right) - \nabla V \left(x\right) \centerdot \partial_{\xi} \mu \left(t,x,\xi\right) = 0 \quad {\rm in} \; \mathcal{D}'(\R \times \R^{2d}),
\label{eq:liouville}
\end{equation}
or equivalently
\begin{equation}
\left\lbrace
\begin{array}{l}
\partial_{t} \mu_t \left(x,\xi\right) + \xi \centerdot \partial_{x} \mu_t \left(x,\xi\right) - \nabla V \left(x\right) \centerdot \partial_{\xi} \mu_t \left(x,\xi\right) = 0 \\
\mu_{t = 0}(x,\xi) = \mu_0(x,\xi)
\end{array}\right. \quad {\rm in} \; \mathcal{D}'(\R^{2d}),
\label{eq:liouville_t_by_t}
\end{equation}
where $\mu_0$ is the correspondent semiclassical limit of $W^\eps\Psi^\eps_0$.

This last equation is interpreted as a transport phenomenon along the classical flow $\Phi$, which can be easily seen by picking up test functions $a \in C_0^\infty(\R^{2d})$ and verifying that $\frac{d}{dt}\int a\circ \Phi_{-t}\, d \mu_{t} =0$ due to \eqref{eq:hamilton} and \eqref{eq:liouville_t_by_t}.

Moreover, the test functions $a \in C_0^\infty(\R^{2d})$ are called \emph{classical observables} and are related with the \emph{quantum observables} $\op_\eps (a) \in \mathcal{L}\left(L^2(\R^d)\right)$ by the Weyl quantization formula,
{\small
\begin{equation}
\op_{\eps}\left(a\right) \Psi \left(x\right) = \frac{1}{\left(2\pi\eps\right)^{d}} \int_{\R^{2d}} e^{\frac{i}{\eps}\xi\centerdot\left(x-y\right)} a\left(\frac{x+y}{2},\xi\right) \Psi\left(y\right) d\xi dy \quad {\rm for} \; \Psi \in L^2(\R^d),
\label{eq:pseudor}
\end{equation}} 
which provides self-adjoint operators for real-valued symbols\cite{dimassi}. These objects, called \emph{semiclassical pseudodifferential operators}, relate to the Wigner transform by means of the formula  
\begin{equation}\label{eq:wigner_pseudor}
\left< \, \op_{\eps}\left(a\right) \Psi \, ,  \,\Psi \, \right>_{L^2(\R^d)} = \left< \, W^{\eps}\Psi \, , \, a \, \right>_{\R^{2d}}.
\end{equation}

The quadratic form $\left< \, \op_{\eps}\left(a\right) \Psi \, ,  \,\Psi \, \right>_{L^2(\R^d)}$ gives in Quantum Mechanics the average value of the observable $\op_\eps(a)$ for a system in the quantum state $\Psi$, exactly as does the integral $\int_{\R^{2d}} a \, d\rho$ in classical Statistical Mechanics for the observable $a$ in a system with mass probability density $\rho$ over the phase space. Besides, last equation carries
$$\left< \, \op_{\eps}\left(a\right) \Psi_t^\eps \, ,  \,\Psi_t^\eps \, \right>_{L^2(\R^d)} \; \underset{\eps \rightarrow 0}{\longrightarrow} \; \int_{\R^{2d}} a \, d\mu_t,$$
which allows us to understand the Wigner measures as mass distributions, linking the quantum evolution of $\Psi^\eps$ given by the Schrödinger equation with that of its semiclassical limits\footnote{From a non-statistical point of view, the classical limit properly speaking would be a particular subsequence $(\eps_{n_k})$ that gives a Dirac mass on a certain point of the phase space, corresponding to a singular particle of mass $1$ (classically localized on that point) whose quantum evolution is described by \eqref{eq:schrodinger}. Although it is always possible to find a sequence of quantum states concentrating to a Dirac mass\cite{gerard}, the general picture is much broader, being possible to have, for instance, $\mu$ continuously spread over the phase space. In any case, the quantum-classical correspondence is better understood statistically (\cite{Wigner}, \cite{tartar}), in which frame the case of the Dirac mass (or a sum of punctual Dirac deltas with total mass $1$) should be seen as a special case of statistical distribution.}.

\subsection{Statement of the problem}\label{sec:statement}
It happens that a very large class of relevant problems do not present potentials with all such regularity. For instance: \emph{conical potentials}, which are of the form
\begin{equation}
V(x) = V_S(x) + \left\| g(x) \right\| F(x),
\label{eq:conical_potential}
\end{equation}
where we make the following technical assumptions:
\begin{itemize}
\item $V$ and $V_S$ satisfy each one the Rellich conditions.
\item $F$ and $V_S$ are $C^\infty(\R^d)$ and there is some non-decreasing positive $K$-sub-additive\footnote{This is: there is $K \geqslant 1$ such that $\mathfrak{p}(x+y) \leqslant K \left( \mathfrak{p}(x) + \mathfrak{p}(y) \right)$.} polynomial $\mathfrak{p}$ that bounds them and also $\nabla F$. 
\item $g : \R^d \longrightarrow \R^p$ with $1 \leqslant p \leqslant d$, $\nabla g$ is full rank and $\Lambda = \lbrace g(x) = 0 \rbrace \neq \emptyset$. 
\end{itemize}
As we shall see, these potentials raise interesting mathematical questions.

Similar problems have been treated in works like \cite{AFFG}, \cite{AP} and \cite{athanassoulis} in a probabilistic way. In other works authors have been analysing the deterministic behaviour of the Wigner measures under the conical potentials defined above, more noticeably in \cite{FGL}, where they found a non-homogeneous version of \eqref{eq:liouville} whose inhomogeneity is an unknown measure supported on
$$\Omega = \left\lbrace (x,\xi) \in \R^{2d} \, : \, g(x) = 0 \; {\rm and} \; \nabla g(x) \, \xi = 0 \right\rbrace,$$
a set onto which it is not generally possible to extend the classical flow in a unique manner, although it is possible everywhere else.

\begin{remark}\label{rem:cotangent}
The set $\Omega$ corresponds exactly to the tangent bundle to $\Lambda$, since any curve $\gamma$ over $\Lambda$ (\emph{i.e.}, such that $g(\gamma(t)) = 0$) passing on $x$ at $t = 0$ must satisfy $\nabla g(x) \dot{\gamma}(0) = 0$. In our work, however, we stay within a structure of phase space, so it will be natural to identify $\Omega$ as the cotangent bundle $T^*\Lambda$.
\end{remark}

This suggests an intriguing possibility involving irregular potentials: what happens to the Wigner measures in a system where the potential allows a complete quantum treatment, but causes the classical flow to be ill-defined? Is there some selection principle from the quantum-classical correspondence that could provide information enough for describing the transport of the measure where the classical flow fails? 

To fix some ideas, forget for a moment about the measures and think of a classical particle submitted to conical potentials like $V\left(x\right) = \pm \left| x \right|$. Naturally, the trajectories are well-defined everywhere away from the line $x = 0$, and it turns out that they can be continuously extended onto $x=0$ in a unique manner provided that $\xi \neq 0$, as in Figure \ref{fig:flow_classical}.

\begin{figure}[htpb!]\center\footnotesize
\begin{multicols}{2}
\includegraphics[width=0.40\textwidth]{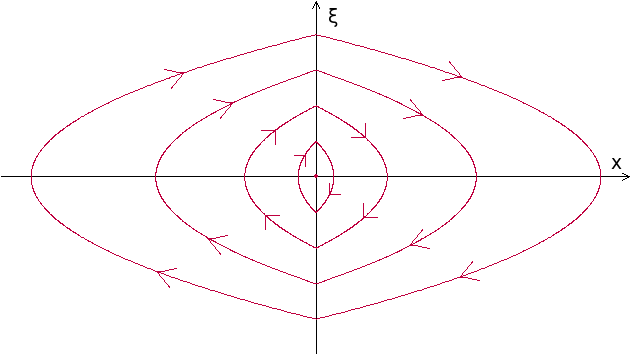} \text{(a) $V\left(x\right) = \left| x \right|$. At the origin the trajectory is constant.} \\
\includegraphics[width=0.26\textwidth]{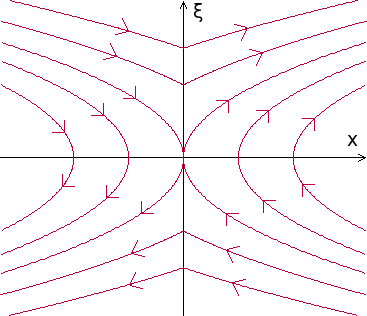} \text{(b) $V\left(x\right) = - \left| x \right|$. At the origin the flow is not well-defined.} \\
\end{multicols}
\caption{\footnotesize A glance on the classical flows for the potentials $V\left(x\right) = \pm \left| x \right|$ near the origin. The arrows indicate their orientation.}
\label{fig:flow_classical}
\end{figure}

In the case $V(x) = +|x|$, the flow can be uniquely defined even over $\xi = 0$ by setting it constant at the origin, as shown in Figure \ref{fig:flow_classical}(a); more generally, in higher dimensional cases, like for $x \in \R^d$, $x = (x',x'')$ with $x' \in \R^p$ and $V(x) = \| x' \|$, there is still room for the particle to move inside the singular set $\Omega = \left\lbrace x' = 0 \; {\rm and} \; \xi' = 0 \right\rbrace$ and it must have some non-ambiguous behaviour therein, induced by the unique quantum evolution of $\Psi_t^\eps$ and their well-defined concentration to $\mu_t$. However, the dynamics in $\Omega$ is classically unknown, for $\nabla V$ makes no sense for $x' = 0$, so we cannot rely on the sole Hamiltonian trajectories to characterize the transport phenomenon that the semiclassical measure undergoes over the singularities. 

Furthermore, there are other kinds of difficulties. In the case $V\left(x\right) = - \left| x \right|$, even in dimension $1$ there is no unique extension for the flow all over the phase space. As we can see in Figure \ref{fig:flow_classical}(b), when coming from the right-hand side below, there are different alternative trajectories after reaching $x = 0$ with zero momentum: going back to the right upwards, crossing to the left downwards, staying at $(0\, ,0)$, or staying there for a moment and then resume moving to one or to the other side.

Let us treat this problem in three different steps. 

\subsection{First question: the dynamics}
In \cite{FGL}, the authors proved that the Hamiltonian flow can always be continuously extended in a unique manner to $\Lambda \setminus \Omega$ and that, whenever the Wigner measure does not charge the singularities in the phase space, \emph{i.e.}, while $\mu_t(\Omega) = 0$, then $\mu$ follows these unique continuous extensions. This result is grounded on the facts that $\mu_t$ does not charge the set $\Lambda \setminus \Omega$ for more than a negligible time, more precisely that $\mu_t\left(\Lambda\setminus \Omega\right) = 0$ almost everywhere in $\R$ with respect to $dt$ (for a matter of completeness, we re-obtain this result in Lemma \ref{lem:m_is_zero_and_etc}), and that the measures obey to the standard Liouville equation with $V$ away from $\Lambda$, where the potential is regular.

In this paper we will obtain in Section \ref{sec:establishing_liouville} a complete description of the dynamics to which the semiclassical measures ought to obey, including near and inside the singularities, by driving an approach similar to that of \cite{FGL}, which makes an extensive use of \emph{symbolic calculus} (Section \ref{sec:symbolic_calculus}) and \emph{two-microlocal measures} (Section \ref{sec:two_microlocal_analysis}):

\begin{theorem} \label{th:main_result_1}
Let $\Psi^\eps$ be the solution to the system \eqref{eq:schrodinger} with a conical potential of the form \eqref{eq:conical_potential}, and denote $\Lambda = \left\lbrace x \in \R^d : g(x) = 0 \right\rbrace$. Then the correspondent Wigner measures $\mu$ obey to the $\mathcal{D}'(\R \times \R^{2d})$ equation:
\begin{eqnarray}\label{eq:complete_dynamics}
\left(\partial_t + \pi_{_{T^*_x \Lambda}} \xi\centerdot \partial_x - \pi_{_{T^*_x \Lambda}} \left. \nabla V_S \right|_{\Lambda} (x) \centerdot \partial_{\xi} \right) 1\!\!1_{_{T^*\Lambda}} \, \mu \left(t,x,\xi\right) && \nonumber \\ + 
\left( \partial_t + \xi \centerdot \partial_x - \nabla V(x) \centerdot \partial_\xi \right) 1\!\!1_{_{(T^*\Lambda)^c}} \, \mu \left(t,x,\xi\right) & = & 0, 
\end{eqnarray}
where $1\!\!1_{_{T^*\Lambda}}$ is the indicatrix of $T^*\Lambda = \Omega$ (and $1\!\!1_{_{(T^*\Lambda)^c}}$ of its complement inside $\R^{2d}$) and for each $x \in \Lambda$, $\pi_{_{T^*_x\Lambda}}$ is the orthogonal projector over $T^*_x \Lambda = \ker \nabla g(x)$ inside $\R^d$.

Furthermore, decomposing $\R^{2d}$ in a neighbourhood of $\Lambda$ as the bundle $E\Lambda$ with fibres $E_\sigma\Lambda = T^*_\sigma \Lambda \oplus N^*_\sigma \Lambda \oplus N_\sigma \Lambda$ (and elements $(\sigma,\zeta,\eta,\rho)$), there exists a measure $\nu$ over $\R \times ES\Lambda$, where $ES_\sigma\Lambda = T^*_\sigma \Lambda \oplus \faktor{N_\sigma \Lambda}{\R_*^+}$, satisfying the asymmetry condition
\begin{equation}\label{eq:asymmetry_formula}
\int_{\faktor{N_\sigma \Lambda}{\R^+_*}} \left( \nabla_\rho V_S (\sigma) + F (\sigma) ^t\nabla g(\sigma) \, \omega \right) \nu \left(t,\sigma,\zeta,d\omega\right) = 0 \quad {\rm in } \; \mathcal{D}' \left(\R\times T^* \Lambda\right)
\end{equation}
and such that
$$1\!\!1_{_{T^*\Lambda}} \, \mu \left(\sigma,\zeta,\eta,\rho\right) \; = \; \delta(\rho) \otimes \delta(\eta) \otimes \int_{\faktor{N_\sigma \Lambda}{\R_*^+}} \nu \left(t,\sigma,\zeta,d\omega\right).$$
\end{theorem} 

\begin{remark}\label{rem:remark_in_proof}
Observe that equation \eqref{eq:complete_dynamics} is the sum of two Liouville terms, one for the potential $V$ outside $\Omega$ and another for $\left. V_S\right|_\Lambda$ over it. Usual transport under $V$ is assured away from $\Omega$ by \cite{FGL}, thus, if no trajectories outside $\Omega$ lead to or from it, these both terms are shown to cancel on their own\footnote{To see this, just test $\mu$ against functions in $C_0^\infty\left(\R^{2d}\setminus\Omega\right)$ and consider the local conservation of mass.} and the equation decouples into two independent transport phenomena, one inside and the other outside $\Omega$, the regularity of the insider flow being guaranteed by the Picard-Lindelöf theorem, as $\left. V_S \right|_{\Omega}$ is smooth in the topological space $\Omega$.
\end{remark}

\begin{remark}\label{rem:static_equilibrium}
The second part of the theorem, the asymmetry formula, will be discussed ahead in Section \ref{sec:second_question} and turns out to indicate that any mass that stays on the singularity will be in static equilibrium over it.
\end{remark}

Remark \ref{rem:remark_in_proof} immediately gives:
\begin{theorem}\label{th:main_result_2}
In the same conditions of Theorem \ref{th:main_result_1}, suppose that no Hamiltonian trajectories lead into $\Omega$. Call $\Phi$ the Hamiltonian flow defined by the trajectories induced by $V$ for $(x,\xi) \notin \Omega$ and by $\left. V_S \right|_\Lambda$ for $(x,\xi) \in \Omega$. Then, writing $\mu(t,x,\xi) = \mu_t(x,\xi)dt$, one has $\mu_t = \Phi_t^* \mu_0$ for all $t \in \R$, where $\mu_0$ is the Wigner measure of the family $(\Psi^\eps_0)_{\eps > 0}$.
\end{theorem}

More precisely, suppose that the Hamiltonian flow $\Phi$ can be extended in a unique way everywhere in a region $\Gamma \subset \R^{2d}$. It is sufficient in Theorem \ref{th:main_result_2} that we choose a time interval $I \subset \R$ such that $\Phi_I\left({\rm supp} (\mu_0) \right) \subset \Gamma$ to assure the transport for any $t \in I$; this recovers the result in \cite{FGL} for a particular choice where $\Gamma \cap \Omega = \emptyset$. In words, it is sufficient that our measure transported by the flow does not hit any point where the trajectories split, the fact that it may charge the singularities being irrelevant.  

\subsection{Second question: the regime change}\label{sec:second_question}
Now, what happens if some trajectories hit $\Omega$? First, realize that in this case there is never uniqueness, since there are necessarily the outgoing trajectory (which is the reverse of the incoming one) and the one whose projection on $T^*\Lambda$ evolves freely and, more importantly, whose projection on $N^*\Lambda$ remains static, what is always kinetically admissible (with $\nabla g(x) \xi = 0$, the velocity normal to $\Lambda$ is $0$, for $\ker \nabla g(x)$ is actually $T^*_x\Lambda$, as informs Theorem \ref{th:main_result_1}). The equation in Theorem \ref{th:main_result_1} says that either a transfer of mass between these two regimes, the insider and the outsider, and the continuation of the exterior transport are possible to happen, but gives no more information. 

The second part of the theorem, however, may solve the question. It has got a rich geometric interpretation, saying that the mass distribution on $\Omega$ has some asymmetry around $\Lambda$ due to the ``shape'' $F ^t\nabla g$ of the conical singularity, and that it is deformed by an exterior normal force $-\partial_\rho V_S$, in such a manner that the portion of mass that remains over the singularity will be in static equilibrium.

Indeed, the measure $\nu$ gives the mass distribution in a sphere bundle with fibres $S_\sigma \Lambda = \faktor{N_\sigma \Lambda}{\R^*_+}$ around the singularity, \emph{i.e.}, the directions $\omega$ in the exterior space from where the solutions $\Psi^\eps$ concentrate more or less intensively to the singular points.

If $\nu = 0$, then of course the measure does not stay at the singularity, so it necessarily continues under the exterior regime (regardless of whether it is unambiguously defined or not) and we got all the information $\nu$ may give; let us then suppose $\nu\neq 0$.

Loosely, let us also consider $\nu$ as a function of $(\sigma,\zeta)$ (as if it was absolutely continuous with respect to $d\sigma d\zeta$) and let us say that $\int_{S_\sigma\Lambda} \nu(t,\sigma,\zeta,d\omega)$ gives the total mass $\mathcal{M}$ on the point $(\sigma,\zeta)$ (though it actually gives a mass density over the phase space) at the instant $t$, supposed not to be $0$. Naturally, $\int_{S_\sigma\Lambda} \omega \, \nu(t,\sigma,\zeta,d\omega)$ gives the average vector of concentration to this point, whose normalization by $\mathcal{M}$ we will call $\vec{D}(\sigma)$.    

\begin{remark}
Realize that the speed the mass may have tangentially to the singular space $\Lambda$, that we call $\zeta$, plays no role in dictating how the quantum concentration will happen thereon; with simple hypotheses on the family $(\Psi^\eps_0)_{\eps > 0}$ (like $\eps$-oscillation, see \cite{gerard}), one has $\int\mu(x,d\xi) < \infty$,  and the same for $\nu$ since $\nu \ll \mu$, so we could be working directly with $\int_{ES_\sigma \Lambda}\nu(t,\sigma,d\zeta,d\omega)$.
\end{remark}

Well, the derivative of a potential is a force, so let us call $\vec{F}^\perp (\sigma) = - \partial_{\rho} V_S (\sigma)$ the force normal to the singular manifold $\Lambda$ at the point $\sigma$. Consequently, the asymmetry formula in Theorem \ref{th:main_result_1} is imposing a simple condition on the mass concentration:
\begin{equation}\label{eq:simple_asymmetry}
F(\sigma) ^t\nabla g(\sigma) \vec{D} (\sigma) = \vec{F}^\perp(\sigma),
\end{equation}
which is to say that, in average, the mass should concentrate alongside the exterior normal force $\vec{F}^\perp$. How strong the concentration will be there with respect to the other points, or whether it is going to be attractive or repulsive, will depend on the shape of the conical singularity at the different points of $\Lambda$, described (so as to say) by $F ^t\nabla g$. 

An example is illustrated in Figure \ref{fig:asymetry} below.

\begin{figure}[htpb!]\center\footnotesize
\includegraphics[width=0.9\textwidth]{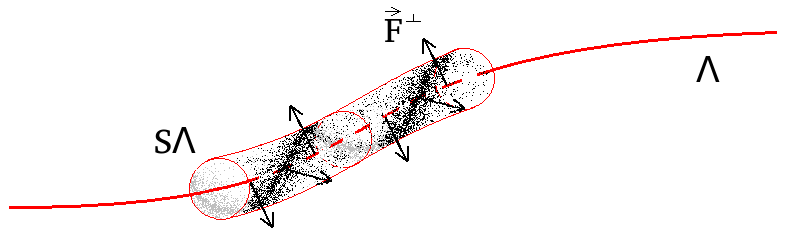}
\caption{\footnotesize Above, we depict $\Lambda$ encircled by its normal bundle in sphere $S\Lambda$ (with fibres $S_\sigma \Lambda = \faktor{N_\sigma \Lambda}{\R^+_*}$). The mass concentration given by $\nu$ is represented by the more or less strongly shadowed regions, and is directed by the normal force $\vec{F}^\perp$ which is spinning around $\Lambda$ in the example. Here, we took $F(\sigma) = 1$ and $\nabla g(\sigma) = 1\!\!1$.}
\label{fig:asymetry}
\end{figure}

The reason why we have said that equation \eqref{eq:asymmetry_formula} is a condition of equilibrium is that the expression inside the integral is similar to what would be the total force normal to $\Lambda$, \emph{i.e.} $-\partial_\rho V$, when calculated on $\Lambda$ (where $g = 0$) and making sense of $\omega$ as some limit of $\frac{1}{\| \nabla g(\sigma)^{-1} g(x) \|} \nabla g(\sigma)^{-1} g(x)$ (a vector in $N_\sigma\Lambda$) when $x$ approaches the singularity from a particular direction. So we are also tempted to interpret \eqref{eq:asymmetry_formula} as saying that the total force normal to $\Lambda$ on some mass staying on the singularity, to be given by the integral, is $0$: this mass is in static equilibrium. 
 
Besides, in formula \eqref{eq:simple_asymmetry}, we have $\left\| \vec{D}(\sigma) \right\| \leqslant 1$, since $\vec{D}(\sigma)$ is an average of norm $1$ vectors in $N_\sigma\Lambda$. Consequently, we must have $\left\|  F(\sigma) ^t\nabla g(\sigma) \right\| \geqslant \left\| \vec{F}^\perp (\sigma) \right\|$ not to be led to an absurd. If this is not the case, then we must have $\nu = 0$ in order to satisfy the asymmetry condition trivially, which means that the Wigner measure will not stay on the singularity. This reasoning will be made rigorous in Section \ref{sec:asymmetry_condition_in_use}, where we will prove:

\begin{theorem}\label{th:main_result_4}
If for some $\sigma \in \Lambda$ one has $\left\| F(\sigma) ^t\nabla g(\sigma) \right\|_{\mathcal{L}(N_\sigma \Lambda)} < \left\| \partial_\rho V_S(\sigma) \right\|_{N_\sigma \Lambda}$, then there exists a neighbourhood $\Gamma \subset \Lambda$ of $\sigma$ such that $\nu = 0$ over $\R \times ES \Gamma$, where $ES_\sigma \Gamma = T^*_\sigma \Gamma \oplus \faktor{N_\sigma \Gamma}{\R_*^+}$.
\end{theorem}

Once established that in some cases the mass is forbidden to stay over the singularity, it is worth studying more deeply the ways it can get in and out $\Omega$:
\begin{theorem}\label{th:accessory_result_1}
Supposing $\left\| F(\sigma) ^t\nabla g(\sigma) \right\|_{\mathcal{L}(N_\sigma \Lambda)} < \left\| \partial_\rho V_S(\sigma) \right\|_{N_\sigma \Lambda}$, for any trajectory $(x(t),\xi(t))$ leading out from or into $\Omega$ in $\sigma \in \Lambda$ at $t_0 \in \R$, set the $N_\sigma \Lambda$ vector $\rho (t) = \frac{2}{(t-t_0)^2}\nabla g(\sigma)^{-1}g(x(t))$; then, if $\lim_{t \rightarrow t_0^\pm} \frac{\rho(t)}{\|\rho(t)\|}$ is well-defined, $\rho (t)$ also has well-defined lateral limits $\rho_0^\pm$ when $t \longrightarrow t_0^\pm$, which are non-zero roots of
\begin{equation}\label{eq:aiai}
\rho_0 = -\partial_\rho V_S(\sigma) - F(\sigma) ^t\nabla g(\sigma)\frac{\nabla g(\sigma) \rho_0}{\| \nabla g(\sigma) \rho_0 \|}.
\end{equation}
Conversely, for any $\rho_0^+$ and $\rho_0^-$ satisfying \eqref{eq:aiai}, there exists a unique continuous extension of the classical flow which passes by $\sigma$ at $t_0$ without staying on $\sigma$ and whose correspondent limits $\lim_{t \longrightarrow t_0^\pm} \rho(t)$ exist and are equal to $\rho_0^\pm$.

If \eqref{eq:aiai} has no non-zero roots, then no trajectory leads in or out $\Omega$ in $\sigma$.
\end{theorem}

If $\| F(\sigma) ^t \nabla g(\sigma) \|_{\mathcal{L}(N_\sigma\Lambda)} \geqslant \| \partial_\rho V_S(\sigma) \|_{N_\sigma\Lambda}$, then $\rho(t)$ may converge laterally to $0$ even if $\frac{1}{\| \nabla g(\sigma) \rho(t) \|}\nabla g(\sigma) \rho(t)$ has a well-defined lateral limit that we denote $\frac{1}{\| \nabla g(\sigma)\rho_0 \|} \nabla g(\sigma) \rho_0$. We will abusively call $\rho_0$ ``zero roots'' of \eqref{eq:aiai} and say that trajectories reach or leave $\Omega$ in $\sigma$ following the respective directions $\rho_0^\pm$ if $\lim_{t \rightarrow t_0^\pm} \frac{1}{\| \nabla g(\sigma) \rho(t) \|} \nabla g(\sigma) \rho(t)$ exists and is equal to $\frac{1}{\| \nabla g(\sigma )\rho_0^\pm \|}\nabla g(\sigma )\rho_0^\pm$. Sometimes it may be that an incoming trajectory only approaches this limit asymptotically at $t_0^- = \infty$.

\begin{theorem}\label{th:accessory_result_2}
If $\| F(\sigma) ^t \nabla g(\sigma) \|_{\mathcal{L}(N_\sigma\Lambda)} \geqslant \| \partial_\rho V_S(\sigma) \|_{N_\sigma\Lambda}$, at least one of the following affirmations holds:
\begin{itemize}
\item Equation \eqref{eq:aiai} has non-zero roots $\rho_0^\pm$ and there are unique trajectories leaving and arriving on $\Omega$ in $\sigma$ through the directions $\rho_0^\pm$;
\item Equation \eqref{eq:aiai} has ``zero roots'', in the sense that there are $\rho_0 \neq 0$ such that
\begin{equation}\label{eq:zero_roots}
F(\sigma) ^t\nabla g(\sigma)\frac{\nabla g(\sigma) \rho_0}{\| \nabla g(\sigma) \rho_0 \|} + \partial_\rho V_S(\sigma) = 0,
\end{equation}
and either there is no trajectories reaching $\Omega$ in $\sigma$ through $\rho_0$, or they exist but do not arrive onto $\Omega$ within any finite time;
\item The classical flow does not touch $\Omega$ in $\sigma$ through any well-defined direction. 
\end{itemize} 
\end{theorem}

\begin{remark}
In any case, if equation \eqref{eq:aiai} has no roots (``zero'' or non-zero), then no classical trajectory passes by $\Omega$ in $\sigma \in \Lambda$.
\end{remark}

\begin{remark}
In \cite{moi}, we will endeavour a more precise study of the link between $\nu$ and the classical flow, generalizing the link between Theorem \ref{th:main_result_4} and Theorems \ref{th:accessory_result_1} and \ref{th:accessory_result_2}.
\end{remark}

In Section \ref{sec:asymmetry_condition_in_use} we will work out the proof of Theorems \ref{th:main_result_4}, \ref{th:accessory_result_1} and \ref{th:accessory_result_2} in coordinates that are more suitable to understand $\rho(t)$ as an approaching direction. Besides, we will see in Section \ref{sec:examples}, by means of a number of examples, that these results allow a full classification of the types of trajectories that may reach or stay on the singularities. In some cases, like for $d = 1$ or $V(x) = \pm \|x\|$ in $\R^d$, they allow a full resolution of the problem when the inequality in Theorem \ref{th:main_result_4} holds, since by solving explicitly \eqref{eq:aiai} one can verify that there is only a unique trajectory leading in and out the singularity without staying thereon, and necessarily the measures will follow it and not charge $\Omega$. 

\medskip

In short, so far we have seen that whenever we have a well-defined flow, we know what the semiclassical measures do: they are transported thereby. If the flow presents trajectory splits, they necessarily happen on $\Omega$, where there is always the possibility of regime change between outsider and insider flows. Then, thanks to the measure $\nu$, we may be able to obtain enough information to decide whether the measures stay or not on the singularity, and in case they stay, we know that they will be carried by the flow generated by $\left. V_S\right|_\Lambda$.

\subsection{Third question: trajectory crossings}
Finally, a last problem is: if a measure does not stay on $\Omega$ and continues in the exterior flow, but even though there are different trajectories to take, can we derive from the well-posed quantum evolution some general criterion for choosing the actual trajectories that the measure will follow? Is there any selection principle for the classical movement of a particle under such conical potentials?

As we will see in Section \ref{sec:splitting}, the answer is negative. The path a Wigner measure (or a particle) takes after its trajectory splits depends crucially on its quantum state concentration, so any selection principle making appeal only to purely classical or semiclassical information is to be dismissed.

This can be justified by:
\begin{theorem}\label{th:main_result_3}
Let be $V(x) = -|x|$ in $\R$ and $\mu^1$ and $\mu^2$ the Wigner measures associated to the solutions of \eqref{eq:schrodinger} with initial data
$$\Psi^{\eps,1}_0(x) = \frac{1}{\eps^{\frac{1}{4}}} \Psi^1\left(\frac{x}{\sqrt{\eps}}\right) \quad {\rm and} \quad \Psi^{\eps,2}_0(x) = \frac{1}{\eps^{\frac{1}{4}}} \Psi^2\left(\frac{x}{\sqrt{\eps}}\right)e^{-i \eps^{\beta -1} x},$$
with $0 < \beta < \frac{1}{10}$, $\Psi^1, \Psi^2 \in C_0^\infty(\R)$ and $\Psi^1$ supported on $x > 0$. Then, for $t \leqslant 0$,
$$\mu_t^1(x,\xi) = \mu_t^2(x,\xi) = \delta\left(x-\frac{t^2}{2}\right) \otimes \delta\left(\xi + t\right);$$
nevertheless, for $t > 0$,  
$$\mu_t^1(x,\xi) = \delta\left(x-\frac{t^2}{2}\right) \otimes \delta\left(\xi - t\right)$$
whereas
$$\mu_t^2(x,\xi) = \delta\left(x+\frac{t^2}{2}\right) \otimes \delta\left(\xi + t\right).$$
\end{theorem}

In pictures, the particle $\mu^1$ follows the path in Figure \ref{fig:paths}(a), and the particle $\mu^2$ moves as in Figure \ref{fig:paths}(b).

\begin{figure}[htpb!]\center\footnotesize
\begin{multicols}{2}
\includegraphics[width=0.3\textwidth]{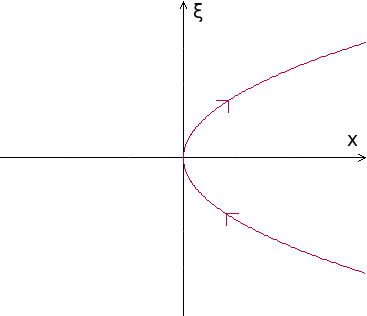} \text{(a) Trajectory of $\mu^1$.} \\
\includegraphics[width=0.3\textwidth]{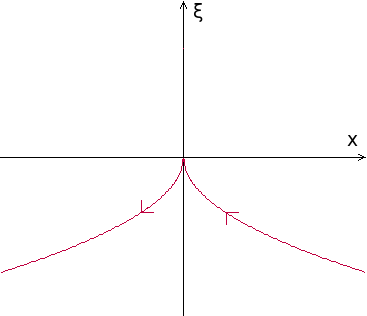} \text{(b) Trajectory of $\mu^2$.} \\
\end{multicols}
\caption{\footnotesize Trajectories followed by two different particles, coinciding for $t \leqslant 0$, but then diverging for $t > 0$, which dismisses any selection principle based only on classical information about the problem.}
\label{fig:paths}
\end{figure}

This result will be obtained with the help of approximative solutions of \eqref{eq:schrodinger} called \emph{wave packets}, which are $L^\infty\left(\R,L^2(\R^d)\right)$ functions generally of the form
$$\varphi^\eps_t(x) = \frac{1}{\eps^{\frac{1}{4}}} v_t \left(\frac{x-x(t)}{\sqrt{\eps}}\right) e^{\frac{i}{\eps}\left[\xi(t)\centerdot \left(x-x(t)\right) + S(t)\right]}$$
($S$ is the classical action), to be properly introduced in Section \ref{sec:the_wave_packets}. The standard methods using wave packet presented in that section, however, only apply for smooth flows, and in both cases the trajectories in Figure \ref{fig:paths} have problems over the axis $x= 0$, not to speak about the lack of regularity of $V$. 

In the case of the returning particle, the problem will be solved by decomposing the initial data into two pieces, for $x > 0$ and $x < 0$, and treating each one with a standard wave packet set to follow a different parabola. We will see in Proposition \ref{prop:psi_follows_the_path_1} that the Wigner measure initially set on the singularity will break out into two pieces $\mu^+$ and $\mu^-$ with weights given by the total mass of the initial data over $x> 0$ and $x< 0$ respectively, $\int_0^\infty |v_0(x) |^2dx$ and $\int_{-\infty}^0 |v_0(x) |^2dx$, each piece gliding to a different side as in Figure \ref{fig:rebounding_particle}.

Yet, this does not give a full example of non-unicity as in Theorem \ref{th:main_result_3}, since, if we evolve the pieces $\mu^+$ and $\mu^-$ to the past, we realize that they do not come from the same side, and this fact could indicate some kind of selection principle.

\begin{figure}[htpb!]\center\footnotesize
\includegraphics[width=0.3\textwidth]{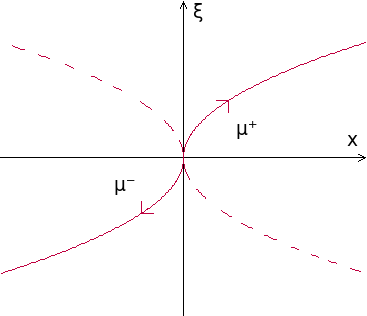}
\caption{\footnotesize Trajectories followed by the measures $\mu^+$ (right) and $\mu^-$ (left). The full line indicates the path for $t\geqslant  0$, whereas the dashed line indicates the past trajectories that the measures ought to have followed in $t < 0$ to reach the singularity.}
\label{fig:rebounding_particle}
\end{figure}

Constructing a quantum solution whose semiclassical measure behaves like in Figure \ref{fig:paths}(b) will be more difficult and will require us to consider a family of wave packets following different trajectories with smaller and smaller initial momenta $\eta$, that in some sense converge to the aimed path with $\eta = 0$, as illustrated in Figure \ref{fig:crossing_particle}. We will then study the concentration of the wave packets with $\eps$ going to $0$ at the same time as the trajectories concentrate, by making $\eta$ go to $0$ with a suitable power of $\eps$.

\begin{figure}[htpb!]\center\footnotesize
\includegraphics[width=0.35\textwidth]{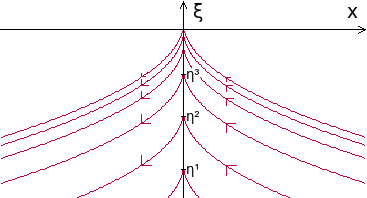}
\caption{\footnotesize The trajectories \eqref{eq:trajectory_wave_packet}} for $|\eta |= |\eta^1 | > | \eta^2 | > | \eta^3 | ...$, approaching the aimed one with $\eta = 0$.
\label{fig:crossing_particle}
\end{figure}

\subsection{Structure of the article}
In Section \ref{sec:preliminaries}, we will introduce the fundamentals of our analysis: wave packet approximations (Sec. \ref{sec:the_wave_packets}), symbolic calculus (Sec. \ref{sec:symbolic_calculus}) and two-microlocal measures (Sec. \ref{sec:two_microlocal_analysis}). In Section \ref{sec:splitting}, we will construct the solutions of the Schrödinger equation that lead to Theorem \ref{th:main_result_3}; the case keeping on the same parabola is treated in Sec. \ref{sec:same_parabola_1}, the other one in Sec. \ref{sec:other_parabola}. Finally, in Section \ref{sec:establishing_liouville} we will prove Theorem \ref{th:main_result_1}, firstly in a particular version for subspaces, what will be done step by step from Sec. \ref{sec:kinetic_term} to \ref{sec:absolute_continuity} (the part where we effectively establish the dynamical equation and the asymmetry condition being Section \ref{sec:establishing_the_equation}). Then this version will be immediately extended to the general case thanks to the coordinate change that we will have set in Section \ref{sec:reducing}. In Section \ref{sec:asymmetry_condition_in_use} we will use the asymmetry condition \eqref{eq:asymmetry_formula} to prove Theorem \ref{th:main_result_4}. Theorems \ref{th:accessory_result_1} and \ref{th:accessory_result_2} are also proven in this section, and in Sec. \ref{sec:examples} we conclude by showing with Examples \ref{ex:1} to \ref{ex:last} how the results in this article allow a full classification of the behaviours that the Wigner measures present and, sometimes, give a full description of the transport phenomenon.

\section{Preliminaries}\label{sec:preliminaries}
In this section, we will present the basics of the main tools that we use in this work. First the wave packet method for approximating solutions of the Schrödinger equation (see for example \cite{comberscure_robert} and \cite{hagedorn}, or \cite{athanas} for a generalized notion of wave packet) that we will adapt later in Section \ref{sec:splitting}, then some simple results in standard symbolic calculus (\cite{dimassi}, \cite{zworski}) which will provide a guideline for proving Theorem \ref{th:main_result_1}, and last some notions about the two-microlocal measures (\cite{CFK1}, \cite{muller}, \cite{nier}), that we will deploy in order to accomplish the necessary refined analysis for obtaining either the dynamical equation for the  Wigner measures and the asymmetry condition on the mass concentration around the singular manifold.

For the sake of simplicity, we will use these measures in a specialized version for $p$-codimensional subspaces of $\R^d$. This does not carry any loss of generality, since in Section \ref{sec:reducing} we will be able to perform a change of coordinates in the problem that will lead us to analyse such a subspace. The geometric nature of these measures, as well as that of the semiclassical ones and of Liouville-like equations, makes it possible to move between different coordinate systems respecting the structures of cotangent, normal and conormal bundles over the singular set.  

\subsection{The wave packets}\label{sec:the_wave_packets}
For a $C^2(\R^d)$ potential $V$ and one of its Hamiltonian trajectories $\left(x\left(t\right),\xi\left(t\right)\right)$, we define the wave packet with initial profile $v_0 \in L^2(\R^d)$ following $\left(x\left(t\right),\xi\left(t\right)\right)$ as
\begin{equation}\label{eq:wave_packet}
\varphi^{\eps}_t \left(x\right) = \frac{1}{\eps^{\frac{d}{4}}} v_t \left( \frac{x-x \left( t \right) }{\sqrt{\eps}} \right) e^{\frac{i}{\eps} \left[ \xi \left(t \right)\centerdot\left( x - x\left(t\right) \right) + S\left(t\right)\right]},
\end{equation}
where $S$ is the classical action $S\left(t\right) = \int_{0}^{t} \left( \frac{1}{2}\xi^{2}\left(s\right) - V\left(x\left(s\right)\right) \right)ds$ and $v$ satisfies the $\eps$-independent differential system
\begin{equation}\label{eq:wave_packet_equation}
\left\lbrace
\begin{array}{l}
i\partial_{t}v_t\left(y\right) = -\frac{1}{2}\Delta v_t\left(y\right) + \left( \frac{1}{2} \nabla^{2}V\left( x\left(t\right) \right) y \centerdot y \right) v_t\left(y\right) \\
v_{t=0}\left(y\right) = v_0 \left(y\right).
\end{array} \right.
\end{equation}

\begin{lemma}\label{lem:path_wave_packet}
Any semiclassical measure associated to the family $(\varphi^\eps)_{\eps >0}$ is
$$\mu_t\left(x,\xi\right) = \left\| v_0 \right\|_{L^2(\R^d)}^2 \delta\left(x-x(t)\right) \otimes \delta\left(\xi-\xi(t)\right).$$
\end{lemma}

\begin{proof}
A straightforward calculation. Writing down $\left< \op_\eps (a) \, \varphi^\eps_t \, , \varphi^\eps_t \right>$ for some $a \in C^\infty_0 (\R^{2d})$, performing some variable changes and a Taylor expansion: 
{\footnotesize
\begin{eqnarray*}
\left< \op_\eps (a) \, \varphi^\eps_t \, , \varphi^\eps_t \right> & = & \frac{1}{(2\pi\sqrt{\eps})^d} \int_{\R^{3d}} e^{i\xi \centerdot (x-y)} a\left(\frac{x+y}{2}+x(t),\eps \xi + \xi(t)\right) v_t\left(\frac{y}{\sqrt{\eps}}\right) \overline{v_t\left(\frac{x}{\sqrt{\eps}}\right)} d\xi dy dx \\
& = & \frac{1}{(2\pi\sqrt{\eps})^d} \int_{\R^{3d}} e^{i\xi \centerdot (x-y)} a\left(\frac{x+y}{2}+x(t),\xi(t)\right) v_t\left(\frac{y}{\sqrt{\eps}}\right) \overline{v_t\left(\frac{x}{\sqrt{\eps}}\right)} d\xi dy dx + R^\eps, \\
\end{eqnarray*}}
then integrating in $\xi$, which gives a Dirac delta, then in $y$ and changing variables once more:
$$\left< \op_\eps (a) \, \varphi^\eps_t \, , \varphi^\eps_t \right> = \int_{\R^d} a\left(\sqrt{\eps}x+x(t),\xi(t)\right) \left| v_t\left(x\right) \right|^2 dx + R^\eps.$$
The result comes from letting $\eps$ go to $0$, where the dominated convergence theorem intervenes inside the integral, and from evaluating the remainder:
{\scriptsize
\begin{eqnarray*}
R^\eps & = & \frac{\eps}{(2\pi\sqrt{\eps})^d} \int_{\R^{3d}} \int_0^1 e^{i\xi \centerdot (x-y)} \xi \centerdot \partial_\xi a\left(\frac{x+y}{2}+x(t), \eps s \xi + \xi(t)\right) v_t\left(\frac{y}{\sqrt{\eps}}\right) \overline{v_t\left(\frac{x}{\sqrt{\eps}}\right)} ds d\xi dy dx \\
& = & \frac{i\eps}{(2\pi\sqrt{\eps})^d} \int_{\R^{3d}} \int_0^1 e^{i\xi \centerdot (x-y)} \, {\rm tr} \left( \partial_x\partial_\xi a\left(\frac{x+y}{2}+x(t), \eps s \xi + \xi(t)\right)\right) v_t\left(\frac{y}{\sqrt{\eps}}\right) \overline{v_t\left(\frac{x}{\sqrt{\eps}}\right)} ds d\xi dy dx \\ 
& \sim & \mathcal{O}(\eps),
\end{eqnarray*}}
which completes the proof.
\end{proof}

Another virtue of the wave packets is that they provide approximative solutions to the Schrödinger equation with convenient initial data, as stated in:
\begin{proposition}\label{prop:evolution_norm}
For fixed initial $(x_0,\xi_0) \in \R^{2d}$, consider a Hamiltonian trajectory $(x(t),\xi(t))$ for a potential $V$ not necessarily smooth everywhere over the space trajectory. Let be $\left]0,\tau\right[ \subset \R$ and $\Upsilon = \left\lbrace x \in \R^d : x = x(t) \; {\rm for} \, t \in \left] \, 0 ,\tau \right[ \right\rbrace$. If $\nabla^2 V$ exists and is Lebesgue integrable in $\Upsilon$, and if $\Psi^\eps$ is the solution of the Schrödinger equation with potential $V$ and initial data $\Psi^\eps_{0} (x) = \frac{1}{\eps^{\frac{1}{4}}}v_0\left(\frac{x-x_0}{\sqrt{\eps}}\right)e^{\frac{i}{\eps} \xi_0 \centerdot (x - x_0)}$, then letting be $\varphi^\eps$ the wave packet initially centred in $(x_0,\xi_0)$ with profile $v_0$, we have
$$\left\| \Psi_{\tau}^{\eps} - \varphi_{\tau}^{\eps} \right\|_{L^2(\R^d)} \leqslant \int_{\left] 0,\tau \right[} \, \left\| \frac{1}{\eps} R_{s}^{\eps} v_{s} \right\|_{L^2(\R^d)} ds,$$
where 
\begin{equation}\label{eq:error_wave_packet}
R^{\eps}_t\left(y\right) = V\left( x\left(t\right) + \sqrt{\eps}y \right) - V\left(x\left(t\right)\right) - \sqrt{\eps}\nabla V\left(x\left(t\right)\right) y - \frac{\eps}{2}\nabla^{2}V\left(x\left(t\right)\right) y \centerdot y.
\end{equation}
\end{proposition}

\begin{proof}
After a direct calculation, one obtains the following differential system for $\varphi^\eps$:
\begin{equation}\label{eq:wave_packet_equation_phi}
\left\lbrace
\begin{array}{l}
i\eps\partial_{t}\varphi^{\eps}_t\left(x\right) = \hat{H}^{\eps} \varphi^{\eps}_t\left(x\right) - R^{\eps}_t\left(\frac{x-x\left(t\right)}{\sqrt{\eps}}\right)\varphi^{\eps}_t\left(x\right) \\
\varphi^{\eps}_{t=0}\left(x\right) = \frac{1}{\eps^{\frac{d}{4}}}v_0\left(\frac{x-x_{0}}{\sqrt{\eps}}\right)e^{\frac{i}{\eps}\xi_{0}\centerdot\left(x-x_{0}\right)},
\end{array} \right.
\end{equation}
where $H^\eps$ is the Hamiltonian operator \eqref{eq:original_hamiltonian} with $V$ as stated, $R^{\eps}$ is explicitly given by equation \eqref{eq:error_wave_packet}. Now, we compare $\Psi^{\eps}$ and $\varphi^{\eps}$ by evaluating 
$$\frac{d}{dt}\left\| \Psi_{t}^{\eps} - \varphi_{t}^{\eps} \right\|_{L^2(\R^d)}^{2} = 2 \, \text{Re} \left< \Psi_{t}^{\eps} - \varphi_{t}^\eps , \partial_{t}\left( \Psi_{t}^{\eps} - \varphi_{t}^\eps\right) \right>_{L^2(\R^d)},$$
which gives, in view of the equations for $\Psi^\eps$, $\varphi^\eps$ and the self-adjointness of $H^\eps$, 
\begin{eqnarray*}
\frac{d}{dt}\left\| \Psi_{t}^{\eps} - \varphi_{t}^{\eps} \right\|_{L^2(\R^d)} \leqslant \left\| \frac{1}{\eps} R_{t}^{\eps} \, v_{t} \right\|_{L^2(\R^d)},
\end{eqnarray*}
thus, for any $\alpha,\beta \in \left] 0 , \tau\right[$, we have
$$\left\| \Psi^\eps_\beta - \varphi^\eps_\beta \right\|_{L^2(\R^d)} - \left\| \Psi^\eps_\alpha - \varphi^\eps_\alpha \right\|_{L^2(\R^d)} \leqslant \left| \int_\alpha^\beta \left\| \frac{1}{\eps} R^\eps_s v_s \right\|_{L^2(\R^d)}ds \right|.$$  

Naturally, the function $t \longmapsto \left\| \Psi^\eps_t - \varphi^\eps_t \right\|_{L^2(\R^d)}$ is continuous and, at $t = 0$, $\Psi^{\eps}_0\left(x\right) = \varphi^{\eps}_0\left(x\right) = \frac{1}{\eps^{\frac{d}{4}}}v_0\left(\frac{x-x_0}{\sqrt{\eps}}\right)e^{\frac{i}{\eps}\xi_0\centerdot x_0}$. Hence, by choosing sequences $\alpha_n$ and $\beta_n$ in $\left]0,\tau\right[$ such that $\alpha_n \longrightarrow 0$ and $\beta_n \longrightarrow \tau$, the proposition follows.
\end{proof}
 
\begin{corollary}
Call $\mu$ the semiclassical measure linked to the exact family of solutions $\left(\Psi^\eps\right)_{\eps>0}$ with initial data as in the theorem above. If $\nabla^3 V$ exists and is Lebesgue integrable, then
$$\left\| \Psi_{t}^{\eps} - \varphi_{t}^{\eps} \right\|_{L^2(\R^d)} \lesssim |t|\sqrt{\eps}$$
and, consequently, given $T > 0$, for any $t \in \left[-T,T\right]$,
$$\mu_{t}\left(x,\xi\right) = \delta\left(x-x\left(t\right)\right) \otimes \delta\left(\xi-\xi\left(t\right)\right).$$
\end{corollary}

\begin{proof}
$V$ being at least of class $C^3(\R^d)$, one can verify from a Taylor formula that
$$R^\eps (y,t) = \frac{\eps\sqrt{\eps}}{2} \int_{0}^{1} \nabla^{3} V \left( x\left(t\right) + s\sqrt{\eps}y \right) \left( 1 - s \right)^{2} ds,$$
and, moreover, that $R^{\eps}$ introduces in the Schrödinger equation a $L^2(\R^d)$ error of order $\mathcal{O}\left(\eps\sqrt{\eps}\right)$. Thus, from Proposition \ref{prop:evolution_norm}, it is clear that $\Psi_{t}^{\eps} = \varphi_{t}^{\eps} + \mathcal{O}\left(|t|\sqrt{\eps}\right)$ in $L^2(\R^d)$; for any $t \in [-T,T]$, this gives, when $\eps\longrightarrow 0$, that the Wigner measure of $\Psi^\eps$ shall coincide with that of the wave packets. The conclusion comes from Lemma \ref{lem:path_wave_packet}.
\end{proof}

\begin{remark}
Actually, the approximation in the corollary remains good for $t$ smaller than the Ehrenfest time $t_{E} = \ln \frac{1}{\eps}$, as $\sqrt{\eps} \ln \frac{1}{\eps} \longrightarrow 0$ when $\eps \longrightarrow 0$; more details in \cite{comberscure_robert}. Estimates beyond the Ehrenfest time are given in \cite{schubert}.
\end{remark}
 
Observe that even if $V$ is not as regular as we required, we can still write $R^{\eps}$ as in \eqref{eq:error_wave_packet} for any $t$ such that $\nabla V\left(x\left(t\right)\right)$ and $\nabla^{2} V\left(x\left(t\right)\right)$ make sense, although in this case it is not clear which is the order of the approximation the wave packet furnishes, nor even whether it is negligible in the semiclassical limit.

Finally, observe that it is also possible to write the actual solution $\Psi^{\eps}$ with initial state $\Psi^\eps_0 = \frac{1}{\eps^{\frac{d}{4}}}v_0\left(\frac{x-x_0}{\sqrt{\eps}}\right)e^{\frac{i}{\eps}\xi_0 \centerdot (x-x_0)}$ under the wave packet form: one defines a $u^{\eps}$ such that
\begin{equation}\label{eq:definition_u}
\Psi^{\eps}\left(x,t\right) = \frac{1}{\eps^{\frac{d}{4}}} u^{\eps}\left( \frac{x-x\left(t\right)}{\sqrt{\eps}} , t \right) e^{\frac{i}{\eps} \left[ \xi\left(t\right) \centerdot \left( x - x\left(t\right) \right) + S\left(t\right) \right] },
\end{equation}
which consequently obeys to
{\small \begin{equation}\label{eq:approximative_wave_packet_equation}
\left\lbrace
\begin{array}{l}
i\partial_{t}u^{\eps}\left(y,t\right) = -\frac{1}{2}\Delta u^{\eps}\left(y,t\right) + \left( \frac{1}{2}\nabla^{2}V\left(x\left(t\right)\right) y \centerdot y \right) u^{\eps}\left(y,t\right) + \frac{1}{\eps}R^{\eps}\left(y,t\right)u^{\eps}\left(y,t\right) \\
u^{\eps}\left(y,0\right) = v_0(y),
\end{array}\right.
\end{equation}}
which is nothing else than the exact Schrödinger equation written in a different form.

\subsection{Symbolic calculus}\label{sec:symbolic_calculus}
Let us consider the $\eps$-pseudodifferential operators $\op_{\eps}\left(a\right) \in \mathcal{L}\left(L^2(\R^d)\right)$ of symbols $a \in C_{0}^{\infty} \left(\mathbb{R}^{2d}\right)$ given in formula \eqref{eq:pseudor}. Of central importance is the fact that they are uniformly bounded in $\mathcal{L}\left(L^2(\R^d)\right)$ with respect to $\eps$ (see, for instance, Corollary 2.4 in \cite{x_ups} and the subsequent discussion): there exist constants $K,\tilde{K} > 0$ such that
\begin{equation}
\| \op_\eps (a) \|_{\mathcal{L}\left(L^2(\R^d)\right)} \leqslant K \sup_{\underset{| \alpha | \leqslant d+1}{\alpha \in \mathbb{N}^d_0}} \; \sup_{\xi \in \R^d} \; \int_{\R^d} \left| \partial_{x}^\alpha a(x,\xi) \right| dx
\label{eq:estimation_x}
\end{equation}
or else
\begin{equation}
\| \op_\eps (a) \|_{\mathcal{L}\left(L^2(\R^d)\right)} \leqslant \tilde{K} \sup_{\underset{| \alpha | \leqslant d+1}{\alpha \in \mathbb{N}^d_0}} \; \sup_{x \in \R^d} \; \int_{\R^d} \left| \partial_{\xi}^\alpha a(x,\xi) \right| d\xi.
\label{eq:estimation_xi}
\end{equation}

\begin{remark}\label{rem:schur_upper_bound}
Inequalities \eqref{eq:estimation_x} and \eqref{eq:estimation_xi} give upper bounds for the Schur estimate of the norm of $\op_\eps(a)$. 
\end{remark}

Nonetheless, formula \eqref{eq:pseudor} can be used for more general symbols, although we may lose boundedness, good properties for symbolic calculation, and be forced to restrict their domains. In particular, for $V$ satisfying the Kato-Rellich conditions,
$$\hat{H}^\eps = \op_\eps(h), \quad {\rm with} \quad h(x,\xi) = \frac{\xi^2}{2} + V(x),$$
doted with domain $H^2(\R^d)$, is unbounded, although it is still self-adjoint. 

Thus, taking a test function $\Xi \in C_0^\infty(\R)$ and $\Psi^{\eps}$ a solution of \eqref{eq:schrodinger}, the semiclassical measure $\mu$ linked to the family $(\Psi^\eps)_{\eps> 0}$ may be given as
\begin{eqnarray}
\left< \, \mu \left(t,x,\xi\right) \, , \Xi(t) \, a\left(x,\xi\right) \, \right>_{\R\times\R^{2d}} = {\rm sc}\lim \int_{\R} \Xi(t) \left< \op_{\eps}\left(a\right) \Psi_{t}^{\eps} \, , \Psi_{t}^{\eps} \right> dt
\label{eq:mu_abs_continuous}
\end{eqnarray}
(here ${\rm sc} \lim$ stands for \emph{semiclassical limit}, abstracting which particular subsequence $(\eps_{n_{k}})_{k \in \N}$ is to be taken). From this expression we can evaluate the distribution $\partial_t \mu$:
\begin{eqnarray}
\left< \, \partial_{t} \mu \left(t,x,\xi\right) \, , \Xi(t) \, a\left(x,\xi\right) \, \right>_{\R\times\R^{2d}}  & = & -\int_{\R\times\R^{2d}} \Xi'(t) \, a\left(x,\xi\right) \, d\mu\left(t,x,\xi\right) \nonumber \\
& = & {\rm sc}\lim \int_{\R} \Xi(t) \frac{d}{dt} \left< \op_{\eps}\left(a\right) \Psi_{t}^{\eps} \, , \Psi_{t}^{\eps} \right> dt;
\label{eq:d_mu}
\end{eqnarray}
moreover, in view of the Schrödinger equation \eqref{eq:schrodinger},
\begin{equation}\frac{d}{dt} \left< \op_{\eps}\left(a\right) \Psi_{t}^{\eps} \, , \Psi_{t}^{\eps} \right> =  \left< \frac{i}{\eps} \left[ \hat{H}^{\eps} , \op_{\eps}\left(a\right) \right] \Psi_{t}^{\eps} \, , \Psi_{t}^{\eps} \right>.
\label{eq:d_op}
\end{equation}
%
By standard symbolic calculus (many of whose formul\ae \, may be found in \cite{zworski}, for instance), in the smooth case we would get 
\begin{equation}\label{eq:commutators_result}
\frac{i}{\eps}\left[ \hat{H}^{\eps} \, , \op_{\eps}\left(a\right) \right] = \op_\eps \left( (\xi \centerdot \partial_x - \nabla V(x) \centerdot \partial_\xi) a \right) + \mathcal{O}(\eps),
\end{equation}
which ultimately induces equation \eqref{eq:liouville} for $\mu$ in the sense of the distributions.

The conical singularities that $V$ presents, however, will require a specific treatment. Roughly, we will have to re-derive ``by hand'' adapted formul\ae \, for a correct symbolic calculus with such potentials, which we will do progressively in Sections \ref{sec:the_inner_part}, \ref{sec:the_outer_part} and \ref{sec:the_middle_part}. 

\subsection{Two-microlocal analysis}\label{sec:two_microlocal_analysis}
Now, let us define a new symbol class $S(p)$ composed by symbols $a \in C^\infty(\R^{2d+p})$ such that
\begin{itemize}
\item For each $\rho \in \R^p$, $(x,\xi) \longmapsto a(x,\xi,\rho)$ is compactly supported on $\R^{2d}_{x,\xi}$.
\item There exists some $R_0 > 0$ and a function $a_\infty \in C^\infty (\R^{2d}\times \mathcal{S}^{p-1})$ such that, for $\| \rho \| > R_0$, one has $a(x,\xi,\rho) = a_\infty \left(x,\xi,\frac{\rho}{\|\rho\|}\right)$.  
\end{itemize}
These symbols will be quantized as
$$\op_\eps^\sharp (a(x,\xi,\rho)) = \op_\eps \left(a\left(x,\xi,\frac{x'}{\eps}\right)\right);$$
observe that the right-hand term above is just the banal quantization of a $\eps$-dependent $\R^{2d}$ function as in \eqref{eq:pseudor}.

\begin{proposition}\label{prop:two_microlocal_measures}
There exists a measure $\nu_\infty$ on $\R\times \R^{2d-p}\times S^{p-1}$ and a trace class operator valued measure $M$ on $\R\times\R^{2(d-p)}$, both positive, such that, for $a \in S(p)$ and $\Xi \in C_0^\infty(\R)$,
\begin{eqnarray}\label{eq:two_microlocal_decomp}
\int_\R \Xi(t) \left< \, \op_\eps^\sharp (a) \Psi^\eps_t \, , \, \Psi^\eps_t \, \right> dt & \underset{\eps \rightarrow 0} {\longrightarrow} & \left< \mu(t,x,\xi) 1\!\!1_{\lbrace x' \neq 0 \rbrace} \, , \Xi(t) \, a_\infty\left(x,\xi,\frac{x'}{\|x'\|}\right) \right>_{\R \times \R^{2d}} \nonumber \\
& + & \left< \, \delta (x') \otimes\nu_\infty(t,x'',\xi,\omega) \, , \Xi(t) \, a_\infty \left(x,\xi,\omega\right) \, \right>_{\R \times \R^{2d} \times \mathcal{S}^{p-1}} \nonumber \\
& + & {\rm tr} \left< \, M(t,x'',\xi'') \, , \, \Xi(t) \, a^w\left(0,x'',\partial_y,\xi'',y\right) \, \right>_{\R \times \R^{2(d-p)}},
\end{eqnarray}
where $a^w(0,x'',\partial_y,\xi'',y)$ is the Weyl quantization of the symbol $(y,\zeta) \longmapsto a(0,x'',\zeta,\xi'',y)$ with $\eps = 1$ and $\mu$ is the usual Wigner measure related to $\Psi^\eps$. 

Furthermore, for a smooth compactly supported function $(x'',\xi'') \longmapsto T_{(x'',\,\xi'')}$ taking values on the set of compact operators on $\R^p$, one has 
{\small
$$\left< M(t,x'',\xi'') \, , \, \Xi(t) \, T_{(x'',\,\xi'')} \right>_{\R \times \R^{2(d-p)}} = {\rm sc} \lim \int_{\R \times \R^{2(d-p)}} \Xi(t) \, T_{(x'',\,\xi'')} U^\eps_{(t,\,x'',\,\xi'')}dx''d\xi''dt,$$}
where $U^\eps_{(t,\,x'',\,\xi'')}$ is the trace class operator with kernel
{\small
\begin{equation}\label{eq:k_u}
kU^\eps_{(t,\,x'',\,\xi'')}(y',x') = \int_{\R^{d-p}} \frac{e^{\frac{i}{\eps} \xi'' \centerdot y''}}{\left(2\pi \eps\right)^{d-p}}  \, \Psi^\eps_t \left( \eps y', x''-\frac{y''}{2} \right) \overline{\Psi^\eps_t \left( \eps x', x''+ \frac{y''}{2} \right)}dy''.
\end{equation}}

Finally, the terms in \eqref{eq:two_microlocal_decomp} are obtained respectively from those in the decomposition 
{\small
$$a(x,\xi,\rho) = a(x,\xi,\rho) \left(1-\chi\left(\frac{x'}{\delta}\right)\right) + a(x,\xi,\rho) \left(1-\chi\left(\frac{\rho}{R}\right)\right)\chi\left(\frac{x'}{\delta}\right) + a(x,\xi,\rho) \chi\left(\frac{\rho}{R}\right)$$}
in the limit when $\eps \longrightarrow 0$, then $R \longrightarrow \infty$, and last $\delta \longrightarrow 0$, where $\chi$ is a cut-off function such that $0 \leqslant \chi \leqslant 1$, $\chi(x) = 1$ for $\| x \| < \frac{1}{2}$ and $\chi(x) = 0$ for $\|x\| \geqslant 1$.
\end{proposition}

\begin{remark}\label{rem:p_equals_d}
If $p = d$, then $U^\eps_t = \left| \tilde{\Psi}^\eps_t \right>\left< \tilde{\Psi}^\eps_t \right|^*$ is just the adjoint of the projector over $\tilde{\Psi}^\eps_t(x) = \eps^\frac{d}{2}\Psi^\eps(\eps x)$, with kernel $kU_t^\eps(x,y) = \tilde{\Psi}^\eps_t(y) \overline{\tilde{\Psi}^\eps_t(x)}$. It follows that 
$${\rm tr} \left( T \, U^\eps_t \right) = \left< \, T \, \tilde{\Psi}^\eps_t \, , \, \tilde{\Psi}^\eps_t \, \right>.$$
Then, because $T$ is compact, in the limit $\eps \longrightarrow 0$ one has that it is simply $\left< T \tilde{\Psi}_t , \tilde{\Psi}_t \right>$, where $\Psi_t$ is some weak limit of $\Psi^\eps_t$.
\end{remark}

A very general treatment of this result can be found in \cite{CFK1}. The introduction of $U^\eps$ is a trivial addition of ours in order to enlighten the calculations to come.

\begin{remark}\label{rem:definition_m}
Observe that $M$ induces a measure $\mathfrak{m}$ on $\R \times \R^{{2d}}$ by means of the formula
\begin{eqnarray*}
\left< \, \delta(x') \otimes \mathfrak{m}(t,x'',\xi,\rho) \, , \, \Xi(t) \, a(x,\xi,\rho) \, \right>_{\R\times\R^{2d+p}} \hspace{4cm} \\ = {\rm tr} \left< \, M(t,x'',\xi'') \, , \, \Xi(t) \, a^w\left(0,x'',\partial_y,\xi'',y\right) \, \right>_{\R \times \R^{2(d-p)}}.
\end{eqnarray*}
Above, $a$ has no need to be in $S(p)$; it is sufficient that it be compact supported in all variables.
\end{remark}

\begin{lemma}\label{lem:m_0_gives_m_0}
$M=0$ if and only if $\mathfrak{m} = 0$.
\end{lemma}

\begin{proof}
That $M = 0$ implies $\mathfrak{m} = 0$, it is obvious. For the converse, it is necessary to show that $\mathfrak{m} = 0$ implies ${\rm tr} \left< M , T \right> = 0$ for all compact supported functions $(t,\,x'',\xi'') \longmapsto T_{(t,\,x'',\,\xi'')}$ taking values in the set of compact operators, since this is the set whose dual are the trace class operators.

Observe that it is sufficient to consider $T_{(t,\,x'',\,\xi'')}$ Hilbert-Schmidt, given that these operators are dense in the set of the compact ones. So, we can consider that it has a kernel $kT_{(t,x'',\,\xi'')} \in L^2\left(\R^{2p}_{x',y'}\right)$ and, defining 
$$a(t,x,\xi,\rho) = \chi(x') \, \mathcal{F}_{y' \rightarrow \xi'} \left(kT_{(t,\,x'',\,\xi'')} \left(\rho+\frac{y'}{2},\rho-\frac{y'}{2}\right)\right)$$
with some $\chi \in C_0^\infty(\R^p)$, it follows that $T_{(t,\,x'',\,\xi'')} = a^w(t,0,x'',\partial_y,\xi'',y)$ and we are done.
\end{proof}

\begin{lemma}\label{lem:m_absolutely_continuous}
The measure $\mathfrak{m}$ is absolutely continuous with respect to $d\xi' d\rho$.
\end{lemma}

\begin{proof}
Indeed, if $a(x,\xi,\rho) = b(x,\xi,\rho)$ almost everywhere with respect to $d\xi'd\rho$, then $a^w(0,x'',\partial_y,\xi'',y) = b^w(0,x'',\partial_y,\xi'',y)$, for taking any $f \in L^2(\R^p)$ and working out the definition: 
\begin{eqnarray*}
a^w(0,x'',\partial_y,\xi'',y) f(y) & = & \frac{1}{(2\pi)^p} \int_{\R^{2p}} e^{i \xi' \centerdot (y - \rho)} a\left(0,x'',\xi',\xi'',\frac{y + \rho}{2}\right) f(\rho) \, d\xi' d\rho \\
& = & \frac{1}{(2\pi)^p} \int_{\R^{2p}} e^{i \xi' \centerdot (y - \rho)} b\left(0,x'',\xi',\xi'',\frac{y + \rho}{2}\right) f(\rho) \, d\xi' d\rho \\
& = & b^w(0,x'',\partial_y,\xi'',y) f(y);
\end{eqnarray*}
then it follows from the definition in Remark \ref{rem:definition_m} that $\left< \mathfrak{m}(t,x'',\xi,\rho) \, , \Xi(t)a(0,x'',\xi,\rho) \right> = \left< \mathfrak{m}(t,x'',\xi,\rho) \, , \Xi(t)b(0,x'',\xi,\rho) \right>$ whenever for each $x'' \in \R^{d-p}$, $a(0,x'', \cdot\,)$ and $b(0,x'',\cdot\,)$ differ only within a subset of $\R^{2p}_{\xi',\rho}$ with null Lebesgue measure, implying the lemma.
\end{proof}

\begin{lemma}\label{lem:two_microlocal_decomposition}
The semiclassical measure $\mu$ decomposes as
\begin{eqnarray*}
\mu(t,x,\xi) = \mu(t,x,\xi) 1\!\!1_{\lbrace x' \neq 0 \rbrace} + \delta(x') \otimes \int_{\mathcal{S}^{p-1}} \nu_\infty (t,x'',\xi,d\omega) \hspace{1.25cm} \\ + \delta(x') \otimes \left( \int_{\R^p} \mathfrak{m}_{(\xi',\rho)}(t,x'',\xi'')\,d\rho \right) d\xi',
\end{eqnarray*}
where $\nu_\infty$ is as in Proposition \ref{prop:two_microlocal_measures} and $\mathfrak{m}_{(\xi',\rho)}(t,x'',\xi'') \, d\rho \, d\xi' = \mathfrak{m}(t,x'',\xi,\rho)$.
\end{lemma}

\begin{proof}
If $a \in C_0^\infty$, define $\tilde{a}(x,\xi,\rho) = a(x,\xi)$ for all $\rho \in \R^p$. Thus, $\tilde{a} \in S(p)$ and $\op_\eps(a) = \op_\eps^\sharp\left(\tilde{a}\right)$, and we can use Proposition \ref{prop:two_microlocal_measures}.
\end{proof}

With the two-microlocal measures, we are equipped to tackle the analysis of the singular term of the commutator in \eqref{eq:d_op}.

\section{Approaching solutions with wave packets}\label{sec:splitting}
In the Introduction we pointed out that the non-uniqueness of the classical flow for the present case only plays a relevant role when the initial data concentrate to a point belonging to a trajectory that leads to the singularity. The behaviour of the measure will depend on the concentration rate and oscillations of the quantum states $\Psi^\eps$ as well as on other characteristics of this family over the crossings, such as the region where these states are supported. 

Below, we will prove some results that altogether are slightly more general than Theorem \ref{th:main_result_3}. We will present concrete cases of solutions to the Schrödinger equation with the conical potential $V(x) = -|x|$, $x \in \R$, that concentrate to a branch of one of the parabol\ae \, leading to the singularity in Figure \ref{fig:flow_classical}(b), and thereafter either swap to the other parabola (Section \ref{sec:other_parabola}, Proposition \ref{prop:crossing_psi}) or keep on the same one (Sections \ref{sec:same_parabola_1}, Proposition \ref{prop:psi_follows_the_path_1}).

These examples refute any possibility of a classical selection principle allowing one to predict the evolution of a particle (\emph{i.e.}, a Wigner measure concentrated to a single point) after it touches the singularity, since they show two particles subjected to the same potential and following the same path for any $t < 0$, but then going each to a different side for $t > 0$.

\subsection{Measures rebounding at the singularity}\label{sec:same_parabola_1}
Let us consider the trajectories
\begin{equation}\label{eq:trajectory_wave_packet_other}
\left\lbrace
\begin{array}{l}
\xi_{\pm}\left(t\right) = \pm t \\
x_{\pm}\left(t\right) = \pm \frac{t^2}{2}
\end{array}
\right. \quad {\rm for} \; t \in \R.
\end{equation}
In this section we will prove:

\begin{proposition}\label{prop:psi_follows_the_path_1}
Let be $\Psi^\eps$ the solution to the Schrödinger equation \eqref{eq:schrodinger} with $V(x) = -|x|$ in $\R$ with initial datum
$$\Psi^\eps_0 \left(x\right) = \frac{1}{\eps^{\frac{1}{4}}} a \left(\frac{x}{\sqrt{\eps}}\right),$$
with $a \in C_0^{\infty} \left(\R\right)$.

For any $t \in \R$ the semiclassical measure associated to the family $\left(\Psi^\eps\right)_{\eps > 0}$ is given by 
$$\mu_t\left(x,\xi\right) = p^{+} \, \delta\left(x-x_{+}(t)\right) \otimes \delta\left(\xi- \xi_{+}(t)\right) + p^{-} \, \delta\left(x-x_{-}(t)\right) \otimes \delta\left(\xi- \xi_{-}(t)\right),$$
where the weights $p^{\pm}$ are given by
$$p^\pm  = \pm \int_0^{\pm \infty} \left| a(x) \right|^2 dx.$$ 
\end{proposition}

\begin{proof}
Given an arbitrary $\delta > 0$, let us cut the evolved state $\Psi^\eps_t$ into three parts,  
$$\Psi^{\eps}_t\left(x\right) = \Psi_{+,t}^{\eps,\delta}\left(x\right) + \Psi_{\centerdot,t}^{\eps,\delta}\left(x\right) + \Psi_{-,t}^{\eps,\delta}\left(x\right),$$
where $\Psi_{+}^{\eps,\delta}$, $\Psi_{\centerdot}^{\eps,\delta}$ and $\Psi_{-}^{\eps,\delta}$ solve the Schrödinger equation with initial data
\begin{equation*}
\begin{array}{lllll}
\Psi_{+,0}^{\eps,\delta}\left(x\right) & = & \frac{1}{\eps^{\frac{1}{4}}}a\left(\frac{x}{\sqrt{\eps}}\right)\chi_{+}^{\delta}\left(\frac{x}{\sqrt{\eps}}\right) & & {\rm supp} \; \chi_{+}^{\delta} \subset \left\lbrace x \in \R : x \geqslant \delta \right\rbrace \vspace{0.1cm} \\
\Psi_{\centerdot,0}^{\eps,\delta}\left(x\right) & = & \frac{1}{\eps^{\frac{1}{4}}}a\left(\frac{x}{\sqrt{\eps}}\right)\chi_{\centerdot}^{\delta}\left(\frac{x}{\sqrt{\eps}}\right) & \quad \text{with} \quad & {\rm supp} \; \chi_{\centerdot}^{\delta} \subset \left\lbrace x \in \R : -2 \delta \leqslant x \leqslant 2 \delta \right\rbrace \vspace{0.1cm} \\
\Psi_{-,0}^{\eps,\delta}\left(x\right) & = & \frac{1}{\eps^{\frac{1}{4}}}a\left(\frac{x}{\sqrt{\eps}}\right)\chi_{-}^{\delta}\left(\frac{x}{\sqrt{\eps}}\right) & & {\rm supp} \; \chi_{-}^{\delta} \subset \left\lbrace x \in \R : x \leqslant -\delta \right\rbrace \\
\end{array}
\end{equation*}
chosen in such a way that $\chi_{+}^{\delta} + \chi_{\centerdot}^{\delta} + \chi_{-}^{\delta} = 1$, all these three functions smooth and taking values in $\left[0,1\right]$.

The middle term's semiclassical measure has total mass of order $\delta \max_{x\in \R} |a(x)|^2$; as a consequence, the full Wigner measure of $\Psi^\eps$ will be, for any $\delta > 0$:
\begin{equation}\label{eq:measure_decomp}
\mu = \mu^{\delta,+} + \mu^{\delta,-} + \gamma^{\delta} + \mathcal{O}(\delta),
\end{equation}
where $\mu^{\delta,\pm}$ are the measures associated to $\Psi^{\eps,\delta}_\pm$, $\mathcal{O}(\delta)$ is a measure with total mass of order $\delta$ issued from the middle term and its interferences with the other terms, and $\gamma^\delta$ is the interference measure between $\Psi^{\eps,\delta}_-$ and $\Psi^{\eps,\delta}_+$, which satisfies, for any strictly positive test function $b \in C_0^\infty\left(\R\times\R^2\right)$, the estimate
\begin{equation}\label{eq:interference_term}
\left| \left< \, \gamma^\delta \, , \, b \, \right> \right| \leqslant \sqrt{\left< \, \mu^{\delta,+} \, , \, b \, \right>\left< \, \mu^{\delta,-} \, , \, b \, \right>},
\end{equation}
as widely known in semiclassical calculus\footnote{As a short justification of this estimate, take $b \in S(\R^{2d})$ strictly positive and define $a = \sqrt{b}$, which will also be a smooth function (since $b$ is strictly positive), and this will give $\op_\eps(b) = \op_\eps(a)^2 + \mathcal{O}(\eps)$ in $\mathcal{L}\left(L^2(\R^d)\right)$. Now calculate:
\begin{eqnarray*}
\left| \left< \op_\eps(b) \Psi^\eps_+ \, , \Psi^\eps_- \right> \right| & = & \left| \left< \op_\eps(a)^2 \Psi^\eps_+ \, , \Psi^\eps_- \right> + \mathcal{O}(\eps) \right| \\
& = & \left| \left< \op_\eps(a) \Psi^\eps_+ \, , \op_\eps(a) \Psi^\eps_- \right> + \mathcal{O}(\eps) \right| \\
& \leqslant & \left\| \op_\eps(a) \Psi^\eps_+ \right\| \left\| \op_\eps(a) \Psi^\eps_- \right\| + \mathcal{O}(\eps) \\
& = & \sqrt{\left< \op_\eps(a) \Psi^\eps_+ \, , \op_\eps(a) \Psi^\eps_+ \right> \left< \op_\eps(a) \Psi^\eps_- \, , \op_\eps(a) \Psi^\eps_- \right>} + \mathcal{O}(\eps) \\
& = & \sqrt{\left< \op_\eps(b) \Psi^\eps_+ \, , \Psi^\eps_+ \right> \left< \op_\eps(b) \Psi^\eps_- \, , \Psi^\eps_- \right>} + \mathcal{O}(\eps);
\end{eqnarray*}
the result comes in the limit where $\eps$ goes to $0$.}. Remark that $\gamma^\delta$ is not necessarily positive.

\medskip

For the study of $\mu^{\delta,\pm}$, let us introduce $\varphi^{\eps,\delta}_{\pm}$, the wave packets defined in \eqref{eq:wave_packet} for $t \in \R$, having profiles $v^{\delta,\pm}$ that obey to
\begin{equation}\label{eq:profile_pm}
\left\lbrace
\begin{array}{l}
i\partial_t v^{\delta,\pm}_t(y) = -\frac{1}{2}\Delta v^{\delta,\pm}_t(y) \\
v^{\delta,\pm}_0(y) = a\left(y\right)\chi^\delta_\pm \left(y\right),
\end{array}\right.
\end{equation}
which is nothing more than the profile equation \eqref{eq:wave_packet_equation} \emph{with the smooth potentials} $\tilde{V}^\pm (x) = \mp x$ (which admit the trajectories \eqref{eq:trajectory_wave_packet_other}).

\begin{lemma}\label{lem:psi_follows_varphi_other}
For any $\delta > 0$, $$\lim_{\eps \longrightarrow 0} \left\| \Psi^{\eps,\delta}_\pm - \varphi^{\eps,\delta}_\pm \right\|_{L^\infty\left(\left[-T,T\right],L^2(\R)\right)} = 0.$$
\end{lemma}
\noindent (The proof is postponed.) 

\medskip

So, the Wigner measures for the components $\Psi_{\pm}^{\eps,\delta}$ are the same as those for $\varphi_{\pm}^{\eps,\delta}$, which one computes explicitly:
$$\mu_t^{\delta,\pm}\left(x,\xi\right) = \pm \delta\left( x - x_{\pm}\left(t\right) \right)\otimes \delta\left( \xi - \xi_{\pm}\left(t\right) \right)\int_{0}^{\pm\infty} \left| a\left(y\right) \, \chi_{\pm}^\delta \left(y\right) \right|^{2}dy;$$
observe that Equation \eqref{eq:interference_term} implies that $\gamma^{\delta}$ is supported on the intersection of the supports of $\mu^{\delta,+}$ and $\mu^{\delta,-}$, but from last formula this intersection turns out to be different from the empty set only for $t = 0$, or, more rigorously, it is contained within $\lbrace t = 0 \rbrace \times \R^{2d}$. It happens that $\gamma^\delta$ is absolutely continuous with respect to $\mu^{\delta,\pm}$, and these are absolutely continuous with respect to the Lebesgue measure $dt$, as we saw in Section \ref{sec:initial_considerations}. As a conclusion, $\gamma^\delta = 0$.

Finally, as $\delta$ is arbitrary, we take the limit $\delta \longrightarrow 0$ and it follows from \eqref{eq:measure_decomp} that $\mu = \mu^+ + \mu^-$, where:
$$\mu^{\pm}\left(t,x,\xi\right) = p^\pm \, \delta\left( x - x_{\pm}\left(t\right) \right)\otimes \delta\left( \xi - \xi_{\pm}\left(t\right) \right),$$
as we had in the proposition's statement.
\end{proof}

\begin{proof}[Proof of Lemma \ref{lem:psi_follows_varphi_other}]
To begin with, since for $t \neq 0$ we have $\nabla^j V(x_\pm(t)) = \nabla^j \tilde{V}^\pm (x_\pm(t))$ for any $j \in \left\lbrace 0,1,2 \right\rbrace$, $\varphi^{\eps,\delta}_\pm$ obeys to \eqref{eq:wave_packet_equation_phi} with an error $\frac{1}{\eps}R^{\eps,\pm}$ defined according to \eqref{eq:error_wave_packet}; let us calculate it for $t \neq 0$:
\begin{eqnarray*}
\frac{1}{\eps} R_{t}^{\eps,\pm} \left(y\right) & = & \frac{1}{\eps} \left( -\left| \pm\frac{t^{2}}{2} + \sqrt{\eps}y \right| + \frac{t^{2}}{2} \pm \sqrt{\eps}y \right) \nonumber \\
& = & -\frac{2y}{\sqrt{\eps}} \left(\frac{ \pm \frac{t^{2}}{2} + \sqrt{\eps}y }{\left| \pm \frac{t^{2}}{2} + \sqrt{\eps}y \right| + \frac{t^{2}}{2} }\right) 1\!\!1_{\left\lbrace \pm y \leqslant -\frac{t^{2}}{2\sqrt{\eps}}\right\rbrace}, 
\end{eqnarray*}
which gives
\begin{equation}\label{eq:error_estimate_other}
\left| \frac{1}{\eps} R_{t}^{\eps,\pm} \left(y,t\right) \right| \leqslant 2\frac{\left|y\right|}{\sqrt{\eps}}1\!\!1_{\left\lbrace \pm y \leqslant -\frac{t^{2}}{2\sqrt{\eps}}\right\rbrace}.
\end{equation}

Additionally, one can solve equation \eqref{eq:profile_pm} for the profile $v^{\delta,\pm}$ of $\varphi^{\epsilon,\delta}_\pm$ explicitly; writing down its solution,
$$v^{\delta,\pm}_t \left(y\right) = e^{\frac{i}{2}t\Delta} \left(a\left(y\right) \chi^\delta_\pm \left(y\right)\right),$$
it remains clear that $v^{\delta,\pm}$ admits a finite development like
\begin{equation}\label{eq:finite_development}
v_{t}^{\delta,\pm}\left(y\right) = \left(1 + \frac{i}{2} t \Delta\right) \left( a\left(y\right)\chi_{\pm}^{\delta}\left(y\right) \right) + t^{2} \underbrace{\int_{0}^{1} \left(1-s\right)\partial_{t}^{2}v_{st}^{\delta,\pm}\left(y\right)ds}_{\tilde{v}_{t}^{\delta,\pm}\left(y\right)},
\end{equation}
where the first term in the right-hand side has support on $\pm y > 0$.

\begin{remark}\label{rem:boundedness}
From \eqref{eq:profile_pm} and the fact that its initial datum is $C_0^\infty(\R)$, it follows that, for any $T > 0$ and $j,k \in \mathbb{N}$, one has $y^j\partial_t^k v^{\delta,\pm} \in L^\infty\left([-T,T],L^2(\R)\right)$, which natually implies that $y^j\partial_t^k \tilde{v}^{\delta,\pm} \in L^\infty\left([-T,T],L^2(\R)\right)$.
\end{remark}

Therefore, from expressions \eqref{eq:error_estimate_other} and \eqref{eq:finite_development}:
{\small
\begin{eqnarray*}
\left\| \frac{1}{\eps} R_{t}^{\eps,\pm} \, v_{t}^{\delta,\pm}\right\|_{L^2(\R)} & \leqslant & \frac{2}{\sqrt{\eps}} \left\| y \, v_{t}^{\delta,\pm}\left(y\right) 1\!\!1_{\left\lbrace \pm y \leqslant -\frac{t^{2}}{2\sqrt{\eps}}\right\rbrace} \right\|_{L^2(\R)} \\
& = & \frac{2t^{2}}{\sqrt{\eps}} \left\| y \, \tilde{v}_{t}^{\delta,\pm}\left(y\right) 1\!\!1_{\left\lbrace \pm y \leqslant -\frac{t^{2}}{2\sqrt{\eps}}\right\rbrace} \right\|_{L^2(\R)},
\end{eqnarray*}}
which results in
{\small
\begin{equation}\label{eq:norm_difference_other}
\sup_{\pm t \in \left[0,T\right]} \left\| \Psi_{\pm,t}^{\eps,\delta} - \varphi_{\pm,t}^{\eps,\delta} \right\|_{L^2(\R)} \leqslant \pm \frac{2}{\sqrt{\eps}} \int_{0}^{\pm T} s^{2} \left\| y \, \tilde{v}_{s}^{\delta,\pm}\left(y\right) 1\!\!1_{\pm y \leqslant -\frac{s^{2}}{2\sqrt{\eps}}} \right\|_{L^2(\R)} ds
\end{equation}}
after applying Proposition \ref{prop:evolution_norm}. To evaluate this integral, fix $\alpha = \frac{1}{17}$ and denote $\tau_\eps = \eps^{\frac{1}{4}-\alpha}$:

\begin{enumerate}
\item For $t \in \left( -\tau_\eps , \tau_{\eps} \right)$, \eqref{eq:norm_difference_other} gives
\begin{eqnarray}\label{eq:norm_estimate_1_other}
\sup_{t \in \left(-\tau_\eps,\tau_\eps\right)} \left\| \Psi_{\pm,t}^{\eps,\delta} - \varphi_{\pm,t}^{\eps,\delta} \right\|_{L^2(\R)} & \leqslant & \frac{2}{3} \frac{\tau_{\eps}^{3}}{\sqrt{\eps}} \left\| y \, \tilde{v}_{t}^{\delta,\pm}\left(y\right) \right\|_{L^\infty\left((-\tau_\eps,\tau_\eps),L^2(\R)\right)} \nonumber \\
& = & \eps^{\frac{1}{4}-3\alpha} \, K_{\delta},
\end{eqnarray}
with $K_\delta > 0$ constant. The fact that $\left\| y \, \tilde{v}^{\delta,\pm}\right\|_{L^\infty\left((-\tau_\eps,\tau_\eps),L^2(\R)\right)}$ is bounded comes from Remark \ref{rem:boundedness}.

\item For the estimate for $t \in \Upsilon^\eps = \left[ -T , -\tau_{\eps} \right] \cup \left[ \tau_{\eps} , T \right]$, remark that $\pm y \leqslant -\frac{t^{2}}{2\sqrt{\eps}} \leqslant -\frac{\eps^{-2\alpha}}{2}$ implies $\left| y \right|^{-5} \leqslant 2^{5}\eps^{10\alpha}$, then:
$$\left\| y \, v_{t}^{\delta,\pm}\left(y\right) 1\!\!1_{\left\lbrace \pm y \leqslant -\frac{t^{2}}{2\sqrt{\eps}}\right\rbrace} \right\|_{L^2(\R)} \leqslant 2^{5} \eps^{10\alpha} \left\| y^{6} \, v_{t}^{\delta,\pm}\left(y\right) \right\|_{L^2(\R)}.$$
Boundedness for the norm in last equation's right-hand side comes from the previous Remark \ref{rem:boundedness}.  

Finally, form \eqref{eq:norm_difference_other}, 
\begin{eqnarray}\label{eq:norm_estimate_2_other}
\sup_{t \in \left(-\tau_\eps,\tau_\eps\right)} \left\| \Psi_{\pm,t}^{\eps,\delta} - \varphi_{\pm,t}^{\eps,\delta} \right\|_{L^2(\R)} & \leqslant & \frac{2^{6} \, T^3}{3} \eps^{10\alpha - \frac{1}{2}} \left\| y^{6} \, v^{\delta,\pm} \right\|_{L^\infty\left(\Upsilon^\eps,L^2(\R)\right)} \nonumber \\
& = & \eps^{10\alpha-\frac{1}{2}} \tilde{K}_\delta
\end{eqnarray}
for a constant $\tilde{K}_\delta > 0$.
\end{enumerate}

Because of our choice of $\alpha = \frac{1}{17}$, both estimates \eqref{eq:norm_estimate_1_other} and \eqref{eq:norm_estimate_2_other} go to $0$ with $\eps \longrightarrow 0$, and we prove the proposition.
\end{proof}

So, in this section we saw the example of a case where the initial measure splits in two pieces, each one gliding to its side as in Figure \ref{fig:rebounding_particle}, accordingly to the quantum distribution of mass along the $x$-axis.

Remark that there is no crossings at all. The part $\mu^-$ that goes to the left downwards did not come from the same side as $\mu^+$, but from the very left. The portion that was in the right for $t < 0$ stays in the right all the time.

\subsection{Measures crossing the singularity}\label{sec:other_parabola}
Now, consider for $\eta \leqslant 0$ the Hamiltonian trajectories of $V(x) = -|x|$:
\begin{equation}\label{eq:trajectory_wave_packet}
\left\lbrace
\begin{array}{l}
\xi^\eta\left(t\right) = \eta \pm t \\
x^\eta\left(t\right) = \eta t \pm \frac{t^2}{2}
\end{array}
\right. \quad \rm{for} \quad \pm t \leqslant 0.
\end{equation} 
In this section, we will prove:

\begin{proposition}\label{prop:crossing_psi}
If $\Psi^{\eps,\eta}$ is solution to the Schrödinger equation \eqref{eq:schrodinger} with $V(x) = -|x|$ in $\R$, $\eta < 0$ and initial data
$$\Psi_0^{\eps,\eta}(x) = \frac{1}{\eps^{\frac{1}{4}}} a\left(\frac{x}{\sqrt{\eps}}\right) e^{\frac{i}{\eps}\eta \centerdot x},$$
with $a \in C_0^\infty(\R)$, $\| a \|_{L^2(\R)} = 1$, then the associated semiclassical measures are
$$\mu^\eta_t(x,\xi) = \delta (x-x^\eta(t)) \otimes \delta(\xi-\xi^\eta(t))$$
for any $t \in \R$.

Besides, if we take $\eta = -\eps^\beta$ with $0 < \beta < \frac{1}{10}$, then the corresponding semiclassical measure will be, for $t \in \R$,
$$\mu_t(x,\xi) = \delta (x-x^0(t)) \otimes \delta(\xi-\xi^0(t)).$$
\end{proposition}

\begin{proof}
For the case with $\eta < 0$ constant, just apply Lemmata \ref{lem:psi_follows_varphi} and \ref{lem:psi_follows_the_path} ahead choosing $\beta = 0$. For the case $\eta = -\eps^\beta$, same thing, but of course taking $0 < \beta < \frac{1}{10}$.
\end{proof}

Before we proceed to the lemmata, let us define $u^{\eps,\eta}$ after \eqref{eq:definition_u} using the trajectories in \eqref{eq:trajectory_wave_packet}, so as $u^{\eps,\eta}$ satisfies the following system:
\begin{equation*}
\left\lbrace
\begin{array}{l}
i\partial_t u^{\eps,\eta}_t\left(y\right) = -\frac{1}{2}\Delta u^{\eps,\eta}_t\left(y\right) + \frac{1}{\eps}R^{\eps,\eta}_t\left(y\right) \, u^{\eps,\eta}_t\left(y\right) \\
u^{\eps,\eta}_0\left(y\right) = a\left(y\right).
\end{array}
\right. 
\end{equation*}
We have:
\begin{lemma}
Above, for $t \neq 0$ we have
\begin{equation}\label{eq:r_nice_form}
R^{\eps,\eta}_t\left(y\right) = 2 \left( x^\eta(t) + \sqrt{\eps} y \right) \left( 1\!\!1_{\left\{t < 0\right\}} 1\!\!1_{\left\{ y < -\frac{x^\eta(t)}{\sqrt{\eps}} \right\}} - 1\!\!1_{\left\{t > 0\right\}} 1\!\!1_{\left\{y > -\frac{x^\eta(t)}{\sqrt{\eps}}\right\}} \right)
\end{equation}
and
\begin{equation}\label{eq:r_nice_estimation}
\left| \frac{1}{\eps} R^{\eps,\eta}_t (y) \right| \leqslant \frac{2 |y|}{\sqrt{\eps}} \left( 1\!\!1_{\left\{t < 0\right\}} 1\!\!1_{\left\{y < -\frac{x^\eta(t)}{\sqrt{\eps}}\right\}} + 1\!\!1_{\left\{t > 0\right\}} 1\!\!1_{\left\{y > -\frac{x^\eta(t)}{\sqrt{\eps}}\right\}} \right).
\end{equation}
\end{lemma}

\begin{proof}
Write down
\begin{eqnarray*}
\frac{1}{\eps} R^{\eps,\eta}_t (y) & = & -\frac{1}{\eps}\left( \left| x^\eta(t) + \sqrt{\eps} y \right| - | x^\eta(t) | - {\rm sign} \left(x^\eta(t)\right) \sqrt{\eps} y \right) \nonumber \\
& = & -\frac{y}{\sqrt{\eps}} \left( \frac{2 \, x^\eta(t) + \sqrt{\eps} y}{\left| x^\eta(t) + \sqrt{\eps}y \right| + |x^\eta(t)|} - {\rm sign} \left(x^\eta(t)\right) \right)
\end{eqnarray*}
and observe that ${\rm sign} \left(x^\eta(t)\right) = - {\rm sign} (t)$ for the trajectory in \eqref{eq:trajectory_wave_packet}.
\end{proof}

Now, define 
\begin{equation}\label{eq:v_tilde}
\tilde{v}^{\eps,\eta}_t \left(y\right) = e^{-\frac{i}{\eps}\int_0^t R^{\eps,\eta}_s \left(y\right) ds} \, a\left(y\right)
\end{equation}
and, given some small $\tau_\eps \in \left(0,T\right)$, consider the following wave packet profile equation linked to the trajectory \eqref{eq:trajectory_wave_packet}:
\begin{equation}\label{eq:v_afterwards}
\left\lbrace
\begin{array}{l}
i\partial_t v^{\eps,\eta}_t\left(y\right) = -\frac{1}{2}\Delta v^{\eps,\eta}_t\left(y\right) \\
v^{\eps,\eta}_{\pm\tau_\eps}\left(y\right) = \tilde{v}^{\eps,\eta}_{\pm\tau_\eps} \left(y\right)
\end{array} \quad {\rm for} \quad t \in \Upsilon^\eps = \left[-T,-\tau_\eps\right] \cup \left[\tau_\eps,T\right].
\right.
\end{equation}

\begin{definition}
We will call $\varphi^{\eps,\eta}$ the wave packet defined as in \eqref{eq:wave_packet} for trajectory \eqref{eq:trajectory_wave_packet}, having profile $\tilde{v}^{\eps,\eta}$ for $t \in \left[-\tau_\eps,\tau_\eps\right]$ and profile $v^{\eps,\eta}$ otherwise.
\end{definition}

\begin{lemma}\label{lem:psi_follows_varphi}
For $\eta_\eps = \eta \eps^\beta$ with $\eta < 0$ and $0 \leqslant \beta < \frac{1}{10}$, one has 
$$\lim_{\eps \longrightarrow 0} \left\| \Psi^{\eps,\eta_\eps} - \varphi^{\eps,\eta_\eps} \right\|_{L^\infty\left(\left[-T,T\right], L^2\left(\R\right)\right)} = 0.$$
\end{lemma} 

\begin{proof}
Let us treat the problem partitioning it in zones by choosing $\tau_\eps = \eps^\alpha$ with $\alpha > \beta$:

\begin{enumerate}
\item $t \in \left[-\tau_\eps,\tau_\eps\right]$:

Denote\footnote{Since now $\eta$ depends on $\eps$, we will drop down the dependencies on $\eta$ in order not to overcharge the notation. We will also let the dependency of the trajectories on $\eps$ implicit until it be crucial to take it into account.} $z^\eps = u^\eps - \tilde{v}^\eps$; then $z^\eps_0 \left(y\right) = 0$,
$$i\partial_t z^\eps_t(y) + \frac{1}{2}\Delta z^\eps_t(y) - \frac{1}{\eps}R^\eps_t\left(y\right) z^\eps_t(y)= -\frac{1}{2}\Delta \tilde{v}^\eps_t(y)$$
and consequently
\begin{eqnarray*}
\frac{d}{dt} \left\| \, z^\eps_t \, \right\|^2_{L^2\left(\R\right)} & = & -2 \, {\rm Im} \left< \, z^\eps_t  \, , \, \Delta \tilde{v}^\eps_t \, \right>_{L^2\left(\R\right)} \\
& \leqslant & 2 \left| \left<  \, \nabla u^\eps_t \, , \, \nabla \tilde{v}^\eps_t \, \right>_{L^2\left(\R\right)} \right| \\
& \leqslant & \left\| \, \nabla u^\eps_t \, \right\|_{L^2\left(\R\right)} \left\| \, \nabla \tilde{v}^\eps_t \, \right\|_{L^2\left(\R\right)}.
\end{eqnarray*}
Given that $\nabla u^\eps$ satisfies 
$$i\partial_t \left(\nabla u^\eps_t(y)\right) + \frac{1}{2}\Delta \left(\nabla u^\eps_t(y)\right) - \frac{1}{\eps}R^\eps_t\left(y\right) \left(\nabla u^\eps_t(y)\right) = \frac{1}{\eps}\nabla R^\eps_t\left(y\right) u^\eps_t(y),$$
one can estimate, in the same way as in Proposition \ref{prop:evolution_norm} (which enables us to avoid calculating $R^\eps_t$ at t = 0),
\begin{eqnarray*}
\left\| \, \nabla u^\eps_t \,\right\|_{L^2(\R)} & \leqslant & \left\| \, \nabla a \, \right\|_{L^2(\R)} + |t| \left\| \, \frac{1}{\eps}\nabla R^\eps \, u^\eps \, \right\|_{L^\infty\left(\,]0,t[\,,L^2(\R)\right)} \\
& \leqslant & \left\| \, \nabla a \, \right\|_{L^2(\R)} + \frac{2|t|}{\sqrt{\eps}},
\end{eqnarray*}
the last line coming from quantum normalization and the facts that $\left\| u^\eps_t \right\|_{L^2(\R)} = 1$ is constant ($\left\| u^\eps_t \right\|_{L^2(\R)} = \left\| \Psi^\eps_t \right\|_{L^2(\R)}$) and
$$\frac{1}{\eps}\left| \nabla R^\eps_t \left(y\right) \right| = \frac{1}{\sqrt{\eps}} \left| \frac{x(t) + \sqrt{\eps}y}{\left| x(t) + \sqrt{\eps}y \right|} - {\rm sign} \left(x(t)\right) \right| \quad {\rm for} \quad t \neq 0,$$
so $\left\| \frac{1}{\eps} R^\eps_t \right\|_{L^\infty(\R)} \leqslant \frac{2}{\sqrt{\eps}}$ for $t \in \; ]\, 0 \, ,t \,[ \,$.

\medskip

As a step aside, notice the following: 
$$\nabla \tilde{v}^\eps_t \left(y\right) = -i\tilde{v}^\eps_t \left(y\right) \nabla \underbrace{ \left(\frac{1}{\eps}\int_0^t R^\eps_s \left(y\right)ds \right)}_{I^\eps(t,y)} + e^{-\frac{i}{\eps}\int_{0}^{t} R^\eps_s(y)ds} \, \nabla a(y).$$
Let us resume the main reasoning.

\medskip

Taking into account the domain restrictions of $R^\eps$ (see \eqref{eq:r_nice_form}), for $\pm t < 0$ one must have satisfied the inequalities $\pm y < 0$ and $t^2 \pm 2\, \eta_\eps \, t \pm 2\sqrt{\eps}y < 0$ in order not to have $R^\eps$ null, which means that, for fixed $y$, $t$ is comprised in $\left[ -\varsigma(y) , \varsigma(y) \right]$, where 
\begin{equation}\label{eq:key_inequality}
\varsigma(y) = \eta_\eps + \sqrt{\eta_\eps^2 + 2\sqrt{\eps}|y|} > 0
\end{equation}
is one of the roots of $\varsigma^2 - 2\, \eta_\eps \, \varsigma -2\sqrt{\eps}y = 0$. Further, remark that $R^\eps_{\pm \varsigma(y)} \left(y\right) = 0$; as a consequence,
\begin{equation}\label{eq:def_i}
I^\eps(t,y) = \left\lbrace
\begin{array}{lll}
\frac{1}{\eps}\int_0^t R^\eps_s\left(y\right) ds & \quad {\rm if} & |t| < \varsigma(y) \\
\frac{1}{\eps}\int_0^{\pm \varsigma(y)} R^\eps_s\left(y\right) ds & \quad {\rm if} & \pm t \geqslant \varsigma(y),
\end{array}\right.
\end{equation}
which implies
\begin{equation}\label{eq:derivative_integral}
\partial_y I^\eps(t,y) = \left\lbrace
\begin{array}{lll}
\frac{2}{\sqrt{\eps}} t & {\rm if} & |t| < \varsigma(y) \\
\frac{2}{\sqrt{\eps}}\varsigma(y) & {\rm if} & t \geqslant \varsigma(y) \\
-\frac{2}{\sqrt{\eps}}\varsigma(y) & {\rm if} & t \leqslant - \varsigma(y).
\end{array}\right.
\end{equation}

Now, considering that $2\beta < \frac{1}{2}$, that $\eta_\eps < 0$ and that there is $K > 0$ such that $|y| < K$, since within $\nabla \tilde{v}^\eps$ we still have multiplying factors $a$ and $\nabla a$ that are compactly supported in $y$, it can be made the estimative:
$$\varsigma(y) = \eta_\eps + | \eta_\eps | \left( 1 + \frac{2\sqrt{\eps} | y |}{\eta_\eps^2} \right)^{\frac{1}{2}} \leqslant \frac{\sqrt{\eps}}{| \eta_\eps |}K,$$
which in any case gives $\left| \partial_y I^\eps(t,y) \right| \lesssim \frac{C_1}{|\eta_\eps|}$ for some constant $C_1 > 0$, hence
$$\left\| \, \nabla \tilde{v}^\eps_t \, \right\|_{L^2(\R)} \lesssim \frac{C_1}{|\eta_\eps|} \left\| \, a \, \right\|_{L^2(\R)} = \frac{C_1}{|\eta_\eps|}.$$

Finally, this results is a superior bound for $\left\| u^\eps_t - \tilde{v}^\eps_t \right\|_{L^2\left(\R\right)}$ in $\left[-\tau_\eps,\tau_\eps\right]$; so, for a constant $C_2 > 0$:
\begin{equation}\label{eq:superior_bound_tau}
\left\| \, z^\eps \, \right\|_{L^\infty\left(\left[-\tau_\eps,\tau_\eps\right],L^2(\R)\right)}^2 \lesssim C_2 \frac{\tau_\eps}{|\eta_\eps|} + C_1 \frac{\tau_\eps^2}{\sqrt{\eps} |\eta_\eps|}
\end{equation}
As $\alpha >  \beta$ by assumption, the only additional constraint we need in order to have the bound above small when $\eps \longrightarrow 0$ is: 
\begin{equation}\label{eq:constraint_1}
2\alpha - \beta -\frac{1}{2} > 0.
\end{equation}

\item $t \in \Upsilon^\eps = \left[-T,-\tau_\eps\right] \cup \left[\tau_\eps,T\right]$:

Hereafter denote $z^\eps = u^\eps - v^\eps$. Now $z^\eps$ obeys to the equation
$$i\partial_t z^\eps_t(y) + \left(\frac{1}{2}\Delta - \frac{1}{\eps}R^\eps_t \left(y\right)\right) z^\eps_t(y) = \frac{1}{\eps} R^\eps_t \left(y\right) v^\eps_t(y)$$
and, therefore, 
$$\left\| \, z^\eps_t \, \right\|_{L^2(\R)} \leqslant \left\| \, z^\eps_{\pm \tau_\eps} \, \right\|_{L^2(\R)} \pm \int_{\pm \tau_\eps}^t \left\| \, \frac{1}{\eps}R^\eps_s \, v^\eps_s \, \right\|_{L^2(\R)} ds$$
according to $t$ being positive or negative.

Recalling the trajectory defined in \eqref{eq:trajectory_wave_packet}, the estimation in \eqref{eq:r_nice_estimation} and the fact that $\eta_\eps < 0$, one has that $R^\eps$ is non-zero only in the region $|y| > \frac{|\eta_\eps t|}{\sqrt{\eps}} \geqslant \frac{|\eta_\eps|\tau_\eps}{\sqrt{\eps}}$, so $\frac{1}{|y|} < \frac{\sqrt{\eps}}{|\eta_\eps| \tau_\eps}$; this gives
\begin{eqnarray*}
\left\| \, \frac{1}{\eps} R^\eps_t \, v^\eps_t \, \right\|_{L^2(\R)} & \leqslant & \frac{2}{\sqrt{\eps}} \left\| \, y \, v^\eps_t \, \right\|_{L^2(\R)} \\
& \leqslant & \frac{2}{|\eta_\eps| \tau_\eps} \left(\frac{\sqrt{\eps}}{|\eta_\eps| \tau_\eps}\right)^k \left\| \, y^{k+2} \, v^\eps_t \, \right\|_{L^2(\R)}.
\end{eqnarray*}

\begin{lemma}\label{lem:norms}
For $t \in \Upsilon^\eps$, $n,m \in \mathbb{N}_0$ and $\beta < \frac{1}{4}$, there exists $K_{n+m} > 0$ constant such that $\left\| y^n \nabla^m v^\eps_t \right\|_{L^2(\mathbb{\R})} \leqslant \frac{K_{n+m}}{|\eta_\eps|^{n+m}}$.
\end{lemma}
\noindent (The proof is postponed.)

\medskip

As a conclusion, for $\eps$ small enough we have
$$\left\| \, \frac{1}{\eps} R^\eps_t \, v^\eps_t \, \right\|_{L^2(\R)} \leqslant \frac{2K_n}{|\eta_\eps|^3 \tau_\eps} \left(\frac{\sqrt{\eps}}{|\eta_\eps|^2 \tau_\eps}\right)^k,$$ which carries the new constraint:
\begin{equation}\label{eq:constraint_2}
\frac{k}{2} - (k+1) \alpha - (2k+3) \beta > 0
\end{equation}
for some $k \in \mathbb{N}$.
\end{enumerate}

The proposition is proven once we remark that for any $0 \leqslant \beta < \frac{1}{10}$, one can find a positive integer $k$ such that both \eqref{eq:constraint_1} and \eqref{eq:constraint_2} will be satisfied for $\alpha > \beta$.
\end{proof}

\begin{proof}[Proof of Lemma \ref{lem:norms}]
To evaluate $\left\| y^{n} v^\eps_t \right\|_{L^2(\R)}$, observe that by recurrence one can show that, for $n \in \mathbb{N}_0$,
$$i\partial_t \left( y^n v^\eps \right) + \frac{1}{2}\Delta \left( y^n v^\eps \right) = \frac{1}{2}n(n-1) y^{n-2} \, v^\eps + n \, y^{n-1} \, \nabla v^\eps$$
and, since $\nabla^m v^\eps$ satisfies the same equation \eqref{eq:v_afterwards} as $v^\eps$,
\begin{equation}\label{eq:v_dv_afterwards}
i\partial_t \left( y^n \nabla^m v^\eps \right) + \frac{1}{2}\Delta \left( y^n \nabla^m v^\eps \right) = \frac{1}{2}n(n-1) y^{n-2} \, \nabla^m v^\eps + n \, y^{n-1} \, \nabla^{m+1} v^\eps,
\end{equation}
from where we have the estimation:
{\footnotesize
\begin{eqnarray*}
\left\| y^n \nabla^m v^\eps_t \right\|_{L^2(\R)} \leqslant \left\| y^n \nabla^m v^\eps_{\pm \tau_\eps} \right\|_{L^2(\R)} + \frac{1}{2}n(n-1) T \left\| y^{n-2} \nabla^m v^\eps \right\|_{L^\infty\left(\Upsilon^\eps,L^2(\R)\right)} \\ + \, nT \left\| y^{n-1} \nabla^{m+1} v^\eps \right\|_{L^\infty\left(\Upsilon^\eps,L^2(\R)\right)}.
\end{eqnarray*}}

The trick will be to transform the $L^\infty\left(\Upsilon^\eps,L^2(\R)\right)$ norms of the terms with $\nabla^m v^\eps$ into $L^2(\R)$ ones, so observe that we have
{\footnotesize
\begin{eqnarray}\label{eq:chega}
\left\| y^{n-2} \nabla^m v^\eps_t \right\|_{L^2(\R)} \leqslant \left\| y^{n-2} \nabla^m v^\eps_{\pm \tau_\eps} \right\|_{L^2(\R)} + \frac{1}{2}(n-2)(n-3) T \left\| y^{n-4} \nabla^m v^\eps \right\|_{L^\infty\left(\Upsilon^\eps,L^2(\R)\right)} \nonumber \\ + \, (n-2) T \left\| y^{n-3} \nabla^{m+1} v^\eps \right\|_{L^\infty\left(\Upsilon^\eps,L^2(\R)\right)}
\end{eqnarray}}
and, of course, that the right-hand side above also bounds $\left\| y^{n-2} \nabla^m v^\eps \right\|_{L^\infty\left(\Upsilon^\eps,L^2(\R)\right)}$. 

Repeating the steps above for the term $\left\| y^{n-4} \nabla^m v^\eps \right\|_{L^\infty\left(\Upsilon^\eps,L^2(\R)\right)}$ that appears in \eqref{eq:chega}, we will obtain an expression with the $L^2(\R)$ norm $\left\| y^{n-4} \nabla^m v^\eps_{\pm \tau_\eps} \right\|_{L^2(\R)}$ (as wished) and, additionally, the terms $\left\| y^{n-6} \nabla^m v^\eps \right\|_{L^\infty\left(\Upsilon^\eps,L^2(\R)\right)}$ and $\left\| y^{n-5} \nabla^{m+1} v^\eps \right\|_{L^\infty\left(\Upsilon^\eps,L^2(\R)\right)}$. Well, then we just repeat the same procedure for $\left\| y^{n-6} \nabla^m v^\eps \right\|_{L^\infty\left(\Upsilon^\eps,L^2(\R)\right)}$, then for the term $\left\| y^{n-8} \nabla^m v^\eps \right\|_{L^\infty\left(\Upsilon^\eps,L^2(\R)\right)}$ that will appear, etc... and what we get is essentially
{\tiny
\begin{equation}\label{eq:estimate_vixi}
\left\| y^n \nabla^m v^\eps_t \right\|_{L^2(\R)} \leqslant \sum_{j =0}^{\left\lfloor \frac{n}{2} \right\rfloor} \left( c^{(1)}_{n,j} \left\| y^{n-2j} \nabla^m v^\eps_{\pm \tau_\eps} \right\|_{L^2(\R)} + c^{(2)}_{n,j} \left\| y^{n-1-2j} \nabla^{m+1} v^\eps \right\|_{L^\infty\left(\Upsilon^\eps,L^2(\R)\right)} \right),
\end{equation}}
with $c^{(1)}_{n,j}$ and $c^{(2)}_{n,j}$ appropriate coefficients. 

Two things are remarkable in this formula. The first one is that all terms $y^{n-2j} \nabla^m v^\eps_{\pm \tau_\eps}$ have the same support (recall their definition, in \eqref{eq:v_tilde} and \eqref{eq:v_afterwards}), which is the compact support of $a$. This bounds $| y |$ uniformly with respect to $n$, $m$, $j$ and $\eps$, implying that
$$\sum_{j =0}^{\left\lfloor \frac{n}{2} \right\rfloor} c^{(1)}_{n,j} \left\| y^{n-2j} \nabla^m v^\eps_{\pm \tau_\eps} \right\|_{L^2(\R)} = d_{n,m} \left\| \nabla^m v^\eps_{\tau_\eps} \right\|_{L^2(\R)},$$
where, again, $d_{n,m}$ is a suitable coefficient not depending on $\eps$.

The second remarkable thing is that among the terms within the $L^\infty\left(\Upsilon^\eps,L^2(\R)\right)$ norms, the highest power of $y$ that we find is $n-1$, and no more $n$, as in the beginning. This suggests that we may do the very same analysis for estimating each term $\left\| y^{n-1-2j} \nabla^{m+1} \right\|_{L^\infty\left(\Upsilon^\eps,L^2(\R)\right)}$ in \eqref{eq:estimate_vixi} and obtain estimates like
{\footnotesize
\begin{eqnarray*}
\left\| y^{n-1-2j} \nabla^{m+1} v^\eps \right\|_{L^\infty\left(\Upsilon^\eps,L^2(\R)\right)} & \leqslant & d_{n-1-2j,m+1} \left\| \nabla^{m+1} v^\eps_{\pm \tau_\eps} \right\|_{L^2(\R)} \\
& & + \sum_{l = 0}^{\left\lfloor \frac{1}{2}\left(n-1-2j\right) \right\rfloor} \left\| y^{n-2-2(j+l)} \nabla^{m+2} v^\eps \right\|_{L^\infty\left(\Upsilon^\eps,L^2(\R)\right)};
\end{eqnarray*}}
again, the maximal power of $y$ to appear inside the $L^\infty\left(\Upsilon^\eps,L^2(\R)\right)$ norms has been reduced by $1$ with respect to the norm being estimated in the left-hand side. Whence, running recursively until we bring the maximal exponent down to $0$, we will end up with: 
$$\left\| y^n \nabla^m v^\eps_t \right\|_{L^2(\R)} \leqslant \sum_{j = 0}^{n-1} \tilde{d}_j \left\| \nabla^{m+j} v^\eps_{\pm \tau_\eps} \right\|_{L^2(\R)} + \tilde{d}_n \left\| \nabla^{m+n} v^\eps \right\|_{L^\infty\left(\Upsilon^\eps,L^2(\R)\right)},$$
with $\eps$-independent coefficients $\tilde{d}_j$. Finally, from equation \ref{eq:v_dv_afterwards}, one knows that the norm $\left\| \nabla^{n+m} v^\eps_t \right\|_{L^2(\R)}$ is constant in time, so we can simplify even more the last estimation and have got:
\begin{equation}\label{eq:another_inequality}
\left\| y^n \nabla^m v^\eps_t \right\|_{L^2(\R)} \leqslant \sum_{j = 0}^{n} \tilde{d}_{j} \left\| \nabla^{m+j} v^\eps_{\pm \tau_\eps} \right\|_{L^2(\R)}.
\end{equation}

Making use of \eqref{eq:v_tilde} and the initial condition \eqref{eq:v_afterwards}, let us calculate the remaining quantities:
{\footnotesize
$$\nabla^{m} v^\eps_{\pm \tau_\eps} (y) = e^{-\frac{i}{\eps}\int_0^{\pm \tau_\eps} R^\eps_s (y) ds} \, \sum_{l = 0}^{m} \left(
\begin{array}{c}
m \\ l 
\end{array}\right)
\nabla^{m - l} a(y) \sum_{\underset{\sum_{s=1}^{m} s \sigma_s = l}{\sigma \in \mathbb{N}^m_0}} \left( c^\sigma \prod_{j = 1}^{m} \left( \partial_y^{j} I^\eps(\pm \tau_\eps,y) \right)^{\sigma_j}\right),$$}
where $c^\sigma$ are complex coefficients.

The way for calculating the expression above is the following: if condition \eqref{eq:constraint_2} is fulfilled, then we have $\alpha < \frac{1}{2} - \beta$, which causes $\tau_\eps$ to be always greater than $|\varsigma(y) | \sim \frac{\sqrt{\eps}}{|\eta_\eps|}$. Then, using \eqref{eq:derivative_integral}, we get $\partial_y^{j+1} I^\eps(\pm \tau_\eps,y) = \pm \frac{2}{\sqrt{\eps}} \partial_y^{j} \varsigma(y)$ for $j \in \mathbb{N}_0$, and, using \eqref{eq:key_inequality} and being $\alpha > 2\beta$ (from \eqref{eq:constraint_1} and the fact that $0 \leqslant \beta < \frac{1}{10}$):
$$\frac{1}{\sqrt{\eps}} \partial_y^{j+1} \varsigma(y) \sim \left( \frac{\sqrt{\eps}}{|\eta_\eps|} \right)^{j} \frac{1}{|\eta_\eps|^{j+1}};$$
thus, if we do the brutal majoration, $\frac{\sqrt{\eps}}{|\eta_\eps|} \lesssim 1$, one gets $\left| \partial_y^{j+2} I^\eps(\pm \tau_\eps,y) \right| \lesssim \frac{1}{|\eta_\eps|^j}$ for all $j \in \N_0$, and also $I^\eps(\pm \tau_\eps,y) \sim \varsigma(y) \lesssim 1$. Additionally, we already had $\partial_y I^\eps(\pm \tau_\eps,y) \sim \frac{1}{|\eta_\eps|}$, so even in the worst case one can always have the estimate
$$\partial_y^j I(\pm\tau^\eps,a) \lesssim \frac{1}{|\eta_\eps|^{j}},$$
which is, of course, far from optimal if $j \geqslant 2$ and just bad if $j = 0$, but fits in our purposes.

It follows that, for $\sigma$ such that $\sum_{s = 1}^{m} j\sigma_j = l$ and conveniently chosen constants $K_j$, 
$$\prod_{j=1}^{m} \left( \partial_y^j I(\pm \tau_\eps,y) \right)^{\sigma_j} \leqslant \frac{1}{|\eta_\eps|^l} \prod_{j = 1}^{m}K_j,$$
so $\left\| \nabla^m v^\eps_{\pm \tau_\eps} \right\|_{L^2(\R)}$ will be dominated by a term of order $\frac{1}{|\eta_\eps|^m}$ and, finally, inequality \eqref{eq:another_inequality} will be bounded by a term of order $\frac{1}{|\eta_\eps|^{n+m}}$, what we wanted to show.
\end{proof}
This completes the proposition's proof.

\begin{lemma}\label{lem:psi_follows_the_path}
With $\eta_\eps= \eta\eps^\beta$, $\eta < 0$, the semiclassical measure associated with the family $\left(\varphi^{\eps,\eta_\eps}\right)_{\eps > 0}$ is transported by a trajectory of shape \eqref{eq:trajectory_wave_packet}: $\left(x^0(t),\xi^0(t)\right)$ if $0 < \beta < \frac{1}{10}$, and $\left(x^\eta(t),\xi^\eta(t)\right)$ if $\beta = 0$.
\end{lemma}

\begin{proof}
The fact that $\varphi^{\eps,\eta_\eps}$ concentrates to a measure that follows the aimed path is not guaranteed  by Lemma \ref{lem:path_wave_packet} since the initial data we inserted in the wave packet equation \eqref{eq:v_afterwards} is not $\eps$-independent, as we required in Section \ref{sec:the_wave_packets}. Let us then calculate the concentration of $\varphi^{\eps,\eta_\eps}$ indirectly.

To begin with, if conditions \eqref{eq:constraint_1} and \eqref{eq:constraint_2} are fulfilled, then $\tau_\eps > \varsigma(y)$ and consequently, from \eqref{eq:v_tilde}, \eqref{eq:key_inequality} and \eqref{eq:def_i}: 
$$\tilde{v}^{\eps,\eta_\eps}_{\pm \tau_\eps}\left(y\right) = e^{-\frac{i}{\eps}\int_0^{\pm \varsigma(y)} R^\eps_s (y) ds}a(y) = e^{-i\sqrt{\eps}\left|\frac{y}{\eta_\eps}\right|^3} e^{-i\left( \left(\frac{y}{\eta_\eps}\right)^{2} \mp \left| \frac{y}{\eta_\eps} \right|  \right)}a(y),$$
thus, setting $\hat{v}^\eps_{\pm,0} (x) = \frac{1}{\eps^{\frac{1}{4}}}\tilde{v}^{\eps,\eta_\eps}_{\pm \tau_\eps}\left(\frac{x}{\sqrt{\eps}}\right)$, one has
$$\hat{v}^\eps_{\pm, 0} (x) = \frac{1}{\eps^{\frac{1}{4}}}a\left(\frac{x}{\sqrt{\eps}}\right) f_1\left(\frac{x}{\eta_\eps \sqrt{\eps}}\right) f_2\left(\frac{\eps^\frac{1}{6}x}{\eta_\eps \sqrt{\eps}}\right),$$
where $|f_1| = |f_2| = 1$ and $\nabla f_1$ and $\nabla f_2$ exist and are locally bounded almost everywhere in $\R$. These facts and standard symbolic calculus allow a straightforward calculation showing that $\hat{v}^\eps_{\pm,0}$ concentrates to the measure $\mu_0^{\hat{v}^\eps_{\pm}}(x,\xi) = \delta\left(x \right) \otimes \delta\left(\xi \right)$.

Now, define $\hat{v}^{\eps}_\pm$ as the functions in $L^\infty(\R,L^2(\R))$ that satisfy the systems
\begin{equation*}
\left\lbrace
\begin{array}{l}
i\eps \partial_t \hat{v}^\eps_{\pm,t} (x) = -\frac{\eps^2}{2} \Delta \hat{v}^\eps_{\pm,t} (x) \\
\hat{v}^\eps_{\pm,0} (x) = \frac{1}{\eps^{\frac{1}{4}}}\tilde{v}^{\eps,\eta_\eps}_{\pm \tau_\eps}\left(\frac{x}{\sqrt{\eps}}\right);
\end{array}\right.
\end{equation*}
it is possible to affirm that the semiclassical measures of $\hat{v}^\eps_\pm$ will by carried by the flow $\Phi_t\left(x,\xi\right) = \left(x + t\, \xi,\xi\right)$, since by standard results (see the Introduction) they should obey to the usual Liouville equation \eqref{eq:liouville_t_by_t} with a null potential. But because initially they are concentrated to the point $(0,0)$ in the phase space, we get $\mu^{\hat{v}^\eps_\pm}_t (x,\xi) = \delta(x)\otimes \delta(\xi)$ for all $t\in \R$.

Well, for $\pm t \geqslant 0$, $\hat{v}^\eps_{\pm,t} (x) = \frac{1}{\eps^{\frac{1}{4}}}v^{\eps,\eta_\eps}_{t \pm \tau_\eps}\left(\frac{x}{\sqrt{\eps}}\right)$, so observe that, for $\pm t \in [\tau_\eps,T]$,
$$\varphi^{\eps,\eta_\eps}_t\left(x\right) = \hat{v}^\eps_{\pm, t\mp \tau_\eps}\left(x-x^{\eta_\eps}(t)\right)e^{\frac{i}{\eps}\left[\xi^{\eta_\eps}(t) \centerdot \left(x-x^{\eta_\eps}(t)\right) + S^{\eta_\eps}(t) \right]};$$
consequently, by picking up a $b \in C^\infty_0 \left(\R^2\right)$, one gets
{\small
\begin{eqnarray*}
\left< \op_\eps(b)\, \varphi^{\eps,\eta_\eps}_t , \varphi^{\eps,\eta_\eps}_t \right>_{L^2(\R)} & = & \left< \op_\eps\left(b\left(x+x^{\eta_\eps}(t),\xi+\xi^{\eta_\eps}(t)\right)\right) \hat{v}^\eps_{\pm, t \mp \tau_\eps} , \hat{v}^\eps_{\pm, t\mp\tau_\eps} \right>_{L^2(\R)} \\
& = & \left< \op_\eps\left(b\left(x+x^{\eta_\eps}(t),\xi+\xi^{\eta_\eps}(t)\right)\right) \hat{v}^\eps_{\pm,t} , \hat{v}^\eps_{\pm,t} \right>_{L^2(\R)} + \mathcal{O}\left(\eps^{\alpha-2\beta}\right),
\end{eqnarray*}}
with the error coming from
$$\left\| \hat{v}^\eps_{\pm,t} - \hat{v}^\eps_{\pm,t\mp\tau_\eps}\right\|_{L^2(\R)} \leqslant \frac{\tau_\eps}{2} \left\| \Delta v^\eps_{\pm\tau_\eps} \right\|_{L^2(\R)} \leqslant K \frac{\tau_\eps}{|\eta_\eps|^2} = K\eps^{\alpha-2\beta},$$
where $K>0$ is constant and we used Lemma \ref{lem:norms}. If $\beta = 0$, then $\eta_\eps = \eta$ is constant and we get, for $t \in  [-T,T] \setminus \lbrace 0 \rbrace$, 
$${\rm sc}\lim \left< \op_\eps(b)\, \varphi^{\eps,\eta}_t , \varphi^{\eps,\eta}_t \right>_{L^2(\R)} = \left< \delta(x) \otimes \delta(\xi) , b(x+x^\eta(t),\xi+\xi^\eta(t)) \right>_{\R^2},$$
which also holds for $t=0$ due to the initial condition for $\varphi^{\eps,\eta}$, implying that
$$\mu^{\varphi^{\eps,\eta}}_t(x,\xi) = \delta(x-x^\eta (t)) \otimes \delta(\xi-\xi^\eta (t))$$
for all $t\in[-T,T]$.  

If $0 < \beta < \frac{1}{10}$, we can still have
{\small
\begin{equation*}
\left< \op_\eps(b)\, \varphi^\eps_t , \varphi^\eps_t \right>_{L^2(\R)} = \left< \op_\eps\left(b\left(x+x^{0}(t),\xi+\xi^{0}(t)\right)\right) \hat{v}^\eps_{\pm,t} , \hat{v}^\eps_{\pm,t} \right>_{L^2(\R)} + o\left(1\right) + \mathcal{O}\left(\eps^{\alpha-2\beta}\right),
\end{equation*}}
and now the error $o(1)$ comes from the difference between calculating $b$ with the trajectories $(x^{\eta_\eps}(t),\xi^{\eta_\eps}(t))$ or the with ``limit'' path $(x^0(t),\xi^0(t))$, which must be negligible in compact times for $\eps$ small enough, given that $b$ is smooth and the flow that defines the trajectories is stable in the region where we are. Since it is possible to choose $\alpha > 2\beta$ within conditions \eqref{eq:constraint_1} and \eqref{eq:constraint_2}, an argument similar to the previous one gives 
$$\mu^{\varphi^{\eps,\eta_\eps}}_t(x,\xi) = \delta(x-x^0 (t)) \otimes \delta(\xi-\xi^0 (t))$$
for all $t\in[-T,T]$ when $0 < \beta < \frac{1}{10}$.
\end{proof}

\begin{remark}\label{rem:eta_positive}
All results in this section also work taking $\eta > 0$ and swapping $t$ negative for positive and conversely in the definition of the trajectories \eqref{eq:trajectory_wave_packet}. 
\end{remark}
Hence, we have found that it is possible that a particle arrive into the singularity from the up left or from the down right and that it continue to the other side down or up, as partially indicated in Figure \ref{fig:paths}(b). Moreover, we also proved that the wave packet approximation is valid for the non-smooth trajectories indicated in Figure \ref{fig:crossing_particle} (and for the reverse ones not indicated in the picture).

\section{Establishing the Liouville equation}\label{sec:establishing_liouville}
In view of the developments in Section \ref{sec:symbolic_calculus}, from equation \eqref{eq:d_op} we are left with the analysis of the commutator
{\small
$$\frac{i}{\eps}\left[ \hat{H}^{\eps} \, , \op_{\eps}\left(a\right) \right] = \frac{i}{\eps}\left[ -\frac{\eps^{2}}{2}\Delta \, , \op_{\eps}\left(a\right) \right] + \frac{i}{\eps}\left[ \, V_S \, , \op_{\eps}\left(a\right) \right] + \frac{i}{\eps}\left[ \, \|g(x)\|F(x) \, , \op_{\eps}\left(a\right) \right].$$}
We will look separately into each of these terms. The first one is kinetic, the other two dynamical. The first and the second are regular enough so one can use standard symbolic calculus; this presents no difficulties and we will treat them explicitly in Sections \ref{sec:kinetic_term} and \ref{sec:dynamical_smooth} only for a matter of completeness.

The third term is complicate because of the conical singularities it presents, which will require us to employ the two-microlocal analysis in Section \ref{sec:two_microlocal_analysis}. This strategy was followed in \cite{FGL}, but here we will describe the two-microlocal measures in more details. Prior to proceeding to this kind of analysis, however, we will need to restrict ourselves to the case where $g(x) = x'$, with $x = (x',x'')$, $x' \in \R^p$ and $1 \leqslant p \leqslant d$, in other words, to the case where the manifold $\Lambda$ formed by the singularities is actually a subspace $\R^p$.

It is in this context that we will be able to prove Proposition \ref{prop:main_result}, which is a particular version of Theorem \ref{th:main_result_1} for $\Lambda = \R^p$. Reducing the general case to this one is the subject of next section.

\subsection{Reducing $\Lambda$ to a subspace $\R^{d-p}$}\label{sec:reducing}
For a general conical potential, thanks to $\nabla g(x)$ having maximal rank we can define a local change of coordinates $\phi$ in neighbourhoods of $\R^d$ where 
\begin{eqnarray*}
z & = & \phi\left(x\right) \\ \left(
\begin{array}{c}
z'\\ z''
\end{array} \right) & = & \left(
\begin{array}{c}
g(x) \\ f(x)
\end{array} \right)
\end{eqnarray*}
for some function $f: \R^d \longrightarrow \R^{d-p}$ locally depending on $g$ in such a manner that $\nabla f(x)$ has maximal rank and, if $x \in \Lambda$, then $\ker \nabla f(x)$ is orthogonal to $\ker \nabla g(x)$\footnote{Such $f$ may be constructed as follows: let be $\mathcal{A}$ an open neighbourhood of $\Lambda$; choose $\kappa : \mathcal{A} \longrightarrow \mathbb{R}^{d-p}$ a local diffeomorphism; take $\tilde{U} \subset \mathbb{R}^d$ a cylindrical neighbourhood of $\Lambda$ such that $\tilde{U} \cap \Lambda \subset \mathcal{A}$. Pick up an open $U \subset \tilde{U}$, so $x \in U$ is given in geodesic coordinates by $x=(\tilde{\sigma},\eta)$ for some $\sigma \in \mathcal{A}$ with coordinates $\tilde{\sigma}$ and $\eta \in N_\sigma\Lambda$. Define $f : U \longrightarrow \mathbb{R}^{d-p}$ as $f(x) = \kappa(\sigma)$. It follows that $\nabla f(x)$ is diffeomorphic over $T_\sigma \Lambda$ and null over $N_\sigma\Lambda$, but since $T_\sigma\Lambda = \ker \nabla g((\tilde{\sigma},0))$ (see Remark \ref{rem:cotangent}), $\ker \nabla f((\tilde{\sigma},0)) \perp \ker \nabla g((\tilde{\sigma},0))$ and we are done.}. 

Now, for the sake of clarity let us consider the coordinate change in tangent space induced by $\phi$:
\begin{eqnarray*}
\tilde{\zeta} & = & \nabla \phi\left(x\right) \xi \\ 
\left(
\begin{array}{c}
\tilde{\zeta}' \\ \tilde{\zeta}''
\end{array} \right) & = & \left[
\begin{array}{c}
\nabla g(x) \\ \nabla f(x)
\end{array} \right] \xi.
\end{eqnarray*}

Writing $\R^d = \ker \nabla g(x) \oplus \ker \nabla f(x)$, we have the decomposition
$$\xi = \pi_{g}(x) \, \xi + \pi_{f}(x) \, \xi,$$
where $\pi_g(x)$ and $\pi_f(x)$ are suitable projectors inside $\R^d$ over the kernels of $\nabla g(x)$ and $\nabla f(x)$; if $x \in \Lambda$, they are orthogonal. Realize that $\left. \nabla f(x) \right|_{\ker \nabla g(x)}$ and $\left. \nabla g(x) \right|_{\ker \nabla f(x)}$ are invertible (due to the maximality of their ranks); let us denote their inverses simply by $\nabla g(x)^{-1}$ and $\nabla f(x)^{-1}$.

Thus one has:
\begin{equation}\label{eq:relations}
\begin{array}{l}
\nabla g(x)^{-1}\nabla g(x) = \pi_f(x) \\ \nabla f(x)^{-1}\nabla f(x) = \pi_g(x)
\end{array}
\quad {\rm and} \quad
\begin{array}{l}
\nabla g(x) \nabla g(x)^{-1} = 1\!\!1_{p \times p} \\ \nabla f(x) \nabla f(x)^{-1} = 1\!\!1_{d-p \times d-p}.
\end{array}
\end{equation}
It follows that $\nabla \phi(x)$ can be inverted in terms of $\nabla g(x)^{-1}$ and $\nabla f(x)^{-1}$; its inverse is
$$\nabla \phi(x)^{-1} = \left[ \: \nabla g(x)^{-1} \;\; \nabla f(x)^{-1} \: \right].$$

Analogously, $^t \nabla g(x)$ and $^t \nabla f(x)$ are invertible as soon as their counter-domains are restricted to $(\ker \nabla g(x))^\perp$ and $(\ker \nabla f(x))^\perp$ (in which case the transpose of the relations in \eqref{eq:relations} hold), allowing us to write the coordinate transformation in cotangent space as:
\begin{eqnarray}\label{eq:cotangent_transform}
\zeta & = & ^t \nabla \phi\left(x\right)^{-1} \xi \nonumber \\
\left( 
\begin{array}{c}
\zeta' \\ \zeta''
\end{array} \right) & = & \left[
\begin{array}{c}
^t\nabla g(x)^{-1} \, ^t\pi_{f} (x) \\ ^t \nabla f(x)^{-1} \, ^t\pi_{g} (x)
\end{array} \right] \xi.
\end{eqnarray}

Geometrically, let be the manifold $\Lambda = \left\lbrace x \in \R^d : g(x) = 0 \right\rbrace$, parametrized locally by the variable $z'' \in \R^{d-p}$. For a $x \in \Lambda$, the cotangent space can be described by $T_x^* \Lambda = \ker \nabla g(x)$ (see Remark \ref{rem:cotangent}) and its elements are parametrized by $\zeta''$. The variables $z'$ and $\zeta'$ in $\R^p$ will be associated with the normal and conormal spaces $N_x \Lambda = \faktor{\R^{d}}{T_x \Lambda}$ and $N_x^* \Lambda = \faktor{\R^{d}}{T_x^* \Lambda}$. We will also be using variables $\omega = \frac{z'}{\|z'\|}$ for $z' \neq 0$, which will be identified as elements of the normal space in sphere of $\Lambda$, $S_x \Lambda = \faktor{N_x \Lambda}{\R^+_*}$. 

\begin{remark}\label{rem:coordinate_global}
Due to the fact that the coordinate $z'$ is defined equally in any local charts by $z' = g(x)$, one can define functions on $\R^p$, more specifically on $N_x \Lambda$, simply by its explicit formulation in $z'$; further, $\Lambda$ not being empty allows us to always calculate a function at $z' = 0$. These facts are implicitly used in the calculations to come.
\end{remark}

At this point we shall state a central result in semiclassical analysis (see for instance Proposition 5.1 of \cite{x_ups} and its proof):
\begin{proposition}\label{prop:change_variables}
Let be $\phi$ a diffeomorphism of $\R^d$ in the sense of manifolds, and $\tilde{\phi}$ the correspondent cotangent bundle transformation, $\tilde{\phi}(x,\xi) = \left(\phi(x),^t \nabla \phi(x)^{-1}\xi\right)$. Let  $T_\phi \in \mathcal{L}\left(L^2(\R^d)\right)$ be the operator such that $T_\phi f = \left( J_\phi \circ \phi^{-1} \right)^{-\frac{1}{2}} f \circ \phi^{-1}$, where $J_\phi$ is the Jacobian of $\phi$. Then, $T_\phi^* T_\phi = 1\!\!1$ and
$$\left< \, \op_\eps(a) \, \Psi^\eps_t , \Psi^\eps_t \, \right> = \left< \, \op_\eps(a\circ \tilde{\phi}^{-1}) \, T_{\phi}\Psi^\eps_t \, , \, T_{\phi} \Psi^\eps_t \,\right> + \eps \: \sup_{\underset{|\alpha|+|\beta| = 1}{\alpha,\beta \in \N_0^d}} N_{d+1}\left(\partial_x^\alpha\partial_\xi^\beta a\right),$$
where $N_{d+1}(a)$ is the upper bound in \eqref{eq:estimation_xi}.
 
Besides, denoting $\psi^\eps_t = T_\phi \Psi^\eps_t$ and $\tilde{V} = V\circ \phi^{-1}$, the local expression for $\psi^\eps_t$ satisfies the equation 
\begin{equation}\label{eq:schroedinger_changed}
i\eps \partial_t \psi^\eps_t (x) = -\frac{\eps^2}{2} \, T_\phi \Delta T_\phi^* \, \psi^\eps_t (x) + \tilde{V}(x) \, \psi^\eps_t (x). 
\end{equation}
\end{proposition}

As a consequence, if $\tilde{\mu}$ is the measure associated to the family $\left(\psi^\eps\right)_{\eps > 0}$, for any symbol $a \in C_0^\infty(\R^{2d})$ and $\Xi \in C_0^\infty (\R)$, one has
$$\left< \, \mu \, , \, \Xi \, a \, \right>_{\R\times \R^{2d}} = \left< \tilde{\mu} \, , \, \Xi \, a \circ \tilde{\phi}^{-1} \, \right>_{\R\times \R^{2d}},$$
\emph{i.e}, any result got for the new measure $\tilde{\mu}$ can be immediately transferred to the original $\mu$ by simply changing the coordinate system. Therefore, Theorem \ref{th:main_result_1} will follow directly from Proposition \ref{prop:main_result} 


\begin{remark}
Since in the rest of this work the variables that we will write are going to be dummy, we will not care about marking the differences between $(x,\xi)$ and $(z,\zeta)$, nor about keeping the notations $\psi$, $\tilde{\mu}$ and $\tilde{V}$ in contrast to $\Psi$, $\mu$ and $V$.
\end{remark}

In short, now we can fairly relay on the study of the concentration of a family $\Psi^\eps$ satisfying equation \eqref{eq:schroedinger_changed} with a potential
$$ V(x) = V_{S}(x) + \| x' \| F(x),$$
with $x' \in \R^p$, \emph{i.e.}, satisfying the Schrödinger equation with a modified Hamiltonian operator
$$\hat{H}^\eps = -\frac{\eps^2}{2} T_\phi \Delta T_{\phi}^* + V,$$
so as that we become interested in the commutators
{\footnotesize
\begin{equation}\label{eq:commutators_new}
\frac{i}{\eps}\left[ \hat{H}^{\eps} \, , \op_{\eps}\left(a\right) \right] = \frac{i}{\eps}\left[ -\frac{\eps^{2}}{2}T_\phi\Delta T_\phi^* \, , \op_{\eps}\left(a\right) \right] + \frac{i}{\eps}\left[ \, V_S \, , \op_{\eps}\left(a\right) \right] + \frac{i}{\eps}\left[ \, \|x'\|F(x) \, , \op_{\eps}\left(a\right) \right].
\end{equation}}
In next sections we will analyse separately each one of these pieces.

\subsection{The kinetic term}\label{sec:kinetic_term}
Let us start the computation of the first term in the right-hand side of \eqref{eq:commutators_new} by the following exact calculation, with arbitrary $\Psi \in H^2(\R^d)$:
$$\frac{i}{\eps} \left[ -\frac{\eps^2}{2}\Delta \, , \op_\eps (a) \right] \Psi(x) =  \op_\eps \left(\xi \centerdot \partial_x a(x,\xi)\right) \Psi(x).$$

Observe that $\xi \centerdot \partial_x a \in C_0^\infty(\R^{2d})$, thus the pseudodifferential operator above can be extended to $L^2(\R^d)$, where it will be uniformly bounded with respect to $\eps$. Moreover, using last identity,  
$$\frac{i}{\eps} \left[ -\frac{\eps^2}{2}\, T_\phi \, \Delta \, T_\phi^* \, , \, \op_\eps(a) \, \right] = \op_\eps\left(D(x)\xi \centerdot \partial_x a \right) + O\left(\eps\right),$$
where $D(x) = \nabla \phi\left(\phi^{-1}(x)\right) ^t\nabla \phi\left(\phi^{-1}(x)\right)$ reads:
\begin{equation*}
D(x) = \left[
\begin{array}{c}
\nabla g\left(\phi^{-1}(x)\right) \\ \nabla f\left(\phi^{-1}(x)\right)
\end{array}\right] \left[ \: ^t\nabla g\left(\phi^{-1}(x)\right) \;\; ^t\nabla f\left(\phi^{-1}(x)\right) \: \right].
\end{equation*}
The result is:
\begin{eqnarray}
\int_{\R} \Xi(t) \left< \frac{i}{\eps} \left[ -\frac{\eps^2}{2}\Delta \, , \op_\eps (a) \right] \Psi^\eps_t \, , \Psi^\eps_t \right> dt \nonumber \hspace{-2cm} && \\
& \Tend{\eps}{0} & - \left< D(x) \xi \centerdot \partial_x \mu(t,x,\xi) \, , \Xi(t) \, a(x,\xi) \, \right>_{\R\times\R^{2d}}.
\label{eq:kinetic_part}
\end{eqnarray}

\begin{remark}\label{rem:convenience}
Observe that $\left(D(x)\xi\right)' = \nabla g\left(\phi^{-1}(x)\right) ^t\nabla \phi\left(\phi^{-1}(x)\right) \xi$ and $\left(D(x)\xi\right)'' = \nabla f\left(\phi^{-1}(x)\right) ^t\nabla \phi\left(\phi^{-1}(x)\right) \xi$. Back to the original coordinates, this gives $\left(D(z)\zeta\right)' = \nabla g(x) \xi$ and $\left(D(z)\zeta\right)'' = \nabla f(x) \xi$. Besides, from \eqref{eq:cotangent_transform} we have $\zeta' = \, ^t\nabla g(x)^{-1} \, ^t\pi_f(x) \xi$ and $\zeta'' = \, ^t\nabla f(x)^{-1} \, ^t\pi_g(x) \xi$, which implies 
$$\nabla g(x) ^t\nabla f(x) \zeta'' = \nabla g(x) \, ^t\pi_g(x) \xi \quad {\rm and} \quad \nabla f(x) ^t\nabla g(x) \zeta' = \nabla f(x) ^t\pi_f(x) \xi;$$
as we chose $f$ so as to have $\pi_g(x)$ and $\pi_f(x)$ orthogonal when $x \in \Lambda$, we are left with $D(0,z'')\zeta = \left( D_g(0,z'')\zeta' \, , \, D_f(0,z'')\zeta'' \right)$, where $D_g(z) = \nabla g\left(\phi^{-1}(z)\right) ^t \nabla g\left(\phi^{-1}(z)\right)$ is invertible for $z=(0,z'')$ and the same for $D_f(z)$ analogously defined.
\end{remark}

\subsection{The dynamical term -- smooth part}\label{sec:dynamical_smooth}
Consider the Taylor developments
{\small
$$V_S (x) = V_S\left(\frac{x+y}{2}\right) + \frac{1}{2}\int_0^1 \nabla V_S \left(\frac{x+y}{2} + s \frac{x-y}{2}\right) \centerdot (x-y) ds$$}
and
{\small
$$\nabla V_S \left(\frac{x+y}{2} + s\frac{x-y}{2}\right) = \nabla V_S \left(\frac{x+y}{2}\right) + \frac{1}{2}\int_0^1 \nabla^2 V_S \left(\frac{x+y}{2} + s's\frac{x-y}{2}\right) \centerdot (x-y) ds';$$}
plugging the latter inside the former\footnote{This kind of procedure will be largely used in the following pages, but we will not repeat the calculations textually everytime; exposing the kernels issued from the second order terms will be sufficient for our analyses.} and subtracting the resulting formula from the development one would obtain doing the same for $V_S(y)$ centred around $\frac{x+y}{2}$, we get:
{\small
\begin{eqnarray}\label{eq:taylor_smooth}
V_S \left( x \right) - V_S \left( y \right) = \nabla V_S \left(\frac{x+y}{2}\right) \centerdot \left( x - y \right)  + \frac{1}{4}\int_0^1\int_{0}^{1} s \left( \nabla^2 V_S \left( \frac{x+y}{2} +s's\frac{x-y}{2} \right) \right. \nonumber \\
\left. - \nabla^2 V_S \left( \frac{x+y}{2} - s's\frac{x-y}{2} \right) \right)\left(x-y\right)^{(2)} ds'ds. 
\end{eqnarray}}

Now, consider also the fact that $\frac{i}{\eps}\left[ \, V_S \, , \op_\eps (a) \, \right]$ has kernel
\begin{equation}\label{eq:before_plancherel}
k(x,y) = \frac{i}{\eps^{d+1}} \mathcal{F}_{\xi}^{-1} a\left(\frac{x+y}{2},\frac{x-y}{\eps}\right) \left(V_S(x)-V_S(y)\right).
\end{equation} 

Since so far we are still dealing with smooth symbols, as in standard symbolic calculus we use the formula $x\mathcal{F}_\xi^{-1}a = i\mathcal{F}_\xi^{-1}(\partial_\xi a)$ to exchange the factors $(x-y)$ in \eqref{eq:taylor_smooth} by factors $\eps \partial_\xi a$ in \eqref{eq:before_plancherel}. Because both $\nabla V_S$ and $\nabla^2 V_S$ do not grow faster than some polynomial and $a$ is compactly supported, it is a direct computation to get 
$$\frac{i}{\eps}\left[ \, V_S \, , \op_\eps (a) \, \right] = \op_\eps \left( -\nabla V_S \centerdot \partial_\xi a \right) + \mathcal{O}\left(\eps\right),$$
where $\mathcal{O}(\eps)$ tends to $0$ in $\mathcal{L}\left(L^2(\R^d)\right)$ and $\op_\eps (-\nabla V_S \centerdot \partial_\xi a )$ is uniformly bounded with respect to $\eps$. This naturally gives
\begin{eqnarray}\label{eq:smooth_piece}
\int_{\R} \Xi(t) \left< \frac{i}{\eps}\left[ \, V_S \, , \op_\eps (a) \, \right] \Psi^\eps_t \, , \Psi^\eps_t \right> dt \nonumber \hspace{-1cm} && \\
& \Tend{\eps}{0} & \left< \, \nabla V_S (x) \centerdot \partial_\xi \mu (t,x,\xi) \, , \Xi(t) \, a(x,\xi) \, \right>_{\R \times \R^{2d}}.
\end{eqnarray}

\subsection{The dynamical term -- singular part}
In order to analyse the commutator with $\| x' \|F$ in \eqref{eq:commutators_new}, we will introduce a cut-off $\chi \in C_{0}^\infty(\R^p)$, $0 \leqslant \chi \leqslant 1$ , $\chi(x') = 0$ for $\|x'\| \geqslant 1$ and $\chi(x') = 1$ for $\|x'\| \leqslant \frac{1}{2}$. Let us cut the symbol $a$ into three parts using parameters $R > 0$ and $\delta > \eps R$, as follows:
{\small
\begin{equation}\label{eq:symbol_decomposition}
a(x,\xi) = a(x,\xi)\, \chi\left(\frac{x'}{\eps R}\right) + a(x,\xi)\left(1-\chi\left(\frac{x'}{\eps R}\right)\right)\chi\left(\frac{x'}{\delta}\right) + a (x,\xi)\left(1-\chi\left(\frac{x'}{\delta}\right)\right).
\end{equation}}

In the context of two-microlocal analysis, each of these pieces is related to a different two-microlocal measure, and that is what we will be talking about in the next sections.

\subsubsection{The inner part}\label{sec:the_inner_part}
Defining $\tilde{\Psi}^\eps(x) = \eps^{\frac{p}{2}} \Psi^\eps (\eps x',x'')$, one calculates:
{\small
\begin{eqnarray}\label{eq:calcul_m_grand}
&& \left< \frac{i}{\eps}\left[ \| x' \| F (x) \, , \op_\eps \left(a(x,\xi)\chi\left(\frac{x'}{\eps R}\right)\right) \right] \Psi_t^\eps \, , \Psi_t^\eps \right> \nonumber \\
& = & \frac{i}{\eps\left(2\pi\eps \right)^{d}} \int_{\R^{3d}} e^{\frac{i}{\eps}\xi \centerdot \left( x - y \right)} \, a\left(\frac{x+y}{2},\xi\right) \chi\left(\frac{x'+y'}{2\eps R}\right) \left( \|x'\| F(x) - \|y'\| F(y) \right) \nonumber \\ && \hspace{9cm} \Psi_t^\eps\left(y\right) \overline{\Psi_t^\eps\left(x\right)} \; dxd\xi dy \nonumber \\
& = & \int_{\R^{3d}} \frac{ie^{i \xi' \centerdot \left( x' - y' \right)}e^{\frac{i}{\eps} \xi'' \centerdot \left( x'' - y'' \right)}}{\left(2\pi \right)^{d} \eps^{d-p}} \, a\left(\eps\frac{x'+y'}{2},\frac{x''+y''}{2},\xi',\xi''\right) \chi\left(\frac{x'+y'}{2R}\right) \nonumber \\ && \hspace{4.2cm} \left( \|x'\| F(\eps x',x'')  - \|y'\| F(\eps y', y'') \right) \tilde{\Psi}_t^\eps\left(y\right) \overline{\tilde{\Psi}_t^\eps\left(x\right)} \; dxd\xi dy \nonumber \\ 
& = & \int_{\R^{3d}} \frac{ie^{i \xi' \centerdot \left( x' - y' \right)}e^{\frac{i}{\eps} \xi'' \centerdot \left( x'' - y'' \right)}}{\left(2\pi \right)^{d} \eps^{d-p}} \, a\left(0,\frac{x''+y''}{2},\xi',\xi''\right) \nonumber \\ && \hspace*{2cm} \chi\left(\frac{x'+y'}{2R}\right) F\left(0,\frac{x''+y''}{2}\right) \left( \|x'\| - \|y'\| \right) \tilde{\Psi}_t^\eps\left(y\right) \overline{\tilde{\Psi}_t^\eps\left(x\right)} \; dxd\xi dy + R^\eps,
\end{eqnarray}}
with $R^\eps$ an error of order $\eps$ in $\R$ whose analysis will be postponed.

Now, for each $x'',\xi'' \in \R^{d-p}$, denote by 
{\small
$$kA^R_{\left(x'',\xi''\right)} \left(x',y'\right) = \frac{i}{\left(2\pi\right)^p} \int_{\R^p} e^{i \xi' \centerdot \left(x'-y'\right)} a\left(0,x'',\xi',\xi''\right) \, \chi\left(\frac{x'+y'}{2R}\right) \left( \|x'\| - \|y'\| \right) F\left(0,x''\right) d\xi'$$}
the integral kernel of the $L^2(\R^p_y)$ operator $A^R_{\left(x'',\xi''\right)} = i\left[ \, \|y\| F(0,x'') \, , a_R^{w}(0,x'',\partial_y,\xi'',y) \, \right]$, where $a_R^w(0,x'',\partial_y,\xi'',y)$ is the Weyl quantization of the symbol $(y,\zeta) \longmapsto a(0,x'',\zeta,\xi'') \, \chi\left(\frac{y}{R}\right)$, and, as in \eqref{eq:k_u}, by $kU_{\left(t,\,x'',\xi''\right)}^\eps$ the kernel of the correspondent bounded $L^2(\R^p)$ operator $U^\eps_{\left(t,\, x'',\xi''\right)}$ (which is an operator-valued generalization of the Wigner transform $W^\eps\Psi^\eps$) introduced in Proposition \ref{prop:two_microlocal_measures}. Then, the object in the previous calculation reads
$$\int_{\R^{2d}} kA^R_{\left(x'',\xi''\right)} \left(x',y'\right) kU^\eps_{\left(t,\, x'',\xi''\right)}\left(y',x'\right) dy'dx' dx''d\xi'' + R^\eps,$$
which gives
{\small
$$\left< \frac{i}{\eps}\left[ \| x' \| F (x) \, , \op_\eps \left(a(x,\xi)\chi\left(\frac{x'}{\eps R}\right)\right) \right] \Psi_t^\eps \, , \Psi_t^\eps \right> = {\rm tr} \int_{\R^{2(d-p)}} A^R_{\left(x'',\xi''\right)} U^\eps_{\left(t,\, x'',\xi''\right)}dx''d\xi'' + R^\eps.$$}

Regarding the error:
$$R^\eps = i\eps \left< \left( B^\eps + C^\eps \right) \tilde{\Psi}_t^\eps \, , \tilde{\Psi}_t^\eps \right>,$$
where $B^\eps$ and $C^\eps$ are the integral operators with the respective kernels:
{\small
$$b^\eps \left(x,y\right) = \frac{1}{\eps^{(d-p)}}\tilde{b}^\eps \left(\frac{x+y}{2},x'-y',\frac{x''-y''}{\eps}\right)$$}
and
{\small
$$c^\eps \left(x,y\right) = \frac{1}{\eps^{(d-p)}}\tilde{c}^\eps\left(\frac{x+y}{2},x'-y',\frac{x'-y'}{\eps}\right),$$}
where one has
{\footnotesize
\begin{align*}
\tilde{b}^\eps (x',x'',y'&,y'') = \int_{0}^1 \partial_{x'} \mathcal{F}_{\xi}^{-1} a\left( \eps s x', x'',y',y''\right) \centerdot x' \, \chi\left( \frac{x'}{R} \right) \\
& \left( \left\| x' + \frac{y'}{2} \right\| F\left(\eps \left(x'+\frac{y'}{2}\right),x''+\frac{\eps y''}{2}\right) - \left\| x'- \frac{y'}{2} \right\| F \left(\eps \left(x'-\frac{y'}{2}\right), x''-\frac{\eps y''}{2} \right) \right) ds
\end{align*}}
and
{\footnotesize
\begin{align*}
\tilde{c}^\eps \left(x',x'',y'.y''\right) = \int_{0}^1 \mathcal{F}_{\xi}^{-1} & a\left(0, x'',y',y''\right) \chi\left( \frac{x'}{R} \right) \\
& \left[ \left\| x' + \frac{y'}{2} \right\| \nabla F \left( \eps s \left( x'+ \frac{y'}{2} \right), x'' + s\frac{\eps y''}{2} \right) \centerdot \left( x'+\frac{y'}{2}, \frac{y''}{2} \right) \right. \\
& \hspace{0.7cm} \left. + \left\| x'-\frac{y'}{2} \right\| \nabla F \left(\eps s \left(x'-\frac{y'}{2}\right), x'' - s\frac{\eps y''}{2} \right) \centerdot \left( x'-\frac{y'}{2}, \frac{y''}{2} \right) \right] ds.
\end{align*}}

\begin{lemma}\label{lem:b_c_bounded}
The operators $B^\eps$ and $C^\eps$ are uniformly bounded with respect to $\eps$.
\end{lemma}

\begin{proof}
Let us prove the lemma for $B^\eps$ by using the conventional Schur test and the fact that $\mathcal{F}^{-1}_\xi a$ is a rapidly decreasing function bounded by the polynomial $\mathfrak{p}$ given in the beginning of Section \ref{sec:statement}).

In fact, noting $\mathfrak{P} = \max_{\|x'\| \leqslant R} \mathfrak{p}(x',0)$ and recalling the sub-additivity of $\mathfrak{p}$, we have some $K \geqslant 1$ such that
{\footnotesize
$$\left| \tilde{b}^\eps(x,y) \right| \; \leqslant \; 1\!\!1_{\lbrace \| x' \| \leqslant R \rbrace} \left(R + \frac{\|y'\|}{2} \right) K \left( \mathfrak{P} + \mathfrak{p}(0,x'') + \mathfrak{p}(\eps y) \right) \max_{\underset{\|z\| \leqslant R}{z \in \R^p}} \left\| \partial_{x'} \mathcal{F}_\xi^{-1} a(z,x'',y) \right\|_{\mathcal{L}(\R^p)},$$}
which shows that $\tilde{b}^\eps$ is also a Schwartz function, implying:
\begin{eqnarray*}
\sup_{x \in \R^d} \int_{\R^d} \left| b^\eps(x,y) \right| dy & \leqslant & \sup_{x \in \R^d} \frac{1}{\eps^{(d-p)}}\int_{\R^d} \left| \tilde{b}^\eps \left(\frac{x+y}{2},x'-y',\frac{x''- y''}{\eps}\right) \right| dy \\
& = & \sup_{x \in \R^d} \int_{\R^d} \left| \tilde{b}^\eps \left(x'-\frac{y'}{2},x''-\frac{\eps y''}{2},y',y''\right) \right| dy \\
& = & \sup_{x \in \R^d} \int_{\R^d} \left| \tilde{b}^\eps \left(x'-\frac{y'}{2},x''-\frac{\eps y''}{2},y\right) \left\langle y \right\rangle^{2d} \right| \left\langle y \right\rangle^{-2d} dy \\
& \leqslant & \max_{x,y \in \R^d} \left| \tilde{b}^\eps \left(x,y\right) \left\langle y \right\rangle^{2d} \right| \int_{\R^d} \left\langle y \right\rangle^{-2d} dy \\
& < & \infty.
\end{eqnarray*}
The estimate $\sup_{y \in \R^d} \int_{\R^d} \left| b^\eps(x,y) \right| < \infty$ is found by following the very same steps above, so Schur's lemma allows us to conclude that $\| B^\eps \|_{\mathcal{L}\left(L^2(\R^d)\right)} < \infty$ uniformly with respect to $\eps$.

Regarding $C^\eps$, the proof is, \emph{mutatis mutandis}, exactly as we have done for $B^\eps$ and will be omitted.
\end{proof}

Now we only need to focus on the lasting term; from what we have seen in Section \ref{sec:two_microlocal_analysis}, in the limit where $\eps \longrightarrow 0$ and then $R \longrightarrow \infty$, it gives
{\footnotesize
\begin{eqnarray}
\int_\R \Xi(t) \left< \frac{i}{\eps}\left[ \| x' \| F (x) \, , \op_\eps \left(a(x,\xi)\chi\left(\frac{x'}{\eps R}\right)\right) \right] \Psi_t^\eps \, , \Psi_t^\eps \right>dt \nonumber \hspace{-7.5cm} && \\ 
&& \longrightarrow \limsup_{R \longrightarrow \infty} \, {\rm tr} \left< \, M\left(t,x'',\xi''\right) \, , \, i \left[ \, \|y\|F\left(0,x''\right) , \Xi(t) \, a^{w}\left(0,x'',\partial_y,\xi''\right) \chi\left(\frac{y}{R}\right) \, \right] \, \right>_{\R \times \R^{2(d-p)}}, \label{eq:inner_part}
\end{eqnarray}}
where $M$ is the two-microlocal operator-valued measure in Proposition \ref{prop:two_microlocal_measures} and which is finite because of the lemma below.

\begin{lemma}\label{lem:estimate_for_M}
One has got the estimate
\begin{align}\label{eq:estimate_for_M}
\limsup_{R \longrightarrow \infty} \, {\rm tr} \left< \, M\left(t,x'',\xi''\right) \right. & \left. \, , \, i \left[ \, \|y\|F\left(0,x''\right) , \Xi(t) \, a^{w}\left(0,x'',\partial_y,\xi''\right) \chi\left(\frac{y}{R}\right) \, \right] \, \right>_{\R \times \R^{2(d-p)}} \hspace{2cm} \nonumber \\
& \leqslant TK \sup_{\underset{|\alpha| \leqslant d+1}{\alpha\in \mathbb{N}_0^d}} \, \sup_{x'' \in \R^{d-p}} \int_{\R^d} \left\| \partial_\xi^\alpha \partial_{\xi'} a(0,x'',\xi) F(0,x'') \right\|_{\R^p} d\xi,
\end{align}
where $T,K > 0$ are constants, $\Xi \in C_0^\infty([-T,T])$ and $a \in C_0^\infty(\R^{2d})$.
\end{lemma}

\begin{proof}
From a calculation similar to that we made in \eqref{eq:calcul_m_grand} and similar estimates, it follows that
$$\left< \frac{i}{\eps}\left[ \| x' \| F (x) \, , \op_\eps \left(a(x,\xi)\chi\left(\frac{x'}{\eps R}\right)\right) \right] \Psi_t^\eps \, , \Psi_t^\eps \right> = \left< P^\eps \, \tilde{\Psi}^\eps_t \, , \, \tilde{\Psi}^\eps_t \right> + \mathcal{O}(\eps),$$
where $P^\eps$ is the operator with kernel
{\footnotesize
\begin{eqnarray*}
p^\eps(x,y) & = & \int_{\R^{d}} \frac{ie^{i \xi \centerdot \left(x-y\right)}}{\left(2\pi\right)^{d}} a\left(0,\frac{x''+y''}{2},\xi',\eps \xi''\right) \chi\left(\frac{x'+y'}{2R}\right) F(0,x'') \left(\|x'\| - \|y'\|\right) d\xi \\
& = &  -\int_{\R^{d}} \frac{e^{i \xi \centerdot \left(x-y\right)}}{\left(2\pi\right)^{d}} \frac{x'+y'}{\|x'\| + \|y'\|} \centerdot \partial_{\xi'} a\left(0,\frac{x''+y''}{2},\xi',\eps \xi'' \right) \chi\left(\frac{x'+y'}{2R}\right) F(0,x'') d\xi
\end{eqnarray*}}
(recall: $\|x'\| - \|y'\| = \frac{x'+y'}{\|x'\| + \|y'\|} \centerdot \left(x'-y'\right)$). Further, let be $\tilde{b}$ the kernel of the operator $\op_{1,\eps}\left(b\right) = \op_1\left(b(x,\xi',\eps\xi'')\right)$, with 
$$b(x,\xi) = -\mathbf{1}_p \centerdot \partial_{\xi} a(0,x'',\xi) \, \chi\left(\frac{x'}{R}\right) F(0,x''),$$
where $\mathbf{1}_p = (1,...\,,1) \oplus (0,...\,,0) \in \R^p \times \R^{d-p}$. 

Observe that $| p^\eps (x,y) | \leqslant |\tilde{b}(x,y)|$, which causes the Schur estimate for the norm of $\op_{1,\eps}\left(b\right)$ to be greater than that for $P^\eps$. Besides, the Schur estimate for $\op_{1,\eps}\left(b\right)$ is upper bounded by an estimate of type \eqref{eq:estimation_xi} (see Remark \ref{rem:schur_upper_bound}), which must, consequently, be an upper bound for the norm of $P^\eps$ as well. This estimate is majorated by the one we stated in the lemma, which is independent of $R$.
\end{proof}

\begin{remark}
In \cite{FGL}, an estimate that turns up to be equivalent to last lemma was obtained by noticing directly that $\frac{i}{\eps}\left[ \|x'\|F(x) , \op_\eps(a) \right]$ is bounded uniformly with respect to $\eps$, which was used in proving that $\frac{i}{\eps}\left[ \hat{H}^\eps, \op_\eps(a) \right]$ is itself bounded.
\end{remark}

In Section \ref{sec:establishing_the_equation}, this $\limsup$ will be shown to be zero.

\subsubsection{The outer part}\label{sec:the_outer_part}
We start by proving with standard symbolic calculus the technical result below:
\begin{lemma}\label{lem:better_commutator}
For $\delta > 0$, $\delta \geqslant \eps R$, one has the following estimation in $\mathcal{L}\left(L^2(\R^d)\right)$:
{\footnotesize
\begin{eqnarray*}
\frac{i}{\eps}\left[ \| x' \|F(x) \, , \, \op_\eps \left( a(x,\xi)\left(1-\chi\left(\frac{x'}{\delta}\right)\right) \right) \right] = \frac{i}{\eps}\left[ \| x' \| F(x) \left(1-\chi\left(\frac{x'}{\delta}\right)\right) \, , \op_\eps \left( a(x,\xi) \right) \right] \\ - \op_\eps \left(\frac{1}{\delta} \|x'\|F(x) \, \chi' \left( \frac{x'}{\delta} \right)\centerdot\partial_{\xi'}a(x,\xi)\right) + \mathcal{O}\left(\frac{\eps}{\delta}\right).
\end{eqnarray*}}
\end{lemma}

\begin{proof}
In view of the identities
{\small
\begin{eqnarray*}
\op_\eps \left( a(x,\xi) \left(1-\chi\left(\frac{x'}{\delta}\right)\right) \right) & = & \left(1-\chi\left(\frac{x'}{\delta}\right)\right) \op_\eps \left( a(x,\xi) \right) + R_{\eps,\delta}^l \\
& = & \op_\eps \left( a(x,\xi) \right) \left(1-\chi\left(\frac{x'}{\delta}\right)\right) + R_{\eps,\delta}^r, 
\end{eqnarray*}}
where $R_{\eps,\delta}^l$ and $R_{\eps,\delta}^r$ have integral kernels
{\small
$$r_{\eps,\delta}^l (x,y) = \frac{i}{\eps^d} \mathcal{F}_{\xi} a\left(\frac{x+y}{2},\frac{x-y}{\eps}\right) \left( \chi\left(\frac{x'}{\delta}\right) - \chi\left(\frac{x'+y'}{2\delta}\right)\right)$$}
and
{\small
$$r_{\eps,\delta}^r (x,y) = \frac{i}{\eps^d} \mathcal{F}_{\xi} a\left(\frac{x+y}{2},\frac{x-y}{\eps}\right) \left( \chi\left(\frac{y'}{\delta}\right) - \chi\left(\frac{x'+y'}{2\delta}\right)\right),$$}
we have 
{\footnotesize
\begin{eqnarray*}
\frac{i}{\eps}\left[ \| x' \|F(x) \, , \, \op_\eps \left( a(x,\xi)\left(1-\chi\left(\frac{x'}{\delta}\right)\right) \right) \right] = \frac{i}{\eps}\left[ \| x' \| F(x) \left(1-\chi\left(\frac{x'}{\delta}\right)\right) \, , \op_\eps \left( a(x,\xi) \right) \right] \\ + \underbrace{\frac{i}{\eps}\left( \|x'\| F(x) R_{\eps,\delta}^l - R_{\eps,\delta}^r \|x'\| F(x) \right)}_{R_{\eps,\delta}},
\end{eqnarray*}}
where $R_{\eps,\delta}$ has kernel
{\footnotesize
\begin{eqnarray*}
r_{\eps,\delta}(x,y) = \frac{i}{\eps^{d+1}} \mathcal{F}_\xi a\left(\frac{x+y}{2},\frac{x-y}{\eps}\right) \left( \chi\left(\frac{x'}{\delta}\right)\|x'\| F(x) - \chi\left(\frac{y'}{\delta}\right)\|y'\| F(y) - \chi\left(\frac{x'+y'}{2\delta}\right)\right. \\
\left.\left(\|x'\| F(x) - \|y'\|F(y)\right)\right).
\end{eqnarray*}}
Using Taylor developments for $\chi$ (where $\chi'$ and $\chi''$ will stand for the first and the second derivatives of $\chi$), this kernel can be written in the form
{\footnotesize
\begin{eqnarray*}
r_{\eps,\delta}(x,y) = \frac{i}{\eps^d} \mathcal{F}_\xi a\left(\frac{x+y}{2},\frac{x-y}{\eps}\right) \left[ \frac{1}{2\delta} \left(\|x'\|F(x) + \|y'\|F(y) \right) \chi'\left(\frac{x'+y'}{2\delta} \right) \centerdot \left( \frac{x'-y'}{\eps} \right) \right. \\
+ \frac{\eps}{4\delta^2}\int_0^1 \int_0^1 s\left( \chi''\left(\frac{x'+y'}{2\delta} + s's\frac{x'-y'}{2\delta}\right)\|x'\|F(x) - \chi''\left(\frac{x'+y'}{2\delta} - s's\frac{x'-y'}{2\delta}\right)\|y'\|F(y) \right) \\ \left. \left( \frac{x'-y'}{\eps}\right)^{(2)} ds'ds \right],
\end{eqnarray*}}
and still, with developments for $\| \cdot \|F$,
{\footnotesize
\begin{eqnarray*}
r_{\eps,\delta}(x,y) = \frac{i}{\eps^d} \mathcal{F}_\xi a\left(\frac{x+y}{2},\frac{x-y}{\eps}\right) \left[ \frac{1}{\delta} \left\| \frac{x'+y'}{2} \right\| F\left(\frac{x+y}{2}\right) \chi'\left(\frac{x'+y'}{2\delta} \right) \centerdot \left( \frac{x'-y'}{\eps} \right) \right. \\
\left. + \frac{\eps}{\delta} \left(\frac{x-y}{\eps}\right) {^tA(x,y)} \left(\frac{x'-y'}{\eps}\right) + \frac{\eps}{4\delta^2}\int_0^1 \int_0^1 s \left( B(x,y) \left\| \frac{x'+y'}{2} \right\|F\left(\frac{x+y}{2}\right) \right.\right. \\
\left.\left. + \eps \left( \frac{x-y}{\eps} \right) {^t C(x,y)} \right) \left( \frac{x'-y'}{\eps}\right)^{(2)} ds'ds \right],
\end{eqnarray*}}
where 
{\footnotesize
$$A(x,y) = \frac{1}{2} \, \chi'\left(\frac{x'+y'}{2\delta}\right) \otimes \sum_{j = 1,2} \int_{0}^{1} (-1)^j \, \nabla \left( \| \cdot \|F \right) \left(\frac{x+y}{2} + (-1)^j \, s'' \frac{x-y}{2}\right) ds'',$$}
{\footnotesize
$$B(x,y) = \sum_{j=1,2} (-1)^j \, \chi''\left(\frac{x'+y'}{2\delta} + (-1)^j \, s's\frac{x'-y'}{2\delta}\right)$$}
and
{\footnotesize
$$C(x,y) = \frac{1}{2} \sum_{j = 1,2} \int_{0}^{1} \chi''\left(\frac{x'+y'}{2\delta}+(-1)^j\, s's\frac{x'-y'}{2\delta}\right) \otimes \nabla \left( \| \cdot \|F \right) \left(\frac{x+y}{2} + (-1)^j \, s'' \frac{x-y}{2}\right) ds''.$$}
Observe now that $A$, $B$ and $C$ are bounded and, furthermore, $B$ is supported on $\left\| \frac{x'+y'}{2} \right\| \leqslant \delta + \frac{\eps}{2}\left\| \frac{x'-y'}{\eps} \right\|$ (given that $\chi''(x')$ is null for $\| x' \| > \delta$). This, along with the identity $i\mathcal{F}_{\xi}^{-1}a(x,y) \, y' = -\mathcal{F}_{\xi}^{-1}(\partial_{\xi'}a)(x,y)$, allows us to write $R_{\eps,\delta}$ as
$$R_{\eps,\delta} = -\op_\eps\left(\frac{1}{\delta}\|x'\|F(x) \, \chi'\left(\frac{x'}{\delta}\right) \centerdot \partial_{\xi'} a(x,\xi) \right) + \tilde{R}_{\eps,\delta},$$
where $\tilde{R}_{\eps,\delta}$ has an integral kernel such that
{\footnotesize
\begin{eqnarray*}
\left| \tilde{r}_{\eps,\delta}(x,y) \right| \leqslant \frac{Q}{\eps^d} \left| \mathcal{F}_\xi a\left(\frac{x+y}{2},\frac{x-y}{\eps}\right) \right| \left| \frac{\eps}{\delta} \left\| \frac{x-y}{\eps} \right\|^2 + \frac{\eps}{\delta^2} \left( \delta + \eps \left\| \frac{x-y}{\eps} \right\| \right) \left\| \frac{x-y}{\eps} \right\|^2 \right|
\end{eqnarray*}}
where $Q > 0$ is some constant big enough. The lemma will follow as one uses the estimation above for the Schur test to show that $\| \tilde{R}_{\eps,\delta} \|_{\mathcal{L}(L^2(\R^d))} \leqslant \frac{\eps}{\delta}$. 
\end{proof}

\begin{remark}\label{rem:useful_in_middle_part}
Replacing $\delta$ by $\eps R$ and $a(x,\xi)$ by $a(x,\xi) \, \chi\left(\frac{x'}{\delta}\right)$, Lemma \ref{lem:better_commutator} gives
{\tiny
\begin{eqnarray*}
\frac{i}{\eps}\left[ \| x' \|F(x) \, , \, \op_\eps \left( a(x,\xi)\left(1-\chi\left(\frac{x'}{\eps R}\right)\right)\, \chi\left(\frac{x'}{\delta}\right) \right) \right] = \frac{i}{\eps}\left[ \| x' \| F(x) \left(1-\chi\left(\frac{x'}{\eps R}\right)\right) \, , \op_\eps \left( a(x,\xi)\, \chi\left(\frac{x'}{\delta}\right) \right) \right] \\ - \op_\eps \left(\frac{1}{\eps R} \|x'\|F(x) \, \chi' \left( \frac{x}{\eps R} \right)\centerdot\partial_{\xi'}a(x,\xi)\, \chi\left(\frac{x'}{\delta}\right)\right) + \mathcal{O}\left(\frac{1}{R}\right),
\end{eqnarray*}}
which is going to be remarkably useful in Section \ref{sec:the_middle_part}.
\end{remark}

\begin{lemma}\label{lem:commutator_smooth}
For $\delta > 0$, $\delta \geqslant \eps R$, one has
{\footnotesize
\begin{eqnarray*}
\frac{i}{\eps} \left[ \|x'\| F(x) \left(1 - \chi\left( \frac{x'}{\delta} \right) \right) \, , \, \op_\eps(a(x,\xi)) \right] = \op_\eps \left( - \nabla\left(\|x'\|F(x)\right) \centerdot \partial_{\xi} a(x,\xi) \left(1-\chi\left(\frac{x'}{\delta}\right) \right) \right) \\ + \op_\eps \left( \frac{1}{\delta} \|x'\|F(x) \, \chi' \left( \frac{x'}{\delta} \right)\centerdot\partial_{\xi'}a(x,\xi) \right) + \mathcal{O}\left(\eps\right) + \mathcal{O}\left(\frac{\eps}{\delta}\right)
\end{eqnarray*}}
in $\mathcal{L}\left(L^2(\R^d)\right)$.
\end{lemma}

\begin{proof}
Because $\|x'\| F(x) \left(1 - \chi\left( \frac{x'}{\delta} \right) \right)$ is everywhere smooth, we can apply usual symbolic calculus as in \ref{eq:taylor_smooth} and \ref{eq:smooth_piece} to get
{\scriptsize
$$\frac{i}{\eps} \left[ \|x'\| F(x) \left(1 - \chi\left( \frac{x'}{\delta} \right) \right) \, , \, \op_\eps(a(x,\xi)) \right] = \op_\eps \left( - \partial_x \left( \|x'\|F(x) \right) \centerdot \partial_{\xi} a(x,\xi) \left(1-\chi\left(\frac{x'}{\delta} \right) \right) \right) + R^{\eps,\delta}$$}
where $R^{\eps,\delta}$ has kernel
{\small
$$r^{\eps,\delta}(x,y) = \frac{i\eps}{4\eps^d}\int_0^1 \int_0^1 \mathcal{F}_{\xi}^{-1}a\left(\frac{x+y}{2},\frac{x-y}{\eps}\right) B\left(x,y\right) \left(\frac{x-y}{\eps}\right)^{(2)} ds'ds,$$}
with $B$ the matrix
{\small
$$B(x,y) = \sum_{j = 1,2} \left(-1\right)^{-1} \nabla^2 \left( \|x'\|F(x)\left(1-\chi\left(\frac{x'}{\delta}\right)\right) \right) \left( \frac{x+y}{2} + (-1)^j s's \frac{x-y}{2} \right).$$}
From the growth properties of $F$, it is easy to see that there exists some $K \geqslant 1$ such that $\| B (x,y) \|_{\mathcal{L}(\R^d)} \leqslant K \left(\mathfrak{p}(x)+\mathfrak{p}(y)\right) \left( 1 + \frac{1}{\delta}\right)$ and, therefore,
{\small
$$\sup_{x \in \R^d} \int_{\R^d} \left| r(x,y) \right| dy \; \leqslant \; \eps K \left(1+\frac{1}{\delta}\right) \max_{x,y \in \R^d} \left| \left(\mathfrak{p}(x)+\mathfrak{p}(y)\right) \mathcal{F}_{\xi}^{-1}a(x,y) \, y^2 \left\langle y \right\rangle^{2d} \right| \int_{\R^d} \left\langle y \right\rangle^{-2d} dy;$$}
the same estimation holding for $\sup_{y \in \R^d} \int_{\R^d} \left| r(x,y) \right| dx$, it follows by the Schur test that $\left\| R^{\eps,\delta} \right\|_{\mathcal{L}\left(L^2(\R^d)\right)} \lesssim \eps + \frac{\eps}{\delta}$.

The rest of the proof consists on the basic derivation
\begin{eqnarray*}
\partial_x \left( \|x'\|F(x) \left(1-\chi\left(\frac{x'}{\delta} \right) \right)\right) \centerdot \partial_{\xi} a(x,\xi) = \left(1-\chi\left(\frac{x'}{\delta} \right) \right) \partial_x \left( \|x'\|F(x) \right) \centerdot \partial_{\xi} a(x,\xi) \\ - \frac{1}{\delta} \|x'\|F(x) \, \chi'\left(\frac{x'}{\delta}\right) \centerdot \partial_{\xi'} a(x,\xi)
\end{eqnarray*}
and on the linearity of the pseudodifferential operators with respect to their symbols.
\end{proof}

\begin{remark}\label{rem:useful_in_middle_part_2}
Analogously to Remark \ref{rem:useful_in_middle_part}, in Section \ref{sec:the_middle_part} we will use that
{\tiny
\begin{eqnarray*}
\frac{i}{\eps} \left[ \|x'\| F(x) \left(1 - \chi\left( \frac{x'}{\eps R} \right) \right) \, , \, \op_\eps\left( a(x,\xi) \, \chi\left(\frac{x'}{\delta}\right) \right) \right] = \op_\eps \left( - \nabla\left(\|x'\|F(x)\right) \centerdot \partial_{\xi} a(x,\xi) \left(1-\chi\left(\frac{x'}{\eps R}\right) \right) \chi\left(\frac{x'}{\delta}\right) \right) \\ + \op_\eps \left( \frac{1}{\eps R} \|x'\|F(x) \, \chi' \left( \frac{x'}{\eps R} \right)\centerdot \partial_{\xi'}a(x,\xi) \chi\left(\frac{x'}{\delta}\right) \right) + \mathcal{O}\left(\eps\right) + \mathcal{O}\left(\frac{1}{R}\right).
\end{eqnarray*}}
There, however, there will be a tiny question about the limit in $\eps$.
\end{remark}

Combining Lemmata \ref{lem:better_commutator} and \ref{lem:commutator_smooth}, we obtain
{\footnotesize
\begin{eqnarray*}
\frac{i}{\eps}\left[ \| x' \|F(x) \, , \, \op_\eps \left( a(x,\xi)\left(1-\chi\left(\frac{x'}{\delta}\right)\right) \right) \right] = \op_\eps \left( - \nabla\left(\|x'\|F(x)\right) \centerdot \partial_{\xi} a(x,\xi) \left(1-\chi\left(\frac{x'}{\delta}\right) \right) \right) \\ + \mathcal{O}(\eps) + \mathcal{O}\left(\frac{\eps}{\delta}\right),
\end{eqnarray*}}
which gives, in the limit where $\eps \longrightarrow 0$ (using Proposition \ref{prop:two_microlocal_measures}, or just the standard semiclassical measure) and then $\delta \longrightarrow 0$ (using dominated convergence),
\begin{eqnarray}
\int_\R \Xi(t) \left< \frac{i}{\eps} \left[ \|x'\|F(x) \, , \, \op_\eps \left( a(x,\xi)\left( 1-\chi\left( \frac{x'}{\delta} \right) \right) \right) \right] \Psi^\eps_t \, , \Psi^\eps_t \right> dt \nonumber \hspace{-7cm} &&\\ 
& \longrightarrow & \left< \, \nabla\left( \|x'\|F(x) \right) \centerdot \partial_\xi \mu(t,x,\xi) 1\!\!1_{\left\lbrace x' \neq 0 \right\rbrace} \, , \Xi(t) \, a(x,\xi) \, \right>_{\R\times\R^{2d}}.
\label{eq:outer_part}
\end{eqnarray}

\subsubsection{The middle part}\label{sec:the_middle_part}
As seen in Remarks \ref{rem:useful_in_middle_part} and \ref{rem:useful_in_middle_part_2}, the calculations in Section \ref{sec:the_inner_part} lead to
{\footnotesize
\begin{eqnarray*}
\frac{i}{\eps} \left[ \|x'\| F(x) \, , \, \op_\eps\left( a(x,\xi) \left(1 - \chi\left( \frac{x'}{\eps R} \right) \right) \chi\left(\frac{x'}{\delta}\right) \right) \right] \hspace{-5cm}  && \\
& = & \op_\eps \left( - \nabla\left(\|x'\|F(x)\right) \centerdot \partial_{\xi} a(x,\xi) \left(1-\chi\left(\frac{x'}{\eps R}\right) \right) \chi\left(\frac{x'}{\delta}\right) \right) + \mathcal{O}\left(\eps\right) + \mathcal{O}\left(\frac{1}{R}\right).
\end{eqnarray*}}
Observing that
\begin{align*}
\nabla (\|x'\|F(x)) \centerdot \partial_\xi a(x,\xi) \left(1-\chi\left(\frac{x'}{\eps R}\right)\right) & = F(x) \frac{x'}{\|x'\|} \centerdot \partial_\xi a(x,\xi) \left(1-\chi\left(\frac{x'}{\eps R}\right)\right) \\
& + \|x'\|\nabla F(x) \centerdot \partial_\xi a(x,\xi) \left(1-\chi\left(\frac{x'}{\eps R}\right)\right),
\end{align*}
and that the term in the first line is the two-scale microlocalization of the symbol
$$\hspace{1.7cm} \frac{\eta}{\|\eta\|}F(x) \centerdot \partial_\xi a(x,\xi)  \left(1-\chi\left(\frac{\eta}{R}\right)\right),$$
whereas the one in the second line, of the symbol
$$\hspace{1.7cm} \eps \|\eta\|\nabla F(x) \centerdot \partial_\xi a(x,\xi)  \left(1-\chi\left(\frac{\eta}{R}\right)\right),$$
it is implied by Proposition \ref{prop:two_microlocal_measures} that: 
\begin{eqnarray}
\int_\R \Xi(t) \left< \frac{i}{\eps} \left[ \|x'\|F(x) \, , \, \op_\eps \left( a(x,\xi)\left( 1-\chi\left( \frac{x'}{\eps R} \right) \right) \chi\left( \frac{x'}{\delta} \right) \right) \right] \Psi^\eps_t \, , \Psi^\eps_t \right> dt \nonumber \hspace{-9.5cm} && \\ 
& \longrightarrow & \left< \delta(x') \otimes \int_{\mathcal{S}^{p-1}} F(x) \, \omega \centerdot \partial_{\xi'} \nu_\infty(t,x'',\xi,d\omega) \, , \Xi(t) \, a(x,\xi) \right>_{\R\times\R^{2d}},
\label{eq:middle_part}
\end{eqnarray}
where $\nu_\infty$ is the two-microlocal measure on sphere introduced in Proposition \ref{prop:two_microlocal_measures}, and $\delta$ is the usual counting measure (or Dirac mass).

\subsection{Establishing the equation}\label{sec:establishing_the_equation}
From equations \eqref{eq:d_mu}, \eqref{eq:d_op} and \eqref{eq:commutators_result}, more the results in the last sections, equations \eqref{eq:kinetic_part}, \eqref{eq:smooth_piece}, \eqref{eq:inner_part}, \eqref{eq:outer_part} and \eqref{eq:middle_part}, we obtain the equation
{\footnotesize
\begin{align}\label{eq:liouville_full_inhomogeneous}
\left< \, \partial_t \mu + D(x)\xi \right. & \left. \centerdot \partial_x \mu - \nabla V_S (x) \centerdot \partial_\xi \mu - \nabla \left( \| x' \| F(x) \right) \centerdot \partial_{\xi} \mu 1\!\!1_{\left\lbrace x' \neq 0 \right\rbrace} \, , \, a(t,x,\xi) \, \right>_{\R\times\R^{2d}} \nonumber \\
& = \limsup_{R\longrightarrow\infty} \, {\rm tr} \left< \, \left[ \|y\| F(0,x'') \, , M(t,x'',\xi'') \right] , a^w(t,0,x'',\partial_y,\xi'') \chi\left(\frac{y}{R}\right) \, \right>_{\R\times\R^{2(d-p)}} \nonumber \\
& \hspace{2.3cm} + \left< \delta(x') \otimes \int_{\mathcal{S}^{p-1}} F(x) \, \omega \centerdot \partial_{\xi'} \nu_\infty(t,x'',\xi,d\omega) \, , \, a(t,x,\xi) \right>_{\R\times\R^{2d}}
\end{align}}
for the full Wigner measure $\mu$ tested against a function $a \in C_0^\infty(\R\times\R^{2d})$. We will now work out this expression until we prove Proposition\ref{prop:main_result} ahead.

\begin{lemma}\label{lem:m_is_zero_and_etc}
In the two-microlocal decomposition given in Lemma \ref{lem:two_microlocal_decomposition}, the operator valued measure $M$ is zero and $\nu_\infty(t,x'',\xi,\omega) = \delta(\xi') \otimes \nu(t,x'',\xi'',\omega)$ for some measure $\nu$ on $\R\times\R^{2(d-p)}\times \mathcal{S}^{p-1}$. Consequently, equation \eqref{eq:liouville_full_inhomogeneous} reads
{\footnotesize
\begin{eqnarray}\label{eq:liouville_full_tested}
\left< \delta(x')\otimes \delta(\xi') \otimes \nu(t,x'',\xi'',\omega) \, , \left( \partial_t + D_f(x)\xi'' \centerdot \partial_{x''} - \nabla V_S(x) \centerdot \partial_\xi - F(x) \, \omega \centerdot \partial_{\xi'} \right) a \right>_{\R\times\R^{2d}\times\mathcal{S}^{p-1}} \nonumber \\ 
+ \left< \mu(t,x,\xi) 1\!\!1_{\left\lbrace x' \neq 0 \right\rbrace} \, , \left( \partial_t + D(x)\xi \centerdot \partial_x - \nabla V(x) \centerdot \partial_\xi \right) a \right>_{\R\times\R^{2d}} \nonumber \\
= \quad 0
\end{eqnarray}}
for all test functions $a \in C_0^\infty(\R\times\R^{2d})$.
\end{lemma}

\begin{remark}
The matrix $D_f$ above was defined in Remark \ref{rem:convenience} in analogy to $D_g$, also defined there. In that remark we also depicted some characteristics of this matrices that help understand the calculations ahead.
\end{remark}

\begin{proof}
To begin with, re-write equation \eqref{eq:liouville_full_inhomogeneous} as
\begin{eqnarray*}
\left< \mu(t,x,\xi) 1\!\!1_{\left\lbrace x' = 0 \right\rbrace} \, , \, \left( \partial_t + D(x)\xi \centerdot \partial_x - \nabla V_S (x) \centerdot \partial_\xi \right) a(t,x,\xi) \, \right> \hspace{3cm} \nonumber \\
+ \left< \mu(t,x,\xi) 1\!\!1_{\left\lbrace x' \neq 0 \right\rbrace} \, , \, \left( \partial_t + D(x)\xi \centerdot \partial_x - \nabla V (x) \centerdot \partial_\xi \right) a(t,x,\xi) \, \right> \hspace{2cm} \nonumber \\
\hspace{1cm} = - \limsup_{R \longrightarrow \infty} \, {\rm tr} \left< \, \left[ \|y\| F(0,x'') \, , M(t,x'',\xi'') \right] , a^w(t,0,x'',\partial_y,\xi'') \chi\left(\frac{y}{R}\right) \, \right> \nonumber \\
\quad\quad + \left< \delta(x') \otimes \int_{\mathcal{S}^{p-1}} \omega \, \nu_\infty(t,x'',\xi,d\omega) \, , \, F(x) \, \partial_{\xi}a(t,x,\xi) \right>.
\end{eqnarray*}
Now, recall that the term in the trace obeys to estimate \eqref{eq:estimate_for_M} given in Lemma \ref{lem:estimate_for_M}; besides, since $\mu$ is a measure in $\mathcal{D'}(\R\times\R^{2d})$, we have 
\begin{equation}\label{eq:estimation_for_mu}
\left< \, \mu(t,x,\xi) \, , a(t,x,\xi) \, \right> \leqslant \max_{(t,x,\xi) \in \R\times\R^{2d}}\left| a(t,x,\xi) \right| \mu\left({\rm supp}(a)\right),
\end{equation}
where $\mu\left({\rm supp}(a)\right) < \infty$ since $\mu$ is finite and ${\rm supp}(a)$, the support of $a$, is compact. Obviously the very same estimate (with $\mu\left({\rm supp}(a)\right)$ in the right-hand side!) is valid for $\mu(t,x,\xi) 1\!\!1_{\left\lbrace x' \neq 0 \right\rbrace}$ and for $\delta(x')\otimes \int_{\mathcal{S}^{p-1}}\omega \, \nu_\infty(t,x'',\xi,d\omega)$. So, for test functions of the form $a_\delta(t,x,\xi) = \delta \, \theta\left(t\right) a_1 \left(\frac{x'}{\delta}\right) a_2 \left(x''\right) b \left(\xi\right)$, equation \eqref{eq:liouville_full_inhomogeneous} gives (recalling also Remark \ref{rem:convenience}):
{\footnotesize
\begin{eqnarray*}
\left< \mu(t,x,\xi) 1\!\!1_{\left\lbrace x' = \,0 \right\rbrace} \, , D_g(x)\xi' \centerdot \partial_{x'} a_1\left(\frac{x'}{\delta}\right) \theta\left(t\right) a_2 \left(x''\right) b \left(\xi\right) \right> \hspace{5cm} \\  
+ \left< \mu(t,x,\xi) 1\!\!1_{\left\lbrace x' \neq 0 \right\rbrace} \, , \left(D(x)\xi\right)' \centerdot \partial_{x'} a_1\left(\frac{x'}{\delta}\right) \theta\left(t\right) a_2 \left(x''\right) b \left(\xi\right) \right> 
+ \mathcal{O}(\delta) = 0,
\end{eqnarray*}}
that at the limit where $\delta \longrightarrow 0$ results in
$$\partial_{x'} a_1 (0) \centerdot \left< \mu(t,x,\xi) 1\!\!1_{\left\lbrace x' =\, 0 \right\rbrace} \, , D_g(0,x'')\xi' \, \theta(t) \, a_2(x'') \, b(\xi) \right> = 0,$$
which means that $\mu(t,x,\xi) 1\!\!1_{\left\lbrace x' =\, 0 \right\rbrace}$ is supported on $\R\times\left\lbrace \xi' = 0 \right\rbrace\times\R^{2d-p}_{x,\xi''}$, and, by positivity, $\nu_{\infty}$ is necessarily supported on $\R\times\left\lbrace \xi' = 0 \right\rbrace\times\R^{2(d-p)}_{x'',\xi''}\times\mathcal{S}^{p-1}$ and the measure $\mathfrak{m}$ introduced in Remark \ref{rem:definition_m} in $\R\times \left\lbrace \xi' = 0 \right\rbrace\times\R^{2d-p}_{x'',\xi'',\rho}$ .

Regarding $\nu_\infty$, any measure supported on $\xi' = 0$ can only be a Dirac mass thereon\footnote{More generally, a distribution supported on such a set can be developed as $\sum_{n \in \mathbb{N}_0} c_n \delta^{(n)}(\xi')$, where $c_n \in \mathbb{C}$ and $\delta^{(n)}$ is the $n$-th distribution derivative of the Dirac $\delta$. As our distribution must be a positive measure as well, the only allowed term in this development is the one with $n = 0$.}, whence there must be $\nu$ as stated in the Lemma. However, for $\mathfrak{m}$ we already knew that it was absolutely continuous with respect to $d\xi'$ (Lemma \ref{lem:m_absolutely_continuous}), so being zero almost everywhere in $\xi'$ is, in other words, saying that $\mathfrak{m} = 0$. From Lemma \ref{lem:m_0_gives_m_0}, this implies that $M=0$ and, finally, that so is the term in the trace in equation \eqref{eq:liouville_full_inhomogeneous}. 

To conclude, just re-write \eqref{eq:liouville_full_inhomogeneous} attaching all the information we have just got and verify that it simplifies to \eqref{eq:liouville_full_tested}.
\end{proof}

\begin{remark}\label{rem:measure_over_the_singularity}
As a scholium of the last proof, one has that $\mu$ is not supported on the region of the $\xi'$ axis away from the origin, $\left\{ (0,\xi') \in \R^{2p} : \xi' \neq 0 \right\}$. This implies that $\mu(t,x,\xi) 1\!\!1_{\left\lbrace x' \neq 0 \right\rbrace} = \mu(t,x,\xi) 1\!\!1_{\left\lbrace (x',\xi') \neq (0,0) \right\rbrace}$, which is a result that \cite{FGL} had already obtained with a similar argument.
\end{remark}

\begin{lemma}\label{lem:technical}
One has the identity $\nabla_{x'} V(x) \, 1\!\!1_{\left\lbrace x' \neq 0 \right\rbrace} 1\!\!1_{\left\lbrace \xi' = 0 \right\rbrace} \mu(t,x,\xi) = 0$.
\end{lemma}

\begin{proof}
Recall estimate \eqref{eq:estimation_for_mu}, which holds for $\mu(t,x,\xi) 1\!\!1_{\left\lbrace x' \neq 0 \right\rbrace}$ and for $\delta(x')\otimes \nu_\infty(t,x'',\xi,\omega)$ as well. Thus, for test functions of the form $a_\delta(t,x,\xi) = \delta \, \theta\left(t\right) \left(x'\right)^2 a \left(x\right) b_1 \left(\frac{\xi'}{\delta}\right) b_2 \left(\xi\right)$ and proceeding in the same manner as in the proof of Lemma \ref{lem:m_is_zero_and_etc}, equation \eqref{eq:liouville_full_tested} gives, in the limit where $\delta \longrightarrow 0$,
$$\partial_{\xi'} b_1(0) \centerdot \left< \, \nabla_{x'}V(x) \, 1\!\!1_{\left\lbrace x' \neq 0 \right\rbrace} 1\!\!1_{\left\lbrace \xi' = 0 \right\rbrace} \mu(t,x,\xi) \, , (x')^2 \, \theta(t) \, a(x) \, b_2(\xi) \, \right> = 0,$$
which means that the distribution $\nabla_{x'}V(x) \, 1\!\!1_{\left\lbrace x' \neq 0 \right\rbrace} 1\!\!1_{\left\lbrace \xi' = 0 \right\rbrace} \mu(t,x,\xi)$ is supported on $\R\times\left\lbrace x' = 0 \right\rbrace\times\R^{2d-p}_{x'',\xi}$. But of course this carries that it is null.
\end{proof}

\begin{lemma}\label{lem:symmetry_condition}
The measure $\nu$ introduced in Lemma \ref{lem:m_is_zero_and_etc} obeys to the following identity in the sense of the distributions on $\R\times\R^{2(d-p)}_{x'',\xi''}$:
$$\int_{\mathcal{S}^{p-1}} \left( \nabla_{x'}V_S(0,x'') + F(0,x'') \, \omega \right) \nu(t,x'',\xi'',d\omega) = 0,$$
where $V_S$ and $F$ are as in \eqref{eq:conical_potential}.
\end{lemma}

\begin{proof}
For test functions of the form $a_\delta(t,x,\xi) = \delta \, \theta\left(t\right) a \left(x\right) b_1 \left(\frac{\xi'}{\delta}\right) b_2 \left(\xi''\right)$, using estimate \eqref{eq:estimation_for_mu} and Lemma \ref{lem:technical}, the present lemma follows from \eqref{eq:liouville_full_tested} at the limit $\delta \longrightarrow 0$.
\end{proof}

To finish establishing a Liouville equation for $\mu$, let us put all Lemmata \ref{lem:m_is_zero_and_etc}, \ref{lem:technical}, \ref{lem:symmetry_condition} and Remark \ref{rem:convenience} together and write equation \eqref{eq:liouville_full_tested} in a distributional and clearer way: 
{\small
\begin{eqnarray*}
\delta(x')\otimes\delta(\xi')\otimes \int_{\mathcal{S}^{p-1}} \left(\partial_t + D_f(x)\xi'' \centerdot \partial_{x''} - \nabla_{x''} V_S(0,x'') \centerdot \partial_{\xi''} \right) \nu (t,x'',\xi'',d\omega) && \\ + \left(\partial_t + D(x)\xi \centerdot \partial_{x} - \nabla V(x) \centerdot \partial_\xi \right) \mu(t,x,\xi) 1\!\!1_{\left\lbrace x' \neq 0 \right\rbrace} & = & 0,
\end{eqnarray*}} 
or, more explicitly, in view of Remark \ref{rem:measure_over_the_singularity}:
\begin{proposition}\label{prop:main_result}
Let be $V(x) = V_S(x) + \| x' \| F(x)$ a conical potential with $x' \in \R^p$, $1 \leqslant p \leqslant d$, $x=(x',x'')$. Let be $\Psi^{\eps}$ the solution to the Schrödinger equation \eqref{eq:schrodinger} with potential $V$. Then, the Wigner measure associated to the concentration of $\Psi^\eps$ in the limit $\eps \longrightarrow 0$ can be decomposed as
$$\mu(t,x,\xi) = 1\!\!1_{\left\lbrace (x',\xi') \neq (0,0) \right\rbrace} \mu(t,x,\xi) + \delta(x')\otimes \delta(\xi') \otimes \int_{\mathcal{S}^{p-1}} \nu(t,x'',\xi'',d\omega),$$
where the measure $\nu$ satisfies the asymmetry condition
$$\int_{\mathcal{S}^{p-1}} \left( \nabla_{x'} V_S(0,x'') + F(0,x'') \, \omega \right) \nu(t,x'',\xi'',d\omega) = 0.$$
Besides, it obeys to the following distributional equation in $\mathcal{D}'(\R\times\R^{2d})$:
\begin{eqnarray*}
\left(\partial_t + (D(x) \xi)'' \centerdot \partial_{x''} - \nabla_{x''} V_S(0,x'') \centerdot \partial_{\xi''} \right) \left( \mu(t,x,\xi) 1\!\!1_{\left\lbrace (x',\xi') = (0,0) \right\rbrace} \right) && \\ + \left(\partial_t + D(x) \xi \centerdot \partial_{x} - \nabla V(x) \centerdot \partial_\xi \right) \left( \mu(t,x,\xi) 1\!\!1_{\left\lbrace (x',\xi') \neq (0,0) \right\rbrace} \right) & = & 0,
\end{eqnarray*}
where $D(x) = \nabla \phi\left(\phi^{-1}(x)\right) ^t \nabla \phi\left(\phi^{-1}(x)\right)$.
\end{proposition}

\subsection{Continuity in $t$ and absolute continuity with respect to $dt$}{\label{sec:absolute_continuity}}
From equation \eqref{eq:mu_abs_continuous}, it is obvious that $\mu$ is absolutely continuous with respect to the Lebesgue measure $dt$, which implies the existence of a function $\R \ni t \longmapsto \mu_t$ taking values in the set of positive measures on $\R^{2d}$ such that $\mu(t,x,\xi) = \mu_t(x,\xi)dt$. The same holds for the two-microlocal measure $\nu$, since it is positive and absolutely continuous with respect to $\mu$, carrying the existence of a function $t \longmapsto \nu_t$ such that $\nu(t,x'',\xi'',\omega) = \nu_t(x'',\xi'',\omega)dt$.

Nevertheless, the same is not true for continuity. In fact, in \cite{FGL} it was shown that $t \longmapsto \mu_t$ is continuous inside a compact $[-T,T]$ by verifying that, for $a \in C_0^\infty(\R^{2d})$, the commutator $\frac{i}{\eps}\left[ \hat{H}^\eps , \op_\eps(a) \right]$ is uniformly bounded with respect to $\eps$ (what we have indirectly obtained during the calculations in this section), so, from \eqref{eq:d_op}, one sees that the family $t \longmapsto \left< \op_{\eps_k} (a) \Psi^{\eps_k} , \Psi^{\eps_k} \right>$ is equicontinuous in $[-T,T]$ in addition to being equibounded, which implies the continuity of $[-T,T] \ni t \longmapsto \left< \mu_t , a \right>_{\R^{2d}}$ by the Ascoli-Arzelà theorem. 

This is not true for $\nu_t$ in general, as the examples in Theorem \ref{th:main_result_3} show. In both cases there, we have $\int_{\mathcal{S}^0}\nu_t(d\omega) = 0$ for any $t \neq 0$ and $\int_{\mathcal{S}^{0}}\nu_0(d\omega) = 1$, so $\nu_t \neq 0$ if and only if $t = 0$, which is a lack of continuity for $\nu_t$ in spite of $\mu_t$. 

On the other hand, one could use an argument of continuity for $\mu_t$ to see, in the one-dimensional case $V(x) = \left\| x \right\|$, that a family $(\Psi^\eps_t)_{\eps >0}$ whose initial data concentrate to $\mu_0(x,\xi) = \delta(x) \otimes \delta(\xi)$ will remain concentrating to this same point, $\mu_t = \delta(x) \otimes \delta(\xi)$. In this case, condition \eqref{eq:asymmetry_formula} allows a complete description of $\nu$: 
$$\nu(t,\omega) = \frac{1}{2}\left(\delta(\omega -1) + \delta(\omega + 1)\right) \otimes dt.$$

\section{An application of the asymmetry condition and examples}
\subsection{The classical flow and the concentration of $\nu$}\label{sec:asymmetry_condition_in_use}
In last section we presented a trivial application of the asymmetry condition \eqref{eq:asymmetry_formula}. In this section we will apply it to obtain Theorem \ref{th:main_result_4}. Moreover, we will prove Theorems \ref{th:accessory_result_1} and \ref{th:accessory_result_2}. In next section we will give examples of applications of these results.  

So, to start with:
\begin{proof}[Proof of Theorem \ref{th:main_result_4}]
Since $V_S$ and $F \, ^t\nabla g$ are continuous functions all over $\R^d$, there exists a neighbourhood $\Gamma$ of $\sigma$ such that the inequality
$$\left\| F(x) ^t\nabla g(x) \right\|_{\mathcal{L}(N_x\Lambda)} < \left\| \partial_\rho V_S(x) \right\|_{N_x\Lambda}$$
holds strictly for every $x \in \Gamma$. Let be $a \in C_0^\infty(\R\times T^*\Gamma)$, and let us test expression \eqref{eq:asymmetry_formula} against this function:
$$\left< \, \nu(t,\sigma,\zeta,\omega) \, , \, a(t,\sigma,\zeta) \left( \partial_\rho V_S(\sigma) + F(\sigma) ^t\nabla g(\sigma) \omega \right) \, \right>_{\R\times ES\Lambda} = 0.$$
Further, realize that the testing term is always non-zero within the support of $a$:
{\footnotesize
\begin{eqnarray*}
\left\| a(t,\sigma,\zeta) \left( \partial_\rho V_S(\sigma) + F(\sigma) ^t\nabla g(\sigma) \omega \right) \right\|_{N_\sigma \Lambda} & \geqslant & \left| a(t,\sigma,\zeta) \right| \left( \left\| \partial_\rho V_S(\sigma) \right\|_{N_\sigma \Lambda} - \left\| F(\sigma) ^t\nabla g(\sigma) \omega \right\|_{N_\sigma\Lambda} \right) \\
& > & 0;
\end{eqnarray*}}
as a consequence, since $\nu$ is always positive, one must have $\nu = 0$ over the support of $a$, more precisely, over $ES\Gamma$.
\end{proof}

Now, in order to study the classical flow in more details, remark that in the transformed coordinates introduced in Section \ref{sec:reducing}, the equation of motion for the component of $x$ in $\Lambda$, $x'$, reads:
\begin{equation*}
\left\lbrace
\begin{array}{l}
\dot{x}'(t) = D_g(x(t)) \xi'(t) \\
\dot{\xi}'(t) = -\partial_{x'} V_S(x(t)) - \| x'(t) \| \partial_{x'}F(x(t)) - F(x(t)) \frac{x'(t)}{\| x'(t) \|}.
\end{array}\right.
\end{equation*} 

Let us suppose that $(x(t),\xi(t))$ is a trajectory that reaches the phase space singular set $\Omega$ on the point $\sigma \in \Lambda$ at a time $t = 0$, so we have $x'(0) = 0$ and $\xi'(0) = 0$. Calculate:
$$x'(t) = t \int_0^1 \dot{x}'(st) ds \qquad {\rm and} \qquad \dot{x}'(st) = \dot{x}'(0) + t\int_0^s \ddot{x}'(rt)dr,$$
where $\dot{x}'(0) = D_g(x(0)) \, \xi'(0) = 0$ and 
$$\ddot{x}'(t) = \left( \nabla_x D_g(x(t)) \, \xi(t) \right) \xi'(t) + D_g(x(t)) \dot{\xi'}(t).$$
Then, for $t \neq 0$, we can define a $\theta(t) = \frac{2}{t^2}x'(t)$ that reads:
\begin{eqnarray}\label{eq:theta_definition}
\theta(t) = 2\int_0^1 \int_0^s \left( \left(\nabla_x D_g(x(rt)) \, \xi(rt) \right) \xi'(rt) + D_g(x(rt))\right. \nonumber  \hspace{3.5cm} \\
\left. \left( -\partial_{x'} V_S(x(rt)) - \| x'(rt) \| \partial_{x'}F(x(rt)) - F(x(rt)) \frac{x'(rt)}{\| x'(rt) \|} \right) \right) drds,
\end{eqnarray}
where $D_g$ was defined in Remark \ref{rem:convenience}; for the reader's convenience, recall that
$$D_g(x) = \nabla g \left( \phi^{-1}(x) \right)^t \nabla g \left( \phi^{-1}(x) \right)$$
is an invertible matrix for $x = (0,x'')$. 

Also, let be $(t_n)_{n \in \mathbb{N}}$ such that $t_n \longrightarrow 0$. Necessarily the sequence $\frac{x'(t_n)}{\| x'(t_n) \|}$ will have convergent subsequences. If all possible sequences $\frac{x'(t_n)}{\| x'(t_n) \|}$ with $t_n > 0$ tending to $0$ converge to the same limit $\frac{\theta_0^+}{\|\theta_0^+\|}$, we call it the \emph{positive lateral limit} of the function $\frac{x'(t)}{\| x'(t) \|}$ and denote $\frac{\theta_0^+}{\|\theta_0^+\|} = \lim_{t \rightarrow 0^+} \frac{x'(t)}{\| x'(t) \|}$; in the same way, one can talk about \emph{negative lateral limits}. Finally, we define the lateral limits of $\theta(t) = \frac{1}{t^2}x'(t)$ in an analogous manner, when they exist.

\begin{lemma}\label{lem:conditions_on_the_flow}
Fix a vector $\theta_0 \in \R^p$, $\theta_0 \neq 0$. If $\theta_0$ is a lateral limit of $\theta(t)$, then $\frac{\theta_0}{\| \theta_0 \|}$ is a lateral limit of $\frac{x'(t)}{\| x'(t) \|}$, either both positive, or both negative. Conversely, if $\frac{\theta_0}{\| \theta_0 \|}$ is a lateral limit of $\frac{x'(t)}{\| x'(t) \|}$, then there exists $\lambda \geqslant 0$ such that $\lambda \theta_0$ is a lateral limit of $\theta(t)$. 

Besides, $\theta_0$ is a solution to the equation
\begin{equation}\label{eq:condition_simple}
\lambda D_g^{-1}(x(0))\theta_0 = -\partial_{x'} V_S(x(0)) - F(x(0)) \frac{\theta_0}{\| \theta_0 \|};
\end{equation}
geometrically, this equation reads (remembering that $\partial_\rho$ is the derivative normal to $\Lambda$):
\begin{equation}\label{eq:condition_geometric}
\lambda \rho_0 = -\partial_{\rho} V_S(\sigma) - F(\sigma) ^t\nabla g (\sigma) \frac{\nabla g(\sigma) \rho_0}{\| \nabla g(\sigma) \rho_0 \|},
\end{equation}
where $\rho_0$ is a lateral limit of the $N_\sigma \Lambda$ vector function $\frac{1}{t^2}\nabla g(x(t))^{-1} g(x(t))$. 

Finally, if $\| F(\sigma) ^t\nabla g (\sigma) \|_{\mathcal{L}(N_\sigma\Lambda)} < \| \partial_\rho V_S(\sigma) \|_{N_\sigma\Lambda}$ and $\frac{x'(t)}{\|x'(t)\|}$ converges laterally, then any lateral limit $\lambda \theta_0$ of $\theta(t)$ is non-zero and, therefore, satisfies the above equations with $\lambda \neq 0$.
\end{lemma}

\begin{proof}
Observe that $\frac{x'(t)}{\| x'(t) \|} = \frac{\theta(t)}{\| \theta(t) \|}$ for any $t \neq 0$, so if $\theta_0$ is non-zero and a lateral limit of $\theta(t)$, then necessarily $\frac{\theta_0}{\| \theta_0 \|}$ is a lateral limit of $\frac{x'(t)}{\| x'(t) \|}$. The converse comes from the definition of $\theta(t)$ in \eqref{eq:theta_definition}: since at the limit $t \longrightarrow 0$ one has the full limits $\xi'(t) \longrightarrow 0$, $x'(t) \longrightarrow 0$, $V_S(x(t)) \longrightarrow V_S(\sigma)$ and $F(x(t)) \longrightarrow F(\sigma)$, then whenever $\frac{x'(t)}{\| x'(t) \|}$ converges laterally to some limit $\frac{\theta_0}{\| \theta_0 \|}$, $\theta(t)$ also converges laterally to a well-defined vector $\tilde{\theta}$. If it is non-zero, by the previous identity $\frac{x'(t)}{\| x'(t) \|} = \frac{\theta(t)}{\| \theta(t) \|}$ one must have $\frac{\tilde{\theta}}{\| \tilde{\theta} \|} = \frac{\theta_0}{\| \theta_0 \|}$, so $\tilde{\theta} = \lambda \theta_0$ with $\lambda > 0$; if it is zero, then $\tilde{\theta} = \lambda \theta_0$ for $\lambda = 0$, and we get the lemma's first paragraph either way.   

\medskip

Equation \eqref{eq:condition_simple} comes from taking lateral limits in equation \eqref{eq:theta_definition} when $\frac{\theta_0}{\|\theta_0\|}$ is a lateral limit of $\frac{x'(t)}{\|x'(t)\|}$ and $\lambda \theta_0$ of $\theta(t)$. Multiplying its both sides by $^t \nabla g(\sigma)$, one gets
\begin{equation}\label{eq:condition_middle_of_the_way}
\lambda\nabla g(\sigma)^{-1}\theta_0 = -\partial_{\rho} V_S(\sigma) - F(\sigma) ^t\nabla g (\sigma) \frac{\theta_0}{\| \theta_0 \|}.
\end{equation}

Two things remain before completing the proof: recognizing equation \eqref{eq:condition_geometric} from \eqref{eq:condition_middle_of_the_way} and, if $\frac{x'(t)}{\| x'(t) \|}$ has lateral limits, proving that any lateral limit $\theta(t)$ is non-zero under the additional hypothesis that 
$$\| F(\sigma) ^t\nabla g (\sigma) \|_{\mathcal{L}(N_\sigma\Lambda)} < \| \partial_\rho V_S(\sigma) \|_{N_\sigma\Lambda}.$$

The latter is done by remarking that, if $\| F(\sigma) ^t\nabla g(\sigma) \| < \| \partial_\rho V_S(\sigma) \|$, surely
$$\lambda \nabla g(\sigma)^{-1} \theta_0 \neq 0$$
by the same arguments we saw in the proof of Theorem \ref{th:main_result_4}), so $\lambda \neq 0$.

\medskip 

Finally, recall that $x'$ is the coordinate of the variable in the normal bundle $N_\sigma \Lambda$; calling it $\rho$ in the original coordinates, we have $x' = \nabla g(\sigma) \rho$, so in \eqref{eq:condition_middle_of_the_way} we have just a lateral limit of $\frac{2}{t^2}\rho(t) = \frac{2}{t^2}\nabla g(\sigma)^{-1} x'(t)$: $\theta_0 = \nabla g(\sigma) \rho_0$. 
\end{proof}

As we have seen in the lemma above, for any trajectory arriving on $\Omega$ within a well-defined direction $\theta_0$, \emph{i.e.}, such that $\lim_{t \rightarrow 0^-} \frac{x'(t)}{\| x'(t) \|} = \frac{\theta_0}{\|\theta_0\|}$, if
$$\| F(\sigma) ^t\nabla g(\sigma)\|_{\mathcal{L}(N_\sigma\Lambda)} < \| \partial_\rho V_S(\sigma) \|_{N_\sigma\Lambda},$$
then this direction is submitted to satisfy equation \eqref{eq:condition_simple} with $\lambda \neq 0$, which implies, in particular, that if this equation has no non-zero roots, then by absurd no trajectory at all can reach $\Omega$.

Regarding the inverse affirmation, Lemma \ref{lem:conditions_on_the_flow} does not say whether there are actual trajectories approaching $\Omega$ in all possible directions satisfying \eqref{eq:condition_simple}. Below we will verify that indeed any $\theta_0$ satisfying \eqref{eq:condition_simple} is realized as an approaching direction by some Hamiltonian trajectory.

\begin{lemma}\label{lem:part_proof_2}
If $\theta_0^+, \theta_0^- \neq 0$ satisfy \eqref{eq:condition_simple}, then there exists a unique trajectory $(x(t),\xi(t))$ reaching $\Omega$ at $t = 0$ in a point $\sigma \in \Lambda$ such that $\lim_{t \rightarrow 0^\pm} \frac{x'(t)}{\|x'(t)\|} = \frac{\theta_0^\pm}{\| \theta_0^\pm \|}$.
\end{lemma}

\begin{proof}
First, let us choose $\theta_0^+ = \theta_0^- = \theta_0$. For $x_0$ and $\xi_0$ such that $x_0' = \xi_0' = 0$ and so that $x_0$ is the coordinate of $\sigma$, take $\tau > 0$, $\lambda > 0$ and $0 < \delta < \| \theta_0 \|$ sufficiently small so as the set
{\small
\begin{eqnarray*}
\mathcal{B}^{[-\tau,\tau]} = \left\lbrace (x,\xi,\vartheta) \in \left(C^0 \left([-\tau,\tau]\right)\right)^3 : (x,\xi,\vartheta)(0) = \left(x_0,\xi_0,\lambda \theta_0\right) \; {\rm and} \hspace{3cm} \right. \\
\left. \sup_{t \in [-\tau,\tau]} \left( \| x(t) - x_0 \| + \| \xi(t) - \xi_0 \| + \| \vartheta(t) - \lambda \theta_0 \| \right) \leqslant \delta \right\rbrace
\end{eqnarray*}}
fits in a proper definition of the application $\mathfrak{F} : \mathcal{B}^{[-\tau,\tau]} \longrightarrow \mathcal{B}^{[-\tau,\tau]}$ given by: 
{\small
$$\mathfrak{F}_x (x,\xi,\vartheta)(t) = x_0 + \int_0^t D(x(s)) \, \xi(s)ds,$$}
{\small
$$\mathfrak{F}_\xi (x,\xi,\vartheta)(t) = \xi_0 + \int_0^t \left( -\nabla V_S(x(s)) - \| x'(s) \| \nabla F(x(s)) - F(x(s)) \frac{(\vartheta(s),\mathbf{0}_{\R^{d-p}})}{\| \vartheta(s) \|} \right)ds$$}
and last
{\small
\begin{eqnarray*}
\mathfrak{F}_\vartheta(x,\xi,\vartheta)(t) = 2 \lambda \int_0^1 \int_0^s \left( \left( \nabla_x D_g(x(rt)) \, \xi(rt) \right) \xi'(rt) + D_g(x(rt))\right. \nonumber \hspace{3cm} \\
\left. \left( -\partial_{x'} V_S(x(rt)) - \| x'(rt) \| \partial_{x'}F(x(rt)) - F(x(rt)) \frac{\vartheta(rt)}{\| \vartheta(rt) \|} \right) \right) drds.
\end{eqnarray*}}
Observe that $\mathcal{B}^{[-\tau,\tau]}$ is a set which is complete with respect to its natural supremum norm employed in its definition. Observe further that, by taking $\lambda$ and $\tau$ as small as necessary, $\mathfrak{F}$ becomes a contraction on $\mathcal{B}^{[-\tau,\tau]}$ equipped with its topology, \emph{i.e.}, there exists $0 < K < 1$ such that
{
\begin{eqnarray*}
\sup_{t \in [-\tau,\tau]} \left( \| \mathfrak{F}_x(x,\xi,\vartheta)(t) - x_0 \| + \| \mathfrak{F}_\xi(x,\xi,\vartheta)(t) - \xi_0 \| + \| \mathfrak{F}_{\vartheta}(x,\xi,\vartheta)(t) - \lambda \theta_0 \| \right) \\ \leqslant K \sup_{t \in [-\tau,\tau]} \left( \| x(t) - x_0 \| + \| \xi(t) - \xi_0 \| + \| \vartheta(t) - \lambda \theta_0 \| \right).
\end{eqnarray*}}

Consequently, by Banach's fixed point theorem, there exists a unique triple $\left(x,\xi,\vartheta\right) \in \mathcal{B}^{[-\tau,\tau]}$ such that $\left(x,\xi,\vartheta\right) = \mathfrak{F}\left(x,\xi,\vartheta\right)$; this is equivalent to saying that the system
{\small
\begin{equation*}
\left\lbrace
\begin{array}{l}
\dot{x}(t) = D(x(t))\xi(t) \\
\dot{\xi}(t) = -\nabla V_S(x(t)) - \| x'(t) \| \nabla F(x(t)) - F(x(t)) \frac{(\vartheta(t),\mathbf{0}_{\R^{d-p}})}{\| \vartheta(t) \|}\\
\frac{d}{dt}\left(\frac{t^2}{2} \vartheta(t)\right) = \lambda \dot{x}'(t)
\end{array}\right.
\end{equation*}}
with initial data $\left(x_0,\xi_0,\lambda \theta_0\right)$ admits a unique solution, which must be such that $\frac{\vartheta(t)}{\| \vartheta(t) \|} = \frac{x'(t)}{x'(t)}$ $\forall t \neq 0$. Therefore, $(x(t),\xi(t))$ must be a trajectory of our conical problem with the properties listed in the lemma for $\theta_0 = \theta_0^+ = \theta_0^-$.

For $\theta_0^+ \neq \theta_0^-$, define the sets $\mathcal{B}^{[0,\tau]}$ and $\mathcal{B}^{[-\tau,0]}$ and proceed as above in order to show the existence for $\pm t \in [0,\tau]$ of trajectories $(x_{\pm}(t),\xi_{\pm}(t))$ such that $\lim_{t \longrightarrow 0^\pm} \frac{x'_\pm(t)}{\|x'_\pm(t)\|} = \frac{\theta_0^\pm}{\|\theta_0^\pm\|}$. Besides, as $\lim_{t \longrightarrow 0^\pm} x_{\pm}(t) = x_0$ and $\lim_{t \longrightarrow 0^\pm} \xi_\pm (t) = \xi_0$, we can build a continuous path $(x(t),\xi(t))$ by setting $x(t) = x_\pm(t)$ and $\xi(t) = \xi_\pm(t)$ for $\pm t \geqslant 0$, and this new trajectory will be Hamiltonian and the unique one to meet the properties stated in the lemma.
\end{proof}
Bringing together Lemmata \ref{lem:conditions_on_the_flow} and \ref{lem:part_proof_2}, one proves Theorem \ref{th:accessory_result_1}. In order to obtain Theorem \ref{th:accessory_result_2}, we observe the following facts:

\medskip

\begin{enumerate}
\item The hypothesis $\| F(\sigma) ^t\nabla g(\sigma) \| < \| \partial_\rho V_S(\sigma) \|$ played no role in the demonstration of Lemma \ref{lem:part_proof_2}, nor in the proof of Lemma \ref{lem:conditions_on_the_flow}, where we could have $\theta_0 \neq 0$ and $\tilde{\theta} = 0$. In any case, if some trajectory is to arrive onto $\Omega$ in $\sigma$ with some well-defined direction, then either \eqref{eq:aiai} or \eqref{eq:zero_roots} must be satisfied.
\item If $\theta_0$ is a ``zero root'' of \eqref{eq:condition_simple} (otherwise said, $\tilde{\theta} = \lambda \theta_0$ with $\lambda = 0$), then we will see in Example \ref{ex:equal_no} that it may happen that no trajectory following the direction $\theta_0$ ever touches $\Omega$ in $\sigma$. On the other hand, in Example \ref{ex:equal_yes} we will see a case where there is a trajectory approaching the singularity, but only asymptotically.
\end{enumerate}

\medskip

The proof of Theorem \ref{th:accessory_result_2} will be complete after showing that any such trajectories only reach $\Omega$ after an infinite time. Letting be $(x(t),\xi(t))$ a trajectory that is not on $\Omega$ at $t = 0$,  define
$$\Upsilon = \left\lbrace t \in \R : \forall s \in [0,t], (x(s),\xi(s)) \notin \Omega \right\rbrace$$
and $\Gamma = \left\lbrace (x(t),\xi(t)) : t \in \Upsilon \right\rbrace$. 

\begin{lemma}
If $\theta_0 \neq 0$ obeys to \eqref{eq:condition_simple} with $\lambda = 0$ and we suppose that there is a trajectory $(x(t),\xi(t))$ such that $(x(0),\xi(0)) \notin \Omega$, that for any $\eps > 0$ there is $t_\eps > 0$ for which $\| (x'(t_\eps),\xi'(t_\eps)) \| < \eps$ and such that $\lim_{\eps \rightarrow 0} \frac{1}{\| x'(t_\eps) \|}x'(t_\eps) = \frac{\theta_0}{\|\theta_0\|}$, then $\sup \Upsilon = \infty$.
\end{lemma}

\begin{proof}
From the hypotheses $\theta_0 \neq 0$ and $\lambda = 0$ in \eqref{eq:condition_simple}, it follows that $\lim_{\eps \rightarrow 0} \partial_{x'} V(x(t_\eps)) = 0$; as a consequence, it becomes possible to find a smooth extension $\tilde{V}$ of $V$ outside the closure of $\Gamma$ such that $\partial_{x'} \tilde{V}(\sigma) = 0$ and that $\tilde{V}(x(t)) = V(x(t))$ for any $t \in \Upsilon$. 

Then $(x'(t),\xi'(t))$ is at the same time a Hamiltonian trajectory of a conical potential and of a standard problem with smooth potential, for which case it is widely known that no trajectory can arrive at an extremum of $\tilde{V}$ with null speed $\xi'$ within a finite time (since this would break the unicity of the constant solution $(x'(t),\xi'(t)) = (0,0)$).
\end{proof}
\clearpage
\subsection{Examples of classification for the semiclassical transport}\label{sec:examples}
In this section we will give examples of how Theorems \ref{th:main_result_4}, \ref{th:accessory_result_1} and \ref{th:accessory_result_2} can be used in order to classify the trajectories that arrive on a conical singularity and, sometimes, to completely describe the transport phenomenon to which the semiclassical measures are submitted. In particular, Examples \ref{ex:equal_no} and \ref{ex:equal_yes} are part of the reasoning that led to obtaining the second assertive in Theorem \ref{th:accessory_result_2}.

To begin with, let us consider $d = 1$ and $g(x) = x$, so $\nabla g(x) = 1$.

\begin{example}[$| \nabla V_S(\sigma) | < | F(\sigma) |$ with no roots]\label{ex:1} Take $V_S(x) = \frac{1}{2}x$ and $F(x) = 1$.

\begin{figure}[htpb!]\center\footnotesize
\begin{multicols}{2}
\includegraphics[width=0.45\textwidth]{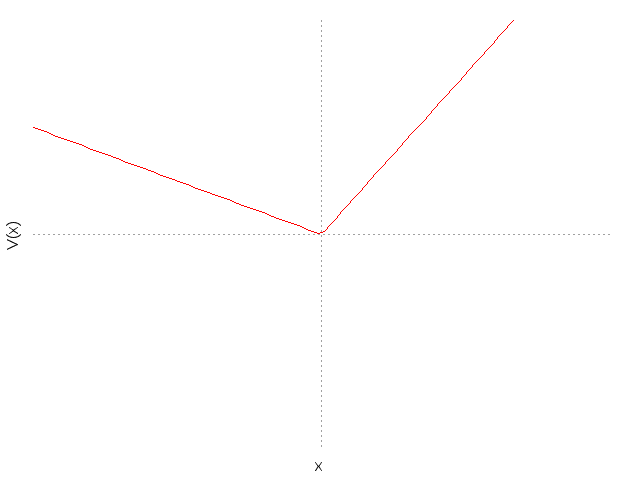} \text{$V(x) = \frac{1}{2}x + |x|$} \\
\includegraphics[width=0.45\textwidth]{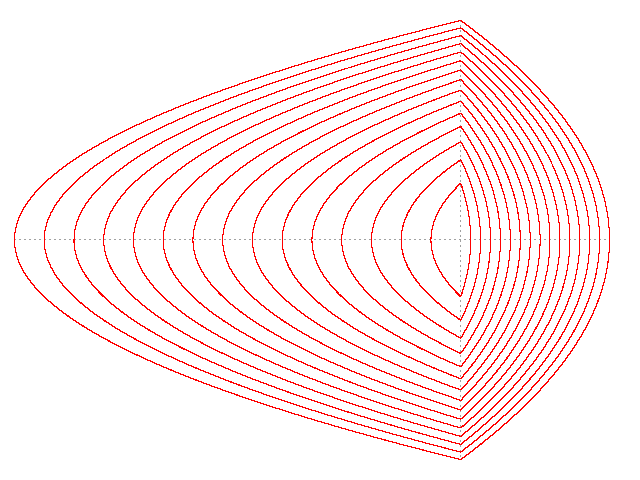} \text{$\Phi$, example \ref{ex:1}} \\
\end{multicols}
\end{figure}

Then \eqref{eq:aiai} admits no non-zero solutions, which is consistent with the fact that the classical flow $\Phi$ presents no trajectories hitting the singularity.

In this case, the asymmetry condition \eqref{eq:asymmetry_formula} gives that
$$\nu(t,\omega) = p \left( \frac{1}{4} \delta(\omega -1) + \frac{3}{4} \delta(\omega + 1) \right) \otimes dt,$$
with $0 \leqslant p \leqslant 1$ the total mass over the singularity and does not depend on $t$. The semiclassical transport here is thus completely solvable starting from some initial measure.
\end{example} 

\begin{example}[$| \nabla V_S(\sigma) | < | F(\sigma) |$ with non-zero roots]\label{ex:2} Consider $V_S(x) = \frac{1}{2}x$ and take $F(x) = -1$.

\begin{figure}[htpb!]\center\footnotesize
\begin{multicols}{2}
\includegraphics[width=0.45\textwidth]{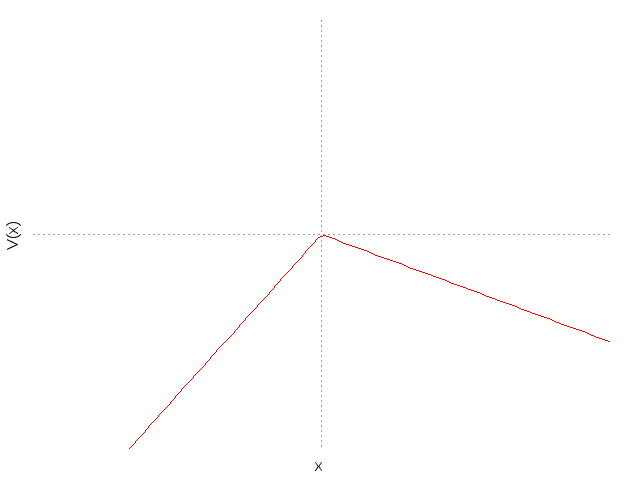} \text{$V(x) = \frac{1}{2}x - |x|$} \\
\includegraphics[width=0.45\textwidth]{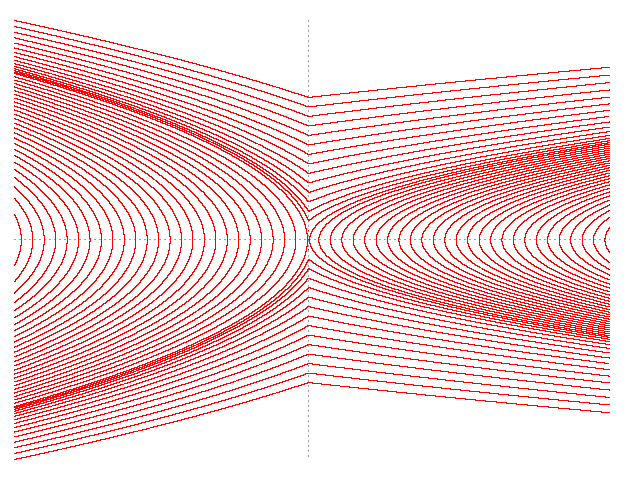} \text{$\Phi$, example \ref{ex:2}} \\
\end{multicols}
\end{figure}

Then \eqref{eq:aiai} admits two solutions: $\rho_0 = \frac{1}{2}$ and $\rho_0 = -\frac{3}{2}$, which is consistent with the fact that the classical flow $\Phi$ does present trajectories hitting the singularity from both directions $x > 0$ and $x < 0$.

In this case, we will have
$$\nu(t,\omega) = p(t) \left( \frac{3}{4} \delta(\omega -1) + \frac{1}{4} \delta(\omega + 1) \right) \otimes dt,$$
with $0 \leqslant p(t) \leqslant 1$ the total mass over the singularity, which may vary with $t$.
\end{example} 

\begin{example}[$| \nabla V_S(\sigma) | = | F(\sigma) |$, with zero-roots without trajectories]\label{ex:equal_no} Take $V_S(x) = x$ and $F(x) = 1$. 

\begin{figure}[htpb!]\center\footnotesize
\begin{multicols}{2}
\includegraphics[width=0.45\textwidth]{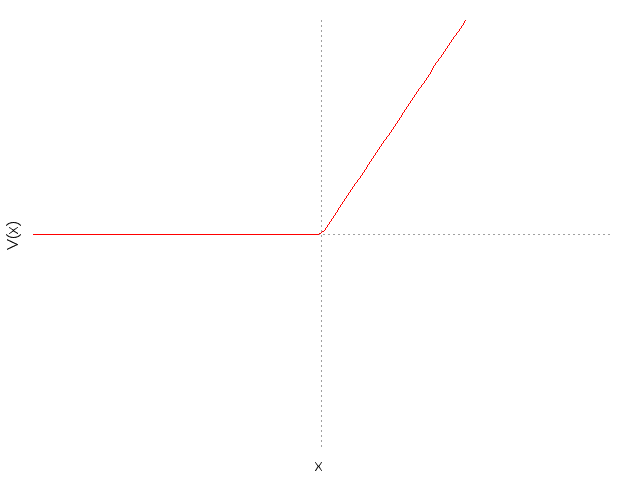} \text{$V(x) = x + |x|$} \\
\includegraphics[width=0.45\textwidth]{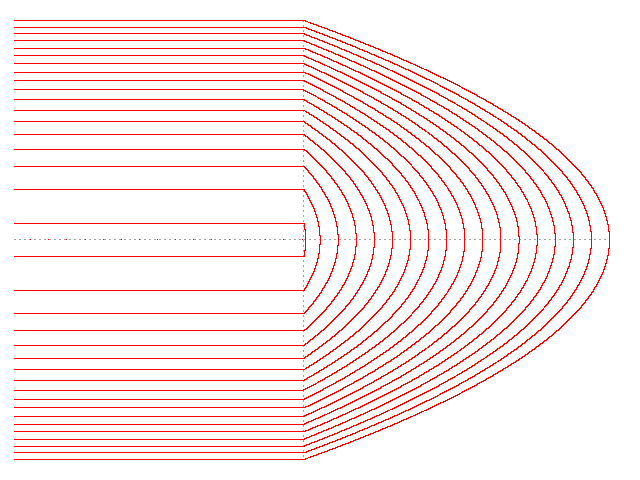} \text{$\Phi$, example \ref{ex:equal_no}} \\
\end{multicols}
\end{figure}

Then equation \eqref{eq:aiai} has no non-zero roots, but equation \eqref{eq:zero_roots} admits any solution $\rho_0 < 0$. In this case, there are no trajectories hitting the singularity from $x< 0$. The semiclassical measure in sphere will be
$$\nu(t,\omega) = p \, \delta(\omega + 1) \otimes dt,$$
where the total mass over the singularity $0 \leqslant p \leqslant 1$ is constant. Again, this is a completely solvable case.
\end{example}

\begin{example}[$| \nabla V_S(\sigma) | = | F(\sigma) |$ with non-zero and zero roots with trajectories]\label{ex:equal_yes} Pick up $V_S(x) = x$ and $F(x) = -(1+x)$.

\begin{figure}[htpb!]\center\footnotesize
\begin{multicols}{2}
\includegraphics[width=0.45\textwidth]{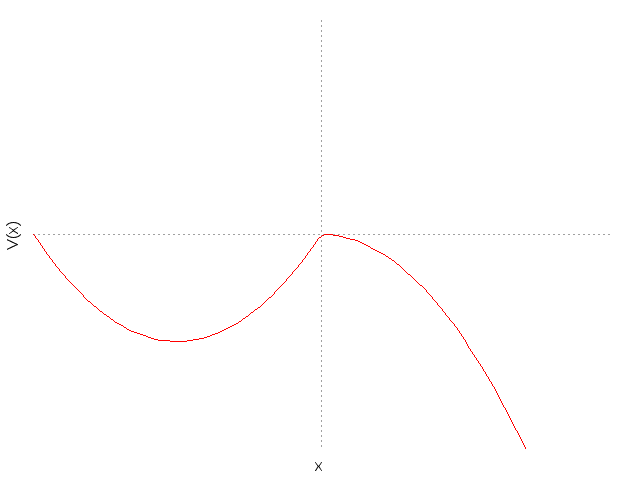} \text{$V(x) = x - (1+x)|x|$} \\
\includegraphics[width=0.45\textwidth]{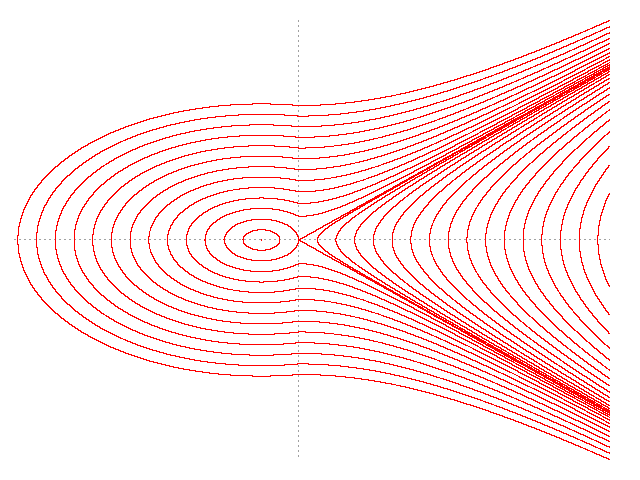} \text{$\Phi$, example \ref{ex:equal_yes}} \\
\end{multicols}
\end{figure}

Then equation \eqref{eq:aiai} admits one solution: $\rho_0 = -2$, and equation \eqref{eq:zero_roots} also admits solutions: any $\rho_0 > 0$. This is consistent with the fact that the classical flow $\Phi$ presents a trajectory hitting the singularity from the direction $x < 0$, and in this case also from the direction $x > 0$. However, the trajectory from the positive side takes an infinitely long time to get close to the singularity.

One has:
$$\nu(t,\omega) = p(t) \, \delta(\omega + 1) \otimes dt,$$
the total mass on the singularity possibly changing with time.
\end{example}

\begin{example}[$| \nabla V_S(\sigma) | > | F(\sigma) |$ with a unique trajectory]\label{ex:5} Take $V_S(x) = 2x$ and choose $F(x) = 1$.

\begin{figure}[htpb!]\center\footnotesize
\begin{multicols}{2}
\includegraphics[width=0.45\textwidth]{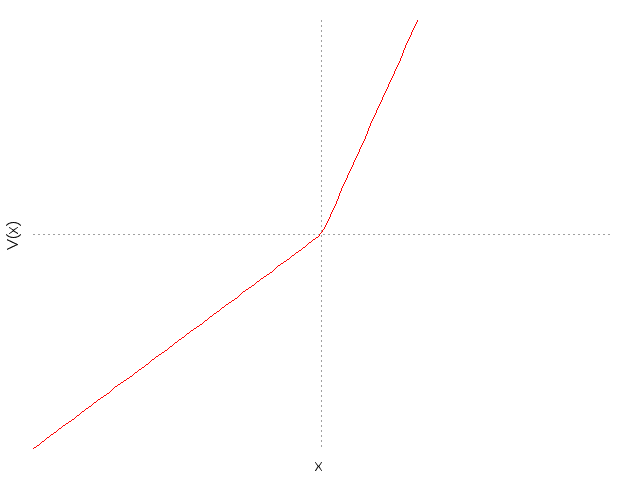} \text{$V(x) = 2x + |x|$} \\
\includegraphics[width=0.45\textwidth]{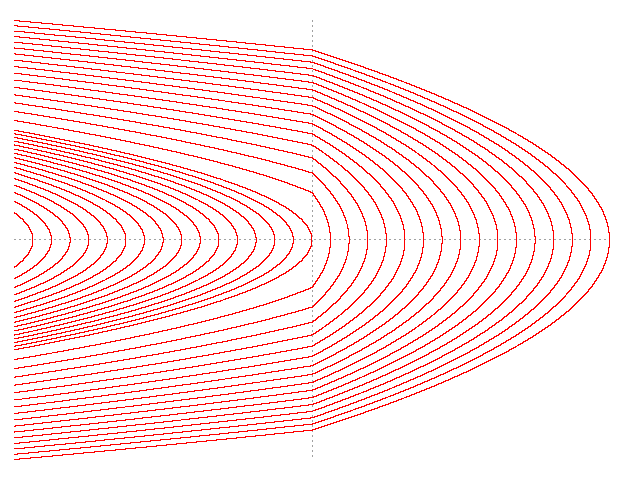} \text{$\Phi$, example \ref{ex:5}} \\
\end{multicols}
\end{figure}

Then equation\eqref{eq:aiai} admits a unique solution: $\rho_0 = -3$, which is consistent with the fact that the classical flow $\Phi$ only presents a trajectory hitting the singularity from the direction $x < 0$.

In this case, since $| \nabla V_S(0) | > | F(0) |$, we will have $\nu = 0$, so the semiclassical measure will necessarily follow the exterior flow with the singularity being part of the parabola passing through the origin from $x < 0$. The problem is hence completely solvable, even though there are trajectories leading to the singular set.
\end{example}

Now let be $p = d =3$ and denote $x=(x_1,x_2,x_3)$.

\begin{example}[$\| \nabla V_S(\sigma) \| > \| F(\sigma) ^t \nabla g(\sigma) \|$ with no trajectories]
Choose an exterior potential $V_S(x) = -2x_1$, $F(x) = -1$ and $g(x) = \left( \frac{1}{2}x_1,x_2,x_3 \right)$. Then 
$$\| \nabla V_S (0) \| = 2 > 1 = \| F(0) ^t \nabla g(0) \|$$
and no $\rho_0$ satisfies equation \eqref{eq:aiai}. As a conclusion, no trajectory in this case can hit the singularity at $x = 0$.
\end{example}

\begin{example}[$\| \nabla V_S(\sigma) \| > \| F(\sigma) ^t\nabla g(\sigma) \| $ with many trajectories]\label{ex:last}
Last, we will take the same $V_S(x) = -2x_1$, $F(x) = -1$, but $g(x) = \left( \frac{1}{3}x_1,x_2,x_3 \right)$. Then 
$$\| \nabla V_S (0) \| = 2 > 1 = \| F(0) ^t \nabla g(0) \|,$$
but now equation \eqref{eq:aiai} admits any solution $\rho_0 = \left( \frac{9}{4}, \rho_2,\rho_3 \right)$ with $\rho_2^2 + \rho_3^2 = \frac{7}{16}$. This is a case where the exterior force polarizes the flow in its direction, but leaves it free to spin around a circle of radius $\frac{\sqrt{7}}{4}$ in the orthogonal plane.
\end{example}

\section*{\small Acknowledgements}
\small
We thank Dr. Clotilde Fermanian for her helpful support and fruitful advises all over the writing of this paper. We also thank Dr. Fabricio Macià for the help in Sections \ref{sec:the_outer_part} and \ref{sec:establishing_the_equation}.

\end{document}